\documentclass[11pt,reqno]{amsart}
\usepackage{amsmath, amsthm, amssymb, stmaryrd}
\usepackage[colorlinks]{hyperref}
\usepackage{enumerate}
\usepackage{url}
\usepackage{tikz}

\let\OLDthebibliography\thebibliography
\renewcommand\thebibliography[1]{
  \OLDthebibliography{#1}
  \setlength{\parskip}{0pt}
  \setlength{\itemsep}{0pt plus 0.3ex}
}

\usepackage{mathtools}

\usepackage{multirow, array}
\usepackage{placeins}
\usepackage{xcolor}

\usepackage{caption}
\advance \topmargin by -\headheight
\advance \topmargin by -\headsep

\newtheorem{thm}{Theorem}[section]
\newtheorem{lemma}[thm]{Lemma}
\newtheorem{prop}[thm]{Proposition}
\newtheorem{cor}[thm]{Corollary}

\theoremstyle{definition}
\newtheorem{defn}[thm]{Definition}

\theoremstyle{remark}
\newtheorem{remark}[thm]{Remark}

\numberwithin{equation}{section}

\makeatletter
\@namedef{subjclassname@2020}{\textup{2020} Mathematics Subject Classification}
\makeatother

\newcommand{\vv}[1]{\mathbf{#1}}

\newcommand*\wrapletters[1]{\wr@pletters#1\@nil}
\def\wr@pletters#1#2\@nil{#1\allowbreak\if&#2&\else\wr@pletters#2\@nil\fi}

 \def \epsilon {{\varepsilon}}

\def \leq{\leqslant} \def \geq {\geqslant}
\def\le{\leqslant} \def\ge{\geqslant}

\def \bP {\mathbb P}

\newcommand{\C}{{\mathbb C}}         
\newcommand{\N}{{\mathbb N}}         

\newcommand{\R}{{\mathbb R}}        
\newcommand{\Z}{{\mathbb Z}}         
\newcommand{\id}{\mathrm{id}}

\def \bh {\mathbf h}

\def \bv {\mathbf v}

\def \bw {\mathbf w}

\def \fa {\mathfrak a}

\def \fd {\mathfrak d}

\def \fg {\mathfrak g}
\def \fh {\mathfrak h}

\def \fl {\mathfrak l}
\def \fm {\mathfrak m}

\def \fq {\mathfrak q}
\def \fr {\mathfrak r}
\def \fs {\mathfrak s}

\def \fu {\mathfrak u}

\def \fB {\mathfrak B}
\def \fC {\mathfrak C}
\def \fD {\mathfrak D}

\def \fJ {\mathfrak J}

\def \fL {\mathfrak L}

\def \fN {\mathfrak N}

\def \fQ {\mathfrak Q}

\def \fS {\mathfrak S}

\def \fZ {\mathfrak Z}

\def \cA {\mathcal A}
\def \cB {\mathcal B}
\def \cC {\mathcal C}

\def \cE {\mathcal E}
\def \cF {\mathcal F}

\def \cI {\mathcal I}

\def \cL {\mathcal L}
\def \cM {\mathcal M}
\def \cN {\mathcal N}

\def \cR {\mathcal R}
\def \cS {\mathcal S}
\def \cT {\mathcal T}

\def \cW {\mathcal W}

\def \cY {\mathcal Y}
\def \cZ {\mathcal Z}

\def \dim {\mathrm{dim}}

\def \supp {{\mathrm{supp}}}

\def \R {{\mathbb{R}}}
\def \Z {{\mathbb{Z}}}
\def \SL {{\mathrm{SL}}}

\def \dd {{\mathrm{d}}}

\def \Ad {{\mathrm{Ad}}}
\def \Stab {{\mathrm{Stab}}}
\def \HS {{\mathcal{HS}}}
\def \SS {{\mathcal{SS}}}
\def \Leb {{\mathrm{Leb}}}

\begin{document}
\title[Effective Ratner]{Effective version of Ratner's equidistribution theorem for $\SL(3,\R)$}
\author[L. Yang]{Lei Yang}
\address{Institute of Advanced Study, 1 Einstein Dr., Princeton, NJ, 08540, USA}
\email{lyang@ias.edu}
\address{Current address: Department of Mathematics, National University of Singapore, 119076, Singapore}
\email{lei.yang@nus.edu.sg}
\thanks{}
\date{}
\begin{abstract}
In this paper, we will prove an effective version of Ratner's equidistribution theorem 
for unipotent orbits in $\SL(3,\R)/\SL(3,\Z)$ with a natural Diophantine condition.
\end{abstract}
\maketitle

\section{Introduction}
\label{sec-intro}

\subsection{Ratner's theorem and its effective versions}
\par In 1991, Ratner \cite{ratner_2} proved the following fundamental theorem on 
equidistribution of unipotent orbits in homogeneous spaces:
\begin{thm}
  \label{thm:ratner}
  Let $G$ be a Lie group and $\Gamma$ be a lattice in $G$, namely, the homogeneous 
  space $X = G/\Gamma$ admits a $G$-invariant probability measure. Let $U = \{u(r): r\in \R\}$ be a one-parameter unipotent subgroup of $G$. 
  For any $x \in X$, the closure of the $U$-orbit  $U x$ of $x$ is a 
  closed $L$-orbit $Lx$ for some Lie subgroup $L \subset G$. Moreover, $Ux$ is equidistributed in $Lx$ with respect 
  to the unique $L$-invariant probability measure $\mu_L$, namely, for any $f \in C^\infty_c(X)$ (where $C^\infty_c(X)$ denotes the 
  space of smooth functions on $X$ with compact support),
  \[ \lim_{T \to \infty} \frac{1}{T} \int_{-T/2}^{T/2} f(u(r)x) \dd r  = \int_{Lx} f \dd \mu_L .\]
\end{thm}
\par Since Ratner's proof relies on her measure classification theorem 
(cf. \cite{ratner-acta}, \cite{ratner-invent}, \cite{Ratner}) which proves a conjecture by Dani \cite{Dani1981}, 
it does not tell how fast it tends to equidistribution. 
Therefore, a natural question is whether we can make the equidistribution effective, namely,
giving an explicit upper bound on the difference 
 \[ \left| \frac{1}{T} \int_{-T/2}^{T/2} f(u(r)x) \dd r - \int_{Lx} f \dd \mu_L \right| \]
 for given $T >0$. Moreover, effective versions of equidistribution of $U$-orbits have many applications to number theory, 
 see  \cite{lindenstrauss-margulis}, \cite{einsiedler-mohammadi}, \cite{LM-preprint2021}, \cite{LMW-preprint2022}, 
 \cite{Venkatesh2010}, \cite{nelson-venkatesh}, \cite{einsiedler-margulis-venkatesh}, \cite{chow-yang}, \cite{Browning2016} 
 and references therein for details. As a result, proving effective versions of Ratner's theorem has 
 attracted much attention and has been a major challenge in homogeneous dynamics. 
 \par In recent years, people have made significant progress in establishing 
 effective equidistribution results. For $G$ being nilpotent, the effective version of Ratner's theorem was established by Green and Tao \cite{green_tao2012}. 
 For $G = \SL(2, \R)$, since the one-parameter unipotent subgroup is horospherical with respect to the diagonal subgroup, 
 one can apply the thickening argument developed in Margulis's thesis (cf. \cite{margulis-thesis}, \cite{Klein-Mar-Effective-Equid}) and 
 spectral gap results for unitary representations of semisimple groups 
 to establish effective equidistribution. 
 See \cite{sarnak1982}, \cite{burger1990}, \cite{Flaminio-Forni}, \cite{Sarnak-Ubis}, \cite{Strombergsson2004} and \cite{Strombergsson2013} for works in this setting. 
 In fact, equidistribution for this case was proved before Ratner's theorem, see \cite{furstenberg1973}, \cite{dani-smillie}.
  Using similar arguments one can also establish effective 
 equidistribution results for horospherical unipotent orbits in homogeneous spaces. See \cite{Klein-Mar-Effective-Equid}, \cite{edwards-effective} for results in this setting. 
 Using techniques from analytic number theory, Strombergsson \cite{Strombergsson2015} 
 established the effective equidistribution for $G/\Gamma = \SL(2, \R) \ltimes \R^2/ \SL(2,\Z) \ltimes \Z^2$ and $U$ being 
 a one-parameter unipotent (and horospherical) subgroup in the semisimple part. Building on Strombergsson's result, Chow and the author \cite{chow-yang} proved an 
 effective equidistribution result for a special family of one-parameter unipotent orbits 
 in $G/\Gamma = \SL(3,\R)/\SL(3,\Z)$. As an application, we proved that 
 Gallagher's theorem in multiplicative Diophantine approximation holds for almost every point on any given planar straight line. 
 The reader is also referred to \cite{Browning2016} for a similar effective equidistibution result which has applications to number theory.
 Strombergsson's result was recently generalized by Kim \cite{kim2021} to $\SL(n, \R) \ltimes \R^n/ \SL(n,\Z) \ltimes \Z^n$ with 
 the unipotent subgroup being horospherical in the semisimple part. 
 Recently, Lindenstrauss, Mohammadi and Wang \cite{LMW-preprint2022} established effective Ratner's 
 theorem for one-parameter unipotent orbits in $G/\Gamma$ where $G = \SL(2, \R) \times \SL(2, \R)$ or $\SL(2,\C)$. 
 The reader is referred to \cite{einsiedler-margulis-venkatesh} and \cite{einsiedler-margulis-mohammadi-venkatesh}
 for effective equidistribution of closed orbits of maximal semisimple subgroups which can be regarded as an effective version 
 of a result by Mozes and Shah \cite{Mozes_Shah}.
 \par Concerning effective density, there are also significant results. Lindenstrauss and Margulis proved \cite{lindenstrauss-margulis} 
 an effective density result for one-parameter unipotent orbits in $\SL(3,\R)/\SL(3,\Z)$ and applied it to prove an effective version of Oppenheim's conjecture.
Recently, Lindenstrauss and Mohammadi \cite{LM-preprint2021} established an optimal effective density result or unipotent orbits 
in $G/\Gamma$ for $G = \SL(2, \R) \times \SL(2, \R)$ or $\SL(2, \C)$. 

\subsection{Notation}
\label{subsec-notation}
\par Throughout this paper we will fix the following notation.
\begin{enumerate}
\item For a normed vector space $V$ and $r >0$, let $B_V(r)$ denote the set of vectors in $V$ with norm $\le r$.
\item For a Lie group $J$ and $r >0$, let $B^J(r)$ denote the $r$-neighborhood of the identity in $J$. 
\item Given a quantity $\cA$, let $O(\cA)$ denote a quantity whose absolute value is $\le C \cA$ 
where $C >0$ is an absolute constant. By $\cA \ll \cB$, we mean $\cA = O(\cB)$. By $\cA \asymp \cB$, we mean $\cA \ll \cB$ and $\cB \ll \cA$.
\item For a normed vector space $V$ and a quantity $\cA$, let $O_V(\cA)$ denote a vector in $V$ whose norm is $O(\cA)$.
\item For a measurable subset $I \subset \R$, let $\Leb(I)$ denote the Lebesgue measure of $I$. 
\item For a finite set $S$, let $|S|$ denote the cardinality of $S$.
\item For $L >0$, let $[L]$ denote the interval $[-L/2, L/2]$.
\end{enumerate}

\subsection{Main results}
\label{subsec-main-results}
\par In this paper we will focus on the following case:
 Let $G = \SL(3,\R)$, $\Gamma = \SL(3,\Z)$, $X = G/\Gamma$. Let $U = \{u(r): r\in \R \}$ be a one-parameter unipotent subgroup of $G$ defined as follows:
\begin{equation}
  \label{eq:def-U}
  u(r) := \begin{bmatrix}
    1 & 0 & 0 \\ 0 & 1 & r \\ 0 & 0 & 1
  \end{bmatrix}.
\end{equation}
Let $\mu_G$ denote the unique $G$-invariant probability measure on $X$.
 \par For $\varpi > 0$, let us define 
\begin{equation}
  \label{eq:def-compact-set}
  X_\varpi := \left\{ x = g\Gamma \in X: \|g \vv v\| \ge \varpi \text{, } \forall \vv v \in \Z^3\setminus\{\vv 0\} \text{ or } \wedge^2 \Z^3\setminus\{\vv 0\}.  \right\}
\end{equation}
\par By Mahler's criterion, every compact subset of $X$ is contained in some $X_\varpi$.
\par For $t \in \R$, let us denote 
\begin{equation}
  \label{eq:define-a}
  a(t) := \begin{bmatrix}
    1 & & \\ & e^{t/2} & \\ & & e^{-t/2}
  \end{bmatrix} \in G,
\end{equation}
and 
\begin{equation}
  \label{eq:define-a0}
  a_0(t) := \begin{bmatrix}
    e^{t/3} & & \\ & e^{-t/6} & \\ & & e^{-t/6}
  \end{bmatrix} \in G.
\end{equation}
We will prove the following result concerning effective equidistribution of $U$-orbit:
\begin{thm}
  \label{thm:effective-U-orbit}
  There exist $\eta >0$ and $t_0 >1$ such that for any $t \ge t_0$ and any $x_0 \in X$, one of the following holds:
  \par (1) For any $f \in C^\infty_c(X)$, we have 
  \[ \left|  \int_{[1]} f(u(r e^t )x_0) \dd r - \int f \dd \mu_G \right| \ll e^{-\eta t } \|f\|_S,\]
  where $\|\cdot\|_S$ denotes a fixed Sobolev norm;
  \par (2) $a(-t) a_0( \ell) x_0 \not\in X_{e^{-|\ell|/3}}$ for some $ 0.99 t \le |\ell| \le 1.01 t$.
\end{thm}

\par Proving Theorem \ref{thm:effective-U-orbit} is equivalent to proving the following theorem:
\begin{thm}
  \label{thm:main-thm}
    Let $\eta, t_0 >0$ be as above. For any $t \ge t_0$, and any $x_0 \in G/\Gamma$, one of the following holds:
  \par (1) For any $f \in C^\infty_c(X)$, we have 
  \[ \left| \int_{[1]} f(a(t) u(r) x_0 ) \dd r  - \int f \dd \mu_G \right| \ll e^{-\eta t} \|f\|_S;\]
  \par (2) $ a_0 (\ell )x_0 \not\in X_{e^{-|\ell|/3}}$ for some  $ 0.99 t \le  |\ell| \le 1.01 t$.
\end{thm}
\par In fact, it follows from the following equality: 
\[  \int_{[1]} f(a(t) u(r) x_0 ) \dd r =  \int_{[1]} f(u(r e^t )y_0) \dd r, \]
where $y_0 = a(t)x_0$.

\par Using Theorem \ref{thm:main-thm}, we can prove the following two corollaries on effective equidistribution of 
expanding curves under translates of diagonal subgroups:

\begin{cor}
  \label{cor:straight-line}
  Let us denote 
  
  $$a_1 (t) := a(t) a_0(t) = \begin{bmatrix}
    e^{t/3} & & \\ & e^{t/3} & \\ & & e^{-2t/3}
  \end{bmatrix}.$$
  Let $ \vv w = (w_1, w_2) \in \R^2$ be a vector with Diophantine 
  exponent $\omega(\vv w) = \kappa \le 0.6$, namely, for any positive integer 
  $q \in \Z_+$, we have 
  $$\max \{\langle q w_1 \rangle, \langle q w_2 \rangle\} \gg q^{-\kappa},$$
  where $\langle \cdot \rangle$ denotes the distance to the nearest integer. 
  For $\vv v =(v_1, v_2) \in \R^2$ let us denote 
  \[ n(\vv v) := \begin{bmatrix}
    1 & & v_2 \\ & 1 & v_1 \\ & & 1
  \end{bmatrix}. \]
  Let us define
  \[\varphi: [1] \to \R^2  \]
  by $\varphi(r) := (w_1 r + w_2 , r)$.
  Then there exist $ \eta, t_0 >0$, such that for any $t \ge t_0$ and any $f \in C^\infty_c(X)$ we have 
  \[ \left| \int_{[1]} f(a_1(t) n(\varphi(r) )\Gamma)   \dd r - \int f \dd \mu_G  \right| \ll e^{-\eta t} \|f\|_S. \] 
\end{cor}
\begin{proof}[Proof assuming Theorem \ref{thm:main-thm}]
  \par Note that 
  \[  a_1(t) n(\varphi(r)) = z(w_1) a(t) u(r) a_0(t) n^\ast (-w_1, w_2) ,\]
  where 
  \[  z(w_1) := \begin{bmatrix}
    1 & w_1 & \\ & 1 & \\ & & 1
  \end{bmatrix}, \]
  and 
  \[  n^\ast (-w_1, w_2) := \begin{bmatrix}
    1 & -w_1 & w_2 \\ & 1 & \\ & & 1
  \end{bmatrix}. \]
  Then for any $f \in C^\infty_c(X)$ with zero integral, we have that 
  \[ \int_{[1]} f(a_1(t) n(\varphi(r))\Gamma) \dd r = \int_{[1]} f_{w_1}(a(t) u(r) a_0(t) n^\ast (-w_1, w_2)\Gamma) \dd r, \]
  where $f_{w_1}(x) = f(z(w_1)x)$. Note that since $\mu_G$ is $G$-invariant, $f_{w_1}$ also has zero integral.
  \par Let $x_0 = a_0(t) n^\ast(-w_1, w_2)\Gamma$. 
  Note that the Diophantine condition on $\vv w$ ensures that $a_0 (\ell)x_0 \in X_{e^{-|\ell|/3}}$ for 
  any $0.99t \le |\ell| \le 1.01 t$. 
  Applying Theorem \ref{thm:main-thm} with $x_0 = a_0(t) n^\ast(-w_1, w_2)\Gamma$ and $f_{w_1}$, we have that 
  \[ \left| \int_{[1]} f_{w_1}(a(t) u(r) a_0(t) n^\ast (-w_1, w_2)\Gamma) \dd r \right| \ll e^{-\eta t} \|f_{w_1}\|_S. \]
  Noting that $\|f_{w_1}\|_S \ll \|f\|_S$, we complete the proof.
\end{proof}

\begin{remark}
  \label{rmk:expanding-line}
  The qualitative version of Corollary \ref{cor:straight-line} was proved in \cite{KNSY} for $\omega(\vv w) \le 2$ using Ratner's theorem. 
  The authors also proved that the equidistribution does not hold for $\omega(\vv w) > 2$. 
  
\end{remark}

\begin{cor}
  \label{cor:curve}
  Let us use the same notation as in Corollary \ref{cor:straight-line}. Let $\psi : [1] \to \R^2$ be a 
  smooth non-degenerate curve in $\R^2$, namely, derivatives of $\psi$ span the whole space $\R^2$ at every $r \in [1]$.
  Then there exist $\eta, t_0>0$ such that for any $t \ge t_0$ and any $f \in C^\infty_c(X)$ we have 
  \[ \left| \int_{[1]} f(a_1(t) n(\psi(r) )\Gamma)   \dd r - \int f \dd \mu_G  \right| \ll e^{-\eta t} \|f\|_S. \]
\end{cor}

\begin{proof}[Proof assuming Theorem \ref{thm:main-thm}]
  \par Let us denote $\psi(r) = (r, \psi_2(r))$. Let us divide $[0,1]$ into small pieces of size $e^{-t/2}$. Let us fix a small piece 
  $\Delta(r_0) = [r_0 - 1/2 e^{-t/2} , r_0 + 1/2 e^{-t/2}]$. Then
  \[\{a_1(t)n(\psi(r))\Gamma: r \in \Delta(r_0) \}\]
  can be approximated by 
  \[ \{z(\psi_2'(r_0)) a(t/2) u(r) x_0: r \in [-1/2,1/2]\} \]
  with exponentially small error, where 
  \[ z (\psi_2'(r_0)) = \begin{bmatrix}
    1 & \psi_2'(r_0) & \\ & 1 & \\ & & 1
  \end{bmatrix},\]
  as defined in Corollary \ref{cor:straight-line}, and 
  \[ x_0 = a_0(t/2) a_1(t/2) z(-\psi_2'(r_0)) n(\psi(r_0)) \Gamma. \]
  Then if we can show that
  \[ \left\{ a(t/2) u(r) x_0: r \in [1] \right\} \]
  effectively equidistributed, we will get that 
  \[\left\{a_1(t)n(\psi(r)): r \in \Delta(r_0) \right\}\]
  is effectively equidistributed. By Theorem \ref{thm:main-thm}, if we have that for any 
  $0.49 t \le |\ell| \le 0.51 t$, 
  \[ a_0(\ell)a_0(t/2) a_1(t/2) z(-\psi_2'(r_0)) n(\psi(r_0)) \Gamma \in X_{e^{-t/10}},  \]
  then we are done. Now for a fixed $|\ell| \in [ 0.49 t , 0.51 t]$, let us estimate the measure of 
  \[ \fm_{\ell} :=  \left\{ r_0\in [1]:a_0(\ell)a_0(t/2) a_1(t/2) z(-\psi_2'(r_0)) n(\psi(r_0)) \Gamma \not\in X_{e^{-t/10}} \right\}.\]
  By \cite[Theorem 1.4]{Bernik-Kleinbock-Margulis}, we have that 
  $$\Leb\left(\fm_{\ell}\right) = O(e^{-\eta_2 t})$$ 
  for some constant $\eta_2 >0$ independent of $\ell$. Let us remove all 
  $\fm_{\ell}$ from $[1]$ and get a subset $\fm \subset [1]$. Then we have that 
  $$\Leb \left([1]\setminus \fm\right) = O(t e^{-\eta_2 t}),$$
  and for any $r_0 \in \fm$, 
  \[\left\{a_1(t)n(\psi(r))\Gamma: r \in \Delta(r_0) \right\}\]
  is effectively equidistributed in $X$. Combining these two facts we conclude that the whole orbit 
  \[ \left\{a_1(t) n(\psi(r)) \Gamma : r \in [1]\right\} \]
  is effectively equidistributed. This completes the proof.
\end{proof}

\begin{remark}
  \label{rmk:expanding-curve}
  The ineffective version of Corollary \ref{cor:curve} was proved in \cite{Shah_2} using Ratner's theorem.
\end{remark}

\subsection*{Acknowledgements} The author thanks Wen Huang, Dmitry Kleinbock, Elon Lindenstrauss, Amir Mohammadi, 
Nimish Shah, Ralf Spatzier, Zhiren Wang and Barak Weiss for valuable discussions and Victor Beresnevich, 
Sam Chow, Wen Huang, Elon Lindenstrauss, Jens Marklof and Pengyu Yang for valuable 
comments on an earlier version of the paper. Especially, the author is deeply indebted to Elon Lindenstrauss for spending an enormous amount of time line-by-line checking the details of the proofs, 
and giving numerous valuable comments and suggestions, which greatly improved the exposition of the present version. 
The author is supported in part by NSFC grant No. 12171338 and the Shiing-Shen Chern membership at the Institute for Advanced Study.


\section{Preliminaries}
\label{sec-prelim}
\par In this section we recall some basic facts on $\SL(3, \R)$ and its Lie algebra 
which will be used in the proof of our main theorem.
\par Let $G = \SL(3,\R)$, and $H \subset G$ be the following subgroup of $G$:

\begin{equation}
\label{eq:def-H}
H := \left\{ \begin{bmatrix} 1 & \\ & h  \end{bmatrix} \in G : h \in \SL(2,\R) \right\}.
\end{equation}
Clearly $H$ is isomorphic to $\SL(2,\R)$.
\par Let $\fg$ and $\fh$ denote the Lie algebras of $G$ and $H$, respectively.
\par Consider the adjoit action of $H$ on $\fg$, we have the following decomposition:
\begin{equation}
  \label{eq:decomp-g}
  \fg = \fh + \fr_0 + \fr_1 + \fr_2,
\end{equation}
where $\fr_0, \fr_1, \fr_2$ are irreducible subrepresentations with respect 
to the adjoint action of $H$, 
and $\dim \fr_1 = \dim \fr_2 =2$, $\dim \fr_0 =1$.
In particular, $\fr_0 = \R \fa_0$ where
\[ \fa_0 = \begin{bmatrix}
  1/3 & 0 & 0 \\ 0 & -1/6 & 0 \\ 0 & 0 & -1/6
\end{bmatrix};  \]
$\fr_1 = \R \bv_1 + \R \bv_2$ where 
\[ \vv v_1 = \begin{bmatrix}
  0 & 0 & 0 \\ 
  1 & 0 & 0 \\
  0 & 0 & 0
\end{bmatrix}, \]
and 
\[\vv v_2 = \begin{bmatrix}
  0 & 0 & 0 \\ 
  0 & 0 & 0 \\
  1 & 0 & 0
\end{bmatrix};\]
$\fr_2 = \R \bw_1 + \R \bw_2$ where 
\[ \vv w_1 = \begin{bmatrix}
  0 & 0 & 1 \\ 
  0 & 0 & 0 \\
  0 & 0 & 0
\end{bmatrix}, \]
and 
\[\vv w_2 = \begin{bmatrix}
  0 & 1 & 0 \\ 
  0 & 0 & 0 \\
  0 & 0 & 0
\end{bmatrix}.\]
The adjoint action of $H$ on $\fr_1$ and $\fr_2$ are given as 
follows: the adjoint action of $H$ on $\fr_1$ is the same as the standard action of $\SL(2,\R)$ on $\R^2$ 
if we choose $\{\vv v_1, \vv v_2\}$ as the basis; for $h$ corresponding to $\begin{bmatrix}
  a & b \\ c & d
\end{bmatrix}$ and $\vv w = x_1 \vv w_1 + x_2 \vv w_2$, 
\[\Ad (h) \vv w = (a x_1 - b x_2) \vv w_1 + (-c x_1 + d x_2) \vv w_2.\]

\par Let $\{\fa , \fu, \fu^\ast \}$ denote the standard basis of $\fh$, where 
\[ \fa := \begin{bmatrix}
  0 & 0 & 0 \\ 0 & 1/2 & 0 \\ 0 & 0 & -1/2 
\end{bmatrix}, \]

\[ \fu := \begin{bmatrix}
  0 & 0 & 0 \\ 0 & 0 & 1 \\ 0 & 0 & 0 
\end{bmatrix}, \]

and
\[ \fu^\ast := \begin{bmatrix}
  0 & 0 & 0 \\ 0 & 0 & 0 \\ 0 & 1 & 0 
\end{bmatrix}. \]

\par Let us denote 
\[ A := \left\{ a(t) = \exp(t \fa): t \in \R \right\} \subset H, \]
\[U := \left\{u(r) := \exp(r \fu) : r \in \R \right\} \subset H,\]
\[U^\ast := \left\{u^\ast (r) := \exp(r \fu^\ast): r \in \R \right\} \subset H,\]
and 
\[ A_0 := \left\{ a_0(t) := \exp(t \fa_0) : t \in \R \right\} \subset Z_G(H). \]
Note that the definitions of $a(t)$ and $a_0(t)$ here match \eqref{eq:define-a} 
and \eqref{eq:define-a0}, respectively.
The adjoint action of $a_0(t)$ on $\fg$ is as follows:
\[ \Ad (a_0(t)) \bh = \bh, \text{ for } \bh \in \fh, \]
\[\Ad (a_0(t)) \bv = e^{-t/2} \bv, \text{ for } \bv \in \fr_1, \]
and 
\[ \Ad (a_0(t)) \bw = e^{t/2} \bw, \text{ for } \bw \in \fr_2. \]
Let us denote 
\begin{equation}
  \label{eq:def-b}
  b(t) := a(t)a_0(-t) = \begin{bmatrix} e^{-t/3} & & \\ & e^{2t/3} & \\ & & e^{-t/3} \end{bmatrix} \in G, \end{equation}
and
\begin{equation}
  \label{eq:def-a1}
  a_1(t) := a(t)a_0(t) = \begin{bmatrix}
    e^{t/3} & & \\ & e^{t/3} & \\ & & e^{-2t/3}
  \end{bmatrix} \in G.
\end{equation}
\par Let us denote 
\begin{equation}
  \label{eq:r+}
  \vv r^+ := \R \vv w_1 + \R \vv v_1,
\end{equation}
and 
\begin{equation}
  \label{eq:r-}
  \vv r^-:= \R \vv w_2 + \R \vv v_2,
\end{equation}
Let us denote by $p_{1,2}: \fg \to \fr_1 + \fr_2$ the projection of $\fg$ to $\fr_1 + \fr_2$. For $j=0, 1,2$, 
let $p_j : \fg \to \fr_j$ denote the projection from $\fg$ to $\fr_j$. 
 Let $p_+$, $p_-$, $p_{\vv w_j}$, $p_{\vv v_j}$ and $p_{\fu^\ast}$ denote 
 the projection to $\vv r^+$, $\vv r^-$, $\R \vv w_j$, $\R \vv v_j$ and $\R \fu^\ast$, respectively.

 \par We will need the following simple lemma on the Lie algebra of $\SL(3, \R)$:
 \begin{lemma}
    \label{lm:lie-bracket}
    For $\vv w^+ \in \vv r^+$, and $\vv w^- \in \vv r^-$,
    \[ [\vv w^+, \vv w^-] \in \R \fa + \R \fa_0, \]
    \[ [\vv w^+, \fa] , [\vv w^+, \fa_0] \in \vv r^+, \]
    \[ [\vv w^-, \fa] , [\vv w^-, \fa_0] \in \vv r^-, \]
    \[ [\vv w^-, \fu] \in \vv r^+, [\vv w^-, \fu^\ast] =\vv 0, \]
    and 
    \[ [\vv w^+, \fu^\ast] \in \vv r^-, [\vv w^+, \fu] = \vv 0. \]
    For $\vv w_1^+, \vv w_2^+ \in \vv r^+$,
    \[ [\vv w_1^+, \vv w_2^+] \in \R \fu. \]
    For $\vv w_1^-, \vv w_2^- \in \vv r^-$,
    \[ [\vv w_1^-, \vv w_2^-] \in \R \fu^\ast. \]
    For $\vv h \in \fh$, 
    \[ [\fa_0, \vv h] = \vv 0. \]
    Also, we have 
    \[ [\fa, \fu] = \fu, \text{  } [\fa, \fu^\ast] = - \fu^\ast, \text{ and }[\fu, \fu^\ast] = 2 \fa. \]
 \end{lemma}
 \begin{proof}
     \par The statements can be verified by direct calculation.
 \end{proof}
 \par We also need the following lemmata concerning the Lie algebra coordinate of the product of two elements in $G$:
 \begin{lemma}
     \label{lm:product-lie-algebra-1}
     \par Let $0 < \rho_1, \rho_2 < 10^{-10}$ be small constants. For $\vv w \in B_{\fr_1 + \fr_2}(\rho_1)$ and $\vv h \in B_{\R \fa_0 + \fh}(\rho_2)$, we have that 
     \begin{equation} 
     \label{eq:product-w-h-1}
     \exp(\vv w) \exp(\vv h) = \exp(\tilde{\vv w} +\tilde{\vv h}), \end{equation}
     where $\tilde{\vv w} \in \fr_1 + \fr_2$ with $\|\vv w - \tilde{\vv w}\| \le 2\rho_2 \|\vv w\|$, and 
     $\tilde{\vv h} \in \R \fa_0 + \fh$ with 
     $$\|\vv h - \tilde{\vv h} \| \le 2\rho_1 \|\vv h\|.$$
     \par Conversely, for $\tilde{\vv w} \in B_{\fr_1 + \fr_2}(\rho_1)$ and $\tilde{\vv h} \in B_{\R \fa_0 + \fh}(\rho_2)$, we have 
     \begin{equation} 
     \label{eq:product-w-h-2}
     \exp(\tilde{\vv w} +\tilde{\vv h}) = \exp(\vv w) \exp(\vv h), \end{equation}
     where $\vv w \in \fr_1 + \fr_2$ with $\|\vv w - \tilde{\vv w}\| \le 4\rho_2 \|\tilde {\vv w}\|$, and 
     $\vv h \in \R \fa_0 + \fh$ with 
     $$\|\vv h - \tilde{\vv h} \| \le 4\rho_1 \|\tilde{\vv h}\|.$$
 \end{lemma}
 \begin{proof}
     \par Let $\rho_0 = \max\left\{\rho_1, \rho_2\right\}$. By the Baker-Campbell-Hausdorff formula, we have 
     \[ \exp(\vv w)\exp(\vv h) = \exp\left(\vv w + \vv h + P\left(\vv w, \vv h\right)\right), \]
     where
     \begin{align*} 
     P(\vv w , \vv h) =& \frac{1}{2} \left[ \vv w, \vv h \right] + \frac{1}{12} \left( \left[ \vv w, \left[\vv w, \vv h\right] \right] + \left[ \vv h, \left[\vv h, \vv w\right] \right]\right) + \cdots.
     \end{align*}
     By Lemma \ref{lm:lie-bracket} and the fact that for any $\vv w_1, \vv w_2 \in \fg$, 
     $\|[\vv w_1, \vv w_2]\| \le \|\vv w_1\| \cdot \|\vv w_2\|$, we have that for any monomial $[\cdots \diamond \cdots]$ in $P(\vv w, \vv h)$ with $i$ Lie brackets, we have $[\cdots \diamond \cdots] \in \fr_1 + \fr_2$ or $[\cdots \diamond \cdots] \in \R \fa_0 + \fh$, and 
     \[ \|[\cdots \diamond \cdots]\| \le \rho_0^{i-1} \|\vv w\|\cdot \|\vv h\|. \]
     Therefore, we have that 
     \[ P(\vv w , \vv h) = \bar{\vv w} + \bar{\vv h}, \]
     where $\bar{\vv w} \in \fr_1 + \fr_2$ with $\|\bar{\vv w}\| \le 2\rho_2 \|\vv w\|$, and $\bar{\vv h} \in \R \fa_0 + \fh$ with $\|\bar{\vv h}\| \le 2\rho_1 \|\vv h\|$. This proves \eqref{eq:product-w-h-1}.
     \par Conversely, for $\tilde{\vv w} \in B_{\fr_1 +\fr_2}(\rho_1)$ and $\tilde{\vv h} \in B_{\R \fa_0 + \fh}(\rho_2)$, let us write 
     \[ \exp(\tilde{\vv w} + \tilde{\vv h}) = \exp(\vv w) \exp(\vv h), \]
     where $\vv w \in B_{\fr_1 + \fr_2}(10^{-10})$ and $\vv h \in B_{\R \fa_0 + \fh}(10^{-10})$. Then by \eqref{eq:product-w-h-1}, we have 
     \[ \| \vv w - \tilde{\vv w} \| \le 2 \|\vv h\| \|\vv w\|, \text{ and }  \| \vv h - \tilde{\vv h} \| \le 2 \|\vv w\| \|\vv h\|. \]
     The above inequalities imply that 
     $$\|\vv w\| \le 2 \|\tilde{\vv w}\| \le 2 \rho_1,$$ 
     and 
     $$\|\vv h\| \le 2 \|\tilde{\vv h}\| \le 2 \rho_2. $$ 
     Therefore, we have 
     \[ \| \vv w - \tilde{\vv w} \| \le 4 \rho_2 \|\tilde{\vv w}\|, \text{ and }  \| \vv h - \tilde{\vv h} \| \le 4 \rho_1 \|\tilde{\vv h}\|. \]
     This proves \eqref{eq:product-w-h-2}.
 \end{proof}
 \begin{lemma}
 \label{lm:product-w}
 Let $0 < \rho < \rho_0 < 10^{-10}$ be two small constants. Let $\vv w, \vv w' \in B_{\fr_1 + \fr_2} (\rho_0)$ be such that 
 \[ \|[\vv w, \vv w']\| \le \rho. \]
 Then 
 \[ \exp(\vv w ) \exp( \vv w') = \exp(\tilde{\vv w}) \exp( \vv h), \]
 where $\tilde{\vv w} \in \fr_1 + \fr_2$ with $\| \tilde{\vv w} - \vv w - \vv w'\| \le 26\rho \rho_0$, and $\vv h \in B_{\R \fa_0 + \fh}(4\rho)$.
 \end{lemma}
 \begin{proof}
  \par By the Baker-Campbell-Hausdorff formula, we have 
     \[ \exp(\vv w)\exp(\vv w') = \exp\left(\vv w + \vv w' + P\left(\vv w, \vv w'\right)\right), \]
     where
     \begin{align*} 
     P(\vv w , \vv w') =& \frac{1}{2} \left[ \vv w, \vv w' \right] + \frac{1}{12} \left( \left[ \vv w, \left[\vv w, \vv w'\right] \right] + \left[ \vv w', \left[\vv w', \vv w\right] \right]\right) + \cdots.
     \end{align*}
     \par By Lemma \ref{lm:lie-bracket}, we have the following observations:
     \par For $i = 2j+1$, every monomial $[\cdots\diamond\cdots]$ in $P(\vv w, \vv w')$ with $i$ Lie brackets is in $\R \fa_0 + \fh$. For $i = 2j$,  every monomial $[\cdots \diamond \cdots]$ in $P(\vv w, \vv w')$ with $i$ Lie brackets is in $\fr_1 + \fr_2$. 
     \par Moreover, by the condition $\|[\vv w, \vv w']\| \le \rho$ and the fact that for any $\vv g_1, \vv g_2 \in \fg$, 
     $\|[\vv g_1, \vv g_2]\| \le \|\vv g_1\| \cdot \|\vv g_2\|$, we have the following estimates: 
     \par For $i = 2j+1$, and any monomial $[\cdots \diamond \cdots]$ with $i$ Lie brackets,
     $$\left\|[\cdots \diamond \cdots]\right\| \le \rho \rho_0^{2j};$$
     
     \par For $i = 2j$, and any $[\cdots\diamond\cdots]$ with $i$ Lie brackets, 
     
     \[ \|[\cdots \diamond \cdots]\| \le \rho \rho_0^{2j-1}. \]
    Thus, we have that
    \[ P(\vv w, \vv w') = \bar{\vv w}  + \bar{\vv h}, \]
    where $\bar{\vv w} \in \fr_1 + \fr_2$ with $\|\bar{\vv w}\| \le 2 \rho \rho_0$, and $\bar{\vv h} \in \R \fa_0 + \fh$ with $\|\bar{\vv h}\| \le 2 \rho$. This implies that 
    \[ \exp(\vv w) \exp(\vv w') = \exp(\hat{\vv w} + \bar{\vv h}),\]
    where $\hat{\vv w} = \vv w + \vv w' + \bar{\vv w}$. In particular, 
    \begin{equation}
    \label{eq:hatw-ww'}
    \| \hat{\vv w} - \vv w - \vv w'\| \le 2 \rho \rho_0.
    \end{equation}
    It is easy to see that 
    \[ \|\hat{\vv w}\| \le 3 \rho_0. \]
    By Lemma \ref{lm:product-lie-algebra-1}, we have that 
    \[ \exp(\hat{\vv w} + \bar{\vv h}) = \exp(\tilde{\vv w}) \exp(\vv h) \]
    where $\tilde{\vv w} \in \fr_1 + \fr_2$ satisfying
    \begin{equation}
    \label{eq:tildew-hatw}
    \| \tilde{\vv w} - \hat{\vv w} \| \le 8 \rho \|\hat{\vv w}\| \le 24 \rho_0 \rho,
    \end{equation}
    and $\vv h \in \R \fa_0 + \fh$ with $\|\vv h\| \le 4 \rho$. Noting that \eqref{eq:hatw-ww'} and \eqref{eq:tildew-hatw} imply that 
    \[ \| \tilde{\vv w} - \vv w - \vv w'\| \le 26 \rho_0 \rho, \]
    we conclude the lemma.
 \end{proof}

 \subsection{Outline of the proof}
 \label{subsec-outline}
 \par In this subsection we will give the outline of the proof. 
 \par The proof is inspired by Ratner's original proof of her measure rigidity theorem \cite{ratner-acta} and recent papers by 
 Lindenstrauss, Mohammadi, and Wang \cite{LM-preprint2021,LMW-preprint2022}. 
 However, compared with previous works, we take a quite different approach in this paper.
 \par In \S \ref{sec-initial-dim-control}, we We will start with $\cF_0 = a(s)u([1])x_0$ for $s = t - \delta_1 t$ and 
 study the dimension of the normalized Lebesgue measure $\mu_{\cF_0}$ on $\cF_0$ in directions transversal to the $A_0 H$-orbit direction. We will show that $\cF_0$ has certain dimension control 
 in the transversal direction unless condition (2) in Theorem \ref{thm:main-thm} holds (see Proposition \ref{prop:initial-bound-dimension}). 
 Here we are allowed to remove an exponentially small proportion from $\cF_0$. At this step quantitative non-divergence 
 results proved in \S \ref{sec-quantitative-nondivergence} are needed.
 This is the starting point of our proof. The argument in this part is similar to the 
 corresponding parts in \cite{LM-preprint2021,LMW-preprint2022}. 

 \par  \S \ref{sec-reduction} to \S \ref{sec-structured-component} is the crucial part. In these sections, we will construct a sequence $\{\cF_i : i \in \N\}$
 of $U$-orbits starting from $\cF_0$ where $\cF_{i+1} = a(s_i)\cF_i$ for each $i \ge 0$. 
 We will prove Proposition \ref{prop:improve-dimension} which says that $\mu_{\cF_{i+1}}$ 
 has a better dimension control compared with $\mu_{\cF_i}$ at a larger scale as a cost. 
 To prove this, in \S \ref{sec-kakeya-model} we introduce 
 a Kakeya-type model to study the behavior of $U$-orbits in a neighborhood. The proof of Proposition \ref{prop:improve-dimension} is divided into three parts: Proposition \ref{prop:dimension-improvement-random}, Proposition \ref{prop:dimension-improvement-strange-part} from \S \ref{sec-dimension-improvement-unstructured}, and Proposition \ref{prop:dimension-improvement-structured-component} from \S \ref{sec-structured-component}.
 The key of the proofs is to calculate the weighted intersection number of the Kakeya-type model. The final outcome of the calculation is the following: 
 Either $\mu_{\cF_{i+1}}$ has a better dimension control, or the orbit is concentrated on a small neighborhood of a 
 piece of an orbit of a subgroup isomorphic to $\SL(2, \R) \ltimes \R^2$ (see Lemma \ref{lm:curve-in-N}). 
 The latter case can be ruled out by our Diophantine condition (see Proposition \ref{prop:sl2xr2-closing-lemma}).
  At each step we are allowed to remove an exponentially small proportion. This part is novel compared with previous works.
 \par The inductive construction given as above will conclude that for some $n >0$, the dimension control of $\mu_{\cF_n}$ along 
 the directions transversal to the $A_0 H$-orbit direction is almost optimal. Then we 
 can apply Proposition \ref{prop:high-dimension-to-equidistribution-a} from \S \ref{sec-high-dim-to-equidistribution} 
 to conclude the effective equidistribution. Proposition \ref{prop:high-dimension-to-equidistribution-a} can be 
 proved by following a Van der Corput type argument due to Venkatesh.


\section{High transversal dimension to effective equidistribution}
\label{sec-high-dim-to-equidistribution}

\par The following statement allows us to get effective equidistribution 
from high transversal dimension control. It can be proved by following Venkatesh's argument (\cite[\S 3]{Venkatesh2010} and \cite[Proposition 4.2]{LM-preprint2021} for details).

\begin{prop}
  \label{prop:high-dimension-to-equidistribution-a}
 There exists $ \theta >0$ such that the following holds: 
 For any $0 < \epsilon < 10^{-1000}$, there exist $ \eta >0$ and $s_0 \gg 1$ depending on $\epsilon >0$ such that for any $s' \ge s_0$, with $ e^{-\epsilon^2 s'} \le \beta \le 10^{-10}$, 
 any Borel probability measure $\rho$ on $[\beta]^2$ with dimension 
  larger than $2-2\theta$ at scale $s'$, that is, for any square $S \subset [\beta]^2$ of size $\beta e^{-s'}$, we have
  \[ \rho(S) \le \left(\beta e^{-s'}\right)^{2-2\theta},\]
  any function $f \in C_c^\infty(X)$, and any $x \in X_{\beta^{1/4}}$, we have  
  \[ \left| \int_{[\beta]^2} \int_{[1]} f(a(2s')\exp( v_1 \vv v_1 +  w_1 \vv w_1)u(\beta r) x) \dd r \dd \rho(w_1, v_1) - \int f \dd \mu_G \right| \ll e^{-\eta s'} \|f\|_S.\]
\end{prop}

\begin{proof}
  \par The statement can be proved by following the proof of \cite[Proposition 4.2]{LM-preprint2021} step by step.
\end{proof}

\section{Quantitative non-divergence}
\label{sec-quantitative-nondivergence}
\par In this section, we will prove a quantitative non-divergence result and a quantitative separation lemma, which will be used to prove Proposition \ref{prop:initial-bound-dimension}.
\begin{prop}
  \label{prop:quantitative-nondivergence}
  There exists $ \alpha > 0$ such that the following holds: 
  For any $\beta, s >0$, if $x \in X$ satisfies that
   $ a_0(-s)x ,  a_0(s)x \not\in X_{e^{- s/3 }}$, then,
  \[ \Leb\left(\left\{r \in [1]: a(s)u(r )x \not\in X_\beta \right\}\right) \ll \beta^{\alpha}. \]
\end{prop}
\begin{proof}
  \par Let us denote $x = g \Gamma$. By the Kleinbock-Margulis quantitative non-divergence theorem 
  (cf. \cite[Theorem 2.2]{kleinbock2008}, \cite[Theorem 5.2]{Klein_Mar}), if the statement does not hold then there exists 
  a nonzero $\vv v \in \bigwedge^j \Z^3$ (where $j = 1$ or $2$), such that 
  \[ \max_{r \in [1]}\| a(s) u(r) g \vv v\| \le 1. \]
  \par \textbf{Case 1.} $\vv v \in \Z^3$: Let us denote $ \vv v' = a(s) g \vv v = (v'_1, v'_2, v'_3)$. Then 
  \[ a(s) u(r) g \vv v = u(r e^{s}) \vv v' = (v'_1 , v'_2 + r e^{s} v'_3, v'_3). \]
  Then $\|a(s)  u(r ) g \vv v\| \le 1$ implies that $|v'_1| \le 1 $, $|v'_2 + r e^{s} v'_3| \le 1$ for any $r \in [1]$.
  The latter easily implies that $|v'_2| \le 1$ and $|v'_3| \le e^{-s}$. This implies that $\|a_1(-s) \vv v'\|\le e^{-s/3}$.
  Noting that $a_1(s) = a(s) a_0(s)$, we complete the proof for \textbf{Case 1}.
  \par \textbf{Case 2.} $\vv v \in \bigwedge^2 \Z^3$: Let us denote $\vv v' = a(s) g \vv v = (v'_1, v'_2, v'_3 )$ where 
  we use the coordinates with respect to the basis $\{\vv e_2 \wedge \vv e_3, \vv e_1 \wedge \vv e_3, \vv e_1 \wedge \vv e_2\}$. Then
  \[ a(s) u(r ) g \vv v = u(r e^s) \vv v' = (v'_1 , v'_2 , v'_3 + r e^{s} v'_2). \]
  By repeating the same argument as in Case 1, we have $|v'_1|, |v'_3| \le 1$ and $|v'_2| \le e^{-s}$ 
  which implies that $\|b(-s) \vv v'\|\le e^{-s/3}$. Noting that $b(s) = a(s) a_0(-s)$, we complete the proof for \textbf{Case 2}.

\end{proof}

\par Note that for any $ \beta> 0 $ small enough and 
any $x \in X_{\beta^{1/4}}$, $B^G(\beta)x$ embeds into $X$ injectively.

\par We also need the following quantitative separation lemma:

\begin{lemma}
  \label{lm:separation-lemma-a}
  For any $\beta >0$, any $\ell \ge 1$, 
  and any $x \in X$ with $a(-\ell)x \in X_{\beta^{1/4}}$, if 
  \[ \exp( r_1 e^\ell \fu + r_2 e^{\ell/2} \vv v_1 + r_3 e^{\ell/2} \vv w_1) x = \exp(\vv v)x \]
  where $ r_1 , r_2, r_3 \in [\beta]$, $\|(r_1, r_2, r_3)\| \ge \beta/4$, and $\vv v \in B_\fg (\beta)$. Let us write
  \[ \vv v = u \fu + u^\ast \fu^\ast + a \fa + a_0 \fa_0 + \sum_{j=1}^2 w_j \vv w_j +  v_j \vv v_j.  \]
  if we have $|u^\ast| < \beta e^{-\ell}$, then $\|(v_2, w_2)\| \ge \beta e^{-\ell/2}$.
\end{lemma}
\begin{proof}
  \par For a contradiction, let us assume that $|u^\ast|<\beta e^{-\ell}, |v_2, w_2| < \beta e^{-\ell/2}$. 
  Let us denote $x = g\Gamma$. Then 
  \[ \exp( r_1 e^\ell \fu + r_2 e^{\ell/2} \vv v_1 + r_3 e^{\ell/2} \vv w_1) g = \exp(\vv v)g \gamma, \]
  for some $\gamma \in \Gamma$. For $\ell >1$ large enough, we have $\gamma \neq e$. Then we have 
  \[  \exp(-\vv v) \exp( r_1 e^\ell \fu + r_2 e^{\ell/2} \vv v_1 + r_3 e^{\ell/2} \vv w_1) = g \gamma g^{-1}. \]
  Then 
  \begin{align*} 
    a(-\ell) g \gamma g^{-1} a(\ell)  &= a(-\ell)\exp(-\vv v) \exp( r_1 e^\ell \fu + r_2 e^\ell \vv w_1) a(\ell) \\
                                          &= \exp(- \Ad (a(-\ell))\vv v ) \exp (r_1  \fu + r_2  \vv v_1 + r_3 \vv w_1).
  \end{align*}
  Since $|u^\ast| < \beta e^{-\ell}, |v_2, w_2| < \beta e^{-\ell/2}$, we have $\|\Ad (a(-\ell))\vv v\| < \beta$. This implies that 
  \[ a(-\ell) g \gamma g^{-1} a(\ell) \in B^G(\beta), \]
  which means that $B^G(\beta) a(-\ell) x$ does not embed into $X$ injectively. This contradicts to that $a(-\ell)x \in X_{\beta^{1/4}}$.
  This completes the proof.
\end{proof}

\section{List of Constants}
\label{subsec-constants}
\par Here we list all constants which will be used in the arguments later.
\begin{enumerate}
  \item $ \delta_1=\delta_2 = 10^{-5}$, $d_0 = 10^{-6}$, $C_1 = 10 d_0^{-1} = 10^7$.
  \item Let $\theta >0$ be the constant from Theorem \ref{prop:high-dimension-to-equidistribution-a}. Without loss of generality, we can assume that $\theta < 10^{-10}$.
  \item Let $\alpha>0$ be the constants as in Proposition \ref{prop:quantitative-nondivergence} and $\alpha_2 = \alpha/4$.
  \item $\epsilon_2 = 10^{-20}  d_0 \theta$, $N_1 = 10^5 \epsilon_2^{-1}$, $N_2 = 10^6$, $\epsilon = e^{- 10^5 N_1}$, $\beta = e^{-\epsilon^3 t}$, and $\delta_3 = 10^{-5}  \theta \epsilon$.
\end{enumerate}

\section{Initial Dimension Control}
\label{sec-initial-dim-control}

\subsection{Initial construction of transversal sets}
\label{subsec-initial-construction}
\par We first introduce some notation. 
\par For $\ell >1$ and a subspace $V \subset \fr_1 + \fr_2$ generated by some vectors from the canonical basis of $\fr_1 + \fr_2$,
namely, $\{ \vv v_{1}, \vv v_2, \vv w_1, \vv w_{2}\}$,
 let us define
\begin{equation}
  \label{eq:def-thin}
  Q(V, \ell, \bar \beta) := \{\exp(\vv w): \vv w \in B_{\fr_1 + \fr_2}(\bar \beta), \|p_{V}(\vv w)\|\le \bar \beta e^{-\ell}\},
\end{equation}
where $p_V$ denotes the projection to $V$ with respect to the canonical basis. 
\par Let us denote 
{
\begin{align}
\label{eq:def-qH}
\fq^{A_0 H}(\ell, \bar \beta )&:= \left\{ \vv h \in B_{\R \fa_0 + \fh} (\bar \beta): \| p_{\fu^\ast}(\vv h) \| \le \bar \beta e^{-\ell}  \right\}, \\
\label{eq:def-QH}
Q^{A_0 H}(\ell, \bar \beta) &:= \left\{ h = \exp(\vv h): \vv h \in \fq^{A_0 H}(\ell, \bar \beta) \right\},
\end{align}
and 
\begin{align}
\label{eq:def-rH}
\fr^{A_0 H}(\ell, \bar \beta) &:= \left\{ \vv h \in B_{\R \fa_0 + \fh} (\bar \beta): \| p_{\fu}(\vv h) \| \le \bar \beta e^{-\ell} \right\}, \\
  \label{eq:def-RH}
  R^{A_0 H}(\ell, \bar \beta) &:= \left\{ h = \exp(\vv h) : \vv h \in \fr^{A_0 H}(\ell, \bar \beta)  \right\}.
\end{align}
}
\par For $s_1, s_2 \ge 0$, $\bar \beta >0$, let us denote 
\begin{equation}
  \label{eq:def-Q} 
  Q^{s_1}_{ s_2}(\bar \beta) := \{\exp(\vv w): \vv w \in \fr_1 + \fr_2 : \|p_+(\vv w)\| \le \bar \beta e^{-s_1}, \|p_- (\vv w)\| \le \bar \beta e^{-s_2} \},\end{equation}
  and $Q_s(\bar \beta) := Q_s^s(\bar \beta)$.
\par Let $x_{0} \in X$ be such that (2) in Theorem \ref{thm:main-thm} does not hold.
 Let us start with the normalized measure on $\cE = a(s) u([1])x_{0}$ where $s = (1 -\delta_1) t >0$. 
This is a $U$-orbit of length $e^{s}$. Let us define $\cF \subset \cE$ as follows:
\begin{defn}
\label{def:initial-construction}
\par  
$\cF \subset \cE$ is defined by removing from $\cE$ points $y_0 = a ( s) u(r)x_{0} \in \cE$ 
satisfying $ a_0(\ell') a(-\ell) y_0 \not\in X_{\beta^{1/4}}$
for some $0 \le \ell \le 2\delta_2 t$ and $ - 2\delta_1 t \le \ell' \le 2\delta_1 t$. Let $\mu_{\cF}$ denote the normalized $U$-orbit measure on $\cF$.
\end{defn}
\par By Proposition \ref{prop:quantitative-nondivergence} and our assumption on $x_{\mathrm{initial}}$, it is easy to show that 
the removed proportion is $O(s^2  \beta^{ \alpha_2})$. 
Therefore,
\begin{equation}
  \label{eq:decomposition-measure}
   \mu_{\cE} =  \mu_{\cF} + O(s^2 \beta^{ \alpha_2}), 
  \end{equation}
where $\mu_{\cE}$ denotes the normalized $U$-invariant measure on $\cE$.

\subsection{Initial dimension bound}
\label{subsec-initialize}
\par We shall prove the following proposition:

\begin{prop}
  \label{prop:initial-bound-dimension}
  For any $\ell \in [\delta_3 t, \delta_1 t]$, and any $x \in \cF$, we have 
  \[ \mu_{\cF} (Q_\ell(\beta) B^{A_0 H}(\beta) x) \le e^{- d_0 \ell}. \]  
\end{prop}

\begin{proof}
  \par Let $\eta_1 = 100 d_0$. Let us divide $a(s)u([1])x_0$ into pieces of length $e^{\eta_1 \ell}$. 
  \par It suffices to show that for each piece $\tilde \cE$ from the above partition with $\tilde \cF := \tilde \cE \cap \cF \neq \emptyset$, and any $x \in \tilde \cF$,
  \[ |\tilde \cF \cap Q_\ell(\beta) B^{A_0 H}(\beta) x|  \le e^{- d_0 \ell} e^{\eta_1 \ell}. \]
  \par For a contradiction, let us assume that the statement does not hold for some $x\in \tilde \cF$. Then we can find at least 
  $e^{\eta_2 \ell}$ (where $\eta_2 = \eta_1/4$) many $\bar x \in \tilde \cF \cap Q_\ell(\beta) B^{A_0 H}(\beta) x$, $r_{\bar x} \in [1]$ such that 
  \[ u(r_{\bar x} e^{2 d_0 \ell}) \bar x = \exp (\vv v_{\bar x}) a_0^{\bar x} h_{\bar x} \bar x, \]
  where $a_0^{\bar x} \in B^{A_0}(\beta)$, $h_{\bar x} \in B^H(\beta)$, and $\vv v_{\bar x} \in B_{\fr_1 + \fr_2}(\beta e^{-\ell})$.  
  Since $\bar x$ satisfies $\bar x \in X_{\beta^{1/4}}$ and $a(-2 d_0 \ell) \bar x \in X_{\beta^{1/4}}$, 
  by Lemma \ref{lm:separation-lemma-a}, if we write $h_{\bar x} = \exp(\vv h_{\bar x})$, and 
  \[ \vv h_{\bar x} =  a_0 \fa + a_1 \fu + a_2 \fu^\ast,\]
  we will have  
  \begin{equation}
  \label{eq:a0ge}
  |a_0| \ge \beta e^{-2 d_0 \ell}\end{equation}
  \par Let us write $\bar x = u_{\bar x} x$ where $u_{\bar x} = u(r e^{\eta_1 \ell})$ for some $r \in [1]$.
   Let us fix a representative $g \in G$ of $x$, then we have 
   \[ u(r_{\bar x} e^{ d_0 \ell}) u_{\bar x} g = \exp (\vv v_{\bar x} ) a^{\bar x}_0 h_{\bar x} u_{\bar x} g \gamma_{\bar x}, \]
  for some $\gamma_{\bar x} \in \Gamma$. 
  The above equality is equivalent to 
  \begin{equation}
    \label{eq:bound-2}
    g \gamma_{\bar x} g^{-1} =  (a_0^{\bar x} h_{\bar x} u_{\bar x})^{-1} \exp(-\vv v_{\bar x})  u(r_{\bar x} e^{2 d_0 \ell}) u_{\bar x}. 
  \end{equation}
  
  \par Without loss of generality, we can assume that for different $\bar x_1$ and $\bar x_2$ as above, 
  $|u^{-1}_{\bar x_1} u_{\bar x_2}| \ge e^{5  d_0 \ell}$. 
  \par For each $\bar x$ we get $\gamma_{\bar x} \in \Gamma$. We claim that those $\gamma_{\bar x}$'s are different. 
  In fact, if there exist $\bar x, \bar x'$ such that $\gamma_{\bar x} = \gamma_{\bar x'}$, then 
  we have 
  \begin{align*} & (a_0^{\bar x} h_{\bar x} u_{\bar x})^{-1} \exp(-\vv v_{\bar x}) u(r_{\bar x} e^{2 d_0 \ell}) u_{\bar x} \\ 
     = & (a_0^{\bar x'} h_{\bar x'} u_{\bar x'})^{-1} \exp(-\vv v_{\bar x'}) u(r_{\bar x'} e^{2d_0 \ell}) u_{\bar x'}. \end{align*}
    This implies that 
    \begin{align*} &  \exp(\vv v'_{\bar x}) (a_0^{\bar x})^{-1} u_{\bar x}^{-1} h_{\bar x}^{-1} u_{\bar x} u(r_{\bar x} e^{2d_0 \ell}) \\ 
      = &  \exp(\vv v'_{\bar x'}) (a_0^{\bar x'})^{-1} u_{\bar x'}^{-1} h_{\bar x'}^{-1} u_{\bar x'} u(r_{\bar x'} e^{2d_0 \ell}),
    \end{align*}
    where $\vv v'_{\bar x} = - \Ad((a_0^{\bar x} h_{\bar x} u_{\bar x})^{-1})\vv v_{\bar x}$ and $\vv v'_{\bar x'}$
    denotes the same expression with $\bar x$ replaced by $\bar x'$. 
    By comparing the $\fr_1 + \fr_2$, $A_0$ and $H$ components of both sides, we 
    have $\exp(\vv v'_{\bar x}) = \exp(\vv v'_{\bar x'})$, $a_0^{\bar x} = a_0^{\bar x'}$ and 
    \[ u_{\bar x}^{-1} h_{\bar x}^{-1} u_{\bar x} u(r_{\bar x} e^{2 d_0 \ell}) = u_{\bar x'}^{-1} h_{\bar x'}^{-1} u_{\bar x'} u(r_{\bar x'} e^{2 d_0 \ell}).\]
    This implies that 
    \begin{equation} 
      \label{eq:bound-3}
      u(-L) h_{\bar x} u(L) = u(r e^{2 d_0 \ell}) h_{\bar x'}, \end{equation}
    where $u(L) = u_{\bar x} u^{-1}_{\bar x'}$ and $r = r_{\bar x} - r'_{\bar x'}$. Note that $|r|\le 1$ and $|L| \ge e^{10 d_0 \ell}$.
    $|r| \le 1$ implies that the norm of the right hand side is $\le e^{2 d_0 \ell}$. 
    On the other hand, the left hand side is equal to
    \[ \exp(\Ad(u(-L))\vv h_{\bar x} ) = \exp (\Ad(u(-L)) (a_0 \fa + a_1 \fu + a_2 \fu^\ast) ) \]
    whose $\fu$ coordinate is $a_1  + a_0 L + a_2 L^2$. Since $|a_2| \ge \beta e^{-2 d_0 \ell}$ (by \eqref{eq:a0ge}) and $|L| \ge e^{5 d_0 \ell}$, 
    we have that
    \[ |a_1  + a_0 L + a_2 L^2| \ge \beta e^{8 d_0 \ell}. \]
    Therefore, we have that the equality \eqref{eq:bound-3} is impossible to hold and conclude the claim.
  
  \par Let us consider the adjoint action of $G$ on $\bigoplus_{j=1}^8 \wedge^j \fg$ and denote 
  $$v_H = \fa \wedge \fu \wedge \fu^{\ast}.$$
  Then the stabilizer of $v_H$
  \[\Stab(v_H) = A_0 H.\]
  Then \eqref{eq:bound-2} implies that 
  \[ \gamma_{\bar x} g^{-1} v_H =  \exp(\Ad ( g^{-1})\vv v'_{\bar x})  g^{-1} v_H. \]
  The norm of $g^{-1} v_H$ is $\le \|g^{-1}\|^{300}$. 
  Let us estimate the norm of $\vv v'_{\bar x}$. Note that $\|u_{\bar x}\| \le e^{\eta_1 \ell}$, we have that the norm of 
  $(a_0^{\bar x} h_{\bar x} u_{\bar x})^{-1}$ is bounded by $2 e^{\eta_1 \ell}$. Therefore, $\|\vv v'_{\bar x}\| \le e^{-\ell + \eta_1 \ell}$.
  Let us denote $v_0 = g^{-1} v_H/ \| g^{-1} v_H\|$, then we have 
  \[\| \gamma_{\bar x} v_0 - v_0\| \le \|\vv v'_{\bar x}\| \le e^{- (1-\eta_1) \ell}. \]
 \par Considering the group generated by $\gamma_{\bar x}$'s, we have the following two cases:
 \par \textbf{Case 1}: $\langle \gamma_{\bar x} \rangle$ is abelian.
 \par \textbf{Case 2}: $\langle \gamma_{\bar x} \rangle$ is not abelian.
 \par For \textbf{Case 1}, we claim that there exists $\gamma_{\bar x}$ whose unipotent part is not trivial. 
 In fact, if every $\gamma_{\bar x}$ is diagonizable, then they belong to a maximal torus in $G$. Note that $\|\gamma_{\bar x}\| \le e^{\eta_1 \ell}$, 
 we have that there are at most $\ell^{100}$ different $\gamma_{\bar x}$'s which contradicts 
 the fact that there are $e^{\eta_2 \ell}$ different $\gamma_{\bar x}$'s. Note that $\gamma_{\bar x} \in \SL(3,\Z)$, 
 if $\gamma_{\bar x}$ has nontrivial unipotent part, it must be unipotent. 
 By repeating the same argument as in \cite{einsiedler-margulis-venkatesh},
 we can find $g' \in G$ satisfying that $\gamma_{\bar x} g'^{-1} v_H = g'^{-1} v_H$ and 
 \[ \|g' - g\| \le \| \gamma_{\bar x} v_0 - v_0\| \|\gamma_{\bar x}\|^{10} \le e^{- (1-\eta_1)\ell} e^{10 \eta_1 \ell} \le e^{-\ell/2}.  \]
 
 Then $g' \gamma_{\bar x} g'^{-1} \in A_1 H$. Since $\gamma_{\bar x}$ is unipotent, we have $g' \gamma_{\bar x} g'^{-1} \in  H$.
 We claim that the lattice $g' \Z^3$ contains a nonzero vector $\vv p $ 
 satisfying $\|\vv p\| \le e^{2 \eta_1 \ell}$ and $\vv p \in \R \vv e_2 + \R \vv e_3$. In fact, we can find a basis 
 $\{\vv p_1, \vv p_2, \vv p_3\}$ of $g'\Z^3$ such that $\|\vv p_j\| \le \|g'\| \le 2 \|g\| \le 2 \beta^{-1}$ for $j=1,2,3$.
  Since $g' \gamma_{\bar x} g'^{-1}$ is a unipotent element in $H$, its fixing vectors 
  \[ V(g' \gamma_{\bar x} g'^{-1}) := \{\vv p \in \R^3 : g' \gamma_{\bar x} g'^{-1} \vv p = \vv p \} \]
  has dimension $2$. Therefore, $g' \gamma_{\bar x} g'^{-1} \vv p_j \neq \vv p_j$ for some $j \in \{1,2,3\}$. On the other hand, we have $g' \gamma_{\bar x} g'^{-1} \vv p_j \in g' \Z^3$. Thus, 
  \[ g' \gamma_{\bar x} g'^{-1} \vv p_j - \vv p_j \in g' \Z^3. \]
  Moreover, since $g' \gamma_{\bar x} g'^{-1} \in H$, it is easy to see that $\vv p := g' \gamma_{\bar x} g'^{-1} \vv p_j - \vv p_j \in \R \vv e_2 + \R \vv e_3$. 
  Its norm is bounded by 
  \begin{align*} 
  \|\vv p\| &\le \|\vv p_j\| + \|g' \gamma_{\bar x} g'^{-1}\| \|\vv p_j\| \\ 
   &\le \|\vv p_j\| + \|g'\| \|\gamma_{\bar x}\| \|{g'}^{-1}\| \|\vv p_j\| \le 2\beta^{-1} + e^{\eta_1 \ell} \beta^{-10} \le e^{2 \eta_1 \ell}. \end{align*}
   Then for any $\eta_3 >0$,
  \[ \| a_0( \eta_3 \ell) \vv p \| = e^{- \eta_3 \ell/6} \|\vv p\| \le e^{-(\eta_3/6 - 2 \eta_1)\ell}.\]
  Let us choose $\eta_3 = 24 \eta_1$. Then it is easy to see that 
  \[ \|a_0(\eta_3 \ell) g' - a_0(\eta_3 \ell) g\| \le 1, \]
  and $a_0( \eta_3 \ell) g' \Gamma \not\in X_{e^{-2\eta_1 \ell}}$. 
  This implies that $a_0(\eta_3 \ell) g \Gamma \not\in X_{\beta^{1/4}}$ which contradicts that $g\Gamma = x \in \cF$.
  \par Let us consider \textbf{Case 2}. Let us take $\gamma = \gamma_{\bar x}$ and $\gamma' = \gamma_{\bar x'}$ not commuting. 
  Using the same argument as in \textbf{Case 1} (again using the argument in \cite{einsiedler-margulis-venkatesh}), we can find we can find $g' \in G$ 
  satisfying that $\gamma g'^{-1} v_H = g'^{-1} v_H$, $\gamma' g'^{-1} v_H = g'^{-1} v_H$  and 
  \[ \|g' - g\| \le   e^{-\ell/2}. \]
  This implies that $g' \gamma g'^{-1}, g' \gamma' g'^{-1} \in A_1 H$. Then their commutator 
  \[ g' \gamma \gamma' \gamma^{-1} \gamma'^{-1} g'^{-1} \in H. \]
  Then we can use the same argument as in \textbf{Case 1} to 
  deduce that $a_0(\eta_3 \ell) g\Gamma \not\in X_{\beta^{1/4}}$ leading to the same contradiction.
  \par This completes the proof of Proposition \ref{prop:initial-bound-dimension}.
\end{proof}

\par The following statement tells that the orbit can not be concentrated on a small neighborhood of a 
piece of $\SL(2, \R) \ltimes \R^2$ orbit:
\begin{prop}
  \label{prop:sl2xr2-closing-lemma}
  For any $x \in \cF$, $j = 1,2$, and any $\ell \in [2\delta_3 t, \delta_1 t]$
  \[ \mu_{\cF}(Q(\fr_j, \ell, \beta)B^{A_0 H}(\beta)x) \le e^{- d_0 \ell/2}.\]
 \end{prop}
 \begin{proof}
  \par Let us prove the statement for $j = 1$. The proof for $j=2$ is similar.
  \par For a contradiction, let us assume that there exists $x \in \cF$ such that
  \[ \mu_{\cF}\left(Q(\fr_1, \ell, \beta)B^{A_0 H}(\beta)x\right) > e^{- d_0 \ell/2}. \]
  Let us denote $\tilde x = a_0(\ell) x$ and $\hat \mu = a_0(\ell)_\ast \mu_{\cF}$.
 Noting that 
 $$a_0(\ell) Q(\fr_1, \ell, \beta)B^{A_0 H}(\beta)x = Q_{\ell/2}(\beta) B^{A_0 H}(\beta)\tilde x,$$
 we have that 
 \[ \hat \mu \left(Q_{\tilde \ell}(\beta) B^{A_0 H}(\beta)\tilde x\right) > e^{- d_0 \tilde \ell}, \]
 where $\tilde \ell = \ell/2 \in [\delta_3 t, \delta_1 t]$. By repeating the argument in the proof of 
 Proposition \ref{prop:initial-bound-dimension}, we can show that 
 $a_0(\eta_3 \tilde \ell) \tilde x = a_{0}(\ell +\eta_3 \tilde \ell)x \not\in X_{\beta^{1/4}}$, 
 which contradicts to that $x \in \cF$. This completes the proof.

 \end{proof}
{
\section{Basic Reduction}
\label{sec-reduction}
\par In this section we introduce some notation and make a basic reduction of the original problem. 
\par Let us first introduce some definitions:
\begin{defn}
  \label{def:dimension-control}
  Given a probability measure $\mu$ on $X$, we say that $\mu$ is $(d, \ell)$-good if for any 
  $x_0 \in X$,
  \[\mu (Q_\ell(\beta) B^{A_0 H}(\beta) x_0) \le e^{- d \ell}.\]
  We say that $\mu$ is $(d, [\ell_1, \ell_2])$-good if it is $(d, \ell)$-good for any $\ell \in [\ell_1 , \ell_2]$.
\end{defn}
\par According to this definition, Proposition \ref{prop:initial-bound-dimension} says that
 $\mu_{\cF}$ is $( d_0, [\delta_3 t, \delta_1 t])$-good. 
\par Starting from $s_0 = \delta_1 t/2$ and $d_0$, let us define $\{(d_i, s_i): i \in \N\}$ 
 by $s_{i+1} = s_i/2$ and $d_{i+1} = d_i + \epsilon_2$.
\par Starting from $\cF_0 = \cF$, we will 
construct a sequence $\{\cF_i : i \in \N\}$ such that every 
$\cF_{i+1}$ is constructed by removing a small proportion from $a(s_i) \cF_{i}$.
\par In the following several sections, we will prove that $\mu_{\cF_{i+1}}$ has better dimension control than $\mu_{\cF_i}$. Similarly to Definition \ref{def:initial-construction},
by removing an exponentially small proportion, we can assume that for any $x \in \supp \mu_{\cF_i}$, any $0 \le \ell \le 2 \delta_2 t$ 
and $ -2 \delta_1 t \le \ell' \le 2 \delta_1 t$, $a_0(\ell')a(-\ell)x \in X_{\beta^{1/4}}$. By repeating the proof of Proposition 
\ref{prop:initial-bound-dimension} and \ref{prop:sl2xr2-closing-lemma}, we have that Proposition \ref{prop:initial-bound-dimension} and \ref{prop:sl2xr2-closing-lemma} hold for $\mu_{\cF_i}$.
\par For notational simplicity, for $s' >0,  \ell \ge 0$, let us denote 
\begin{equation}
  \label{eq:pi-s-l}
  \Pi_{s', \ell} := a(2s')B^{A_0 A}(\beta) B^{U^\ast}(\beta)\exp\left( B_{\vv r^-}(\beta e^{-\ell}) \right) \exp\left(B_{\vv r^+}(\beta e^{-\ell})\right)  B^{U}(\beta)a(-2s'),
\end{equation}
and
\begin{equation}
  \label{eq:pi-s}
  \Pi_{s'} := \Pi_{s', 0}.
\end{equation}

\par For $x \in X$, let us denote 
\[ \Pi(x) := \Pi_{s'} x. \]

\par Given that $\mu_{\cF_i}$ is $(d_i, s_i)$-good, our aim is to prove that $\mu = a(s_i)_\ast \mu_{\cF_i}$ is 
$(d_i + \epsilon_2, s_{i+1})$-good after removing an exponentially small proportion. 
\par We have the following reduction of the original task:
\begin{prop}
\label{prop:reduction}
For $x \in X$, let us call a neighborhood of the form $\Pi(x) = \Pi_{s'} x$ a playground. A playground $\Pi(x)$ is called a full playground if 
\[ \mu(\Pi(x)) \ge \beta^{15}. \]
For any full playground $\Pi(x)$, let us denote 
\[ \mu_{\Pi(x)} := \mu(\Pi(x))^{-1} \mu|_{\Pi(x)}. \]
If for any full playground $\Pi(x)$, there exists 
\[ \cR_x \subset \Pi(x) \]
such that 
$$\mu_{\Pi(x)} (\cR_x) \le \beta^{20},$$ 
and after removing $\cR_x$, $\mu_{\Pi(x)}$ is $(d_i+2\epsilon_2, s_{i+1})$-good. Then there exists $\cR_{\diamond} \subset X$ such that 
\[ \mu(\cR_{\diamond}) \le \beta^2, \]
and after removing $\cR_{\diamond}$, $\mu$ is $(d_i + \epsilon_2, s_{i+1})$-good.
\end{prop}
\begin{proof}
\par Note that there exists a finite subset $\fQ \subset X$ with $|\fQ| \le \beta^{-9}$ and 
\[ X = \bigcup_{x \in \fQ}  \Pi(x). \]
Let $\fQ_1 \subset \fQ$ denote the set of $x \in \fQ$ such that $\Pi(x)$ is a full playground, and 
$\fQ_2 = \fQ \setminus \fQ_1$. Then 
\[ \mu \left( \bigcup_{x \in \fQ_2}  \Pi(x) \right) \le \beta^{-9} \times \beta^{15} \le \beta^{6}.  \]
By the given condition, for every $x \in \fQ_1$, we get $\cR_x \subset \Pi(x)$ with 
\[ \mu_{\Pi(x)} (\cR_x) \le \beta^{20}, \]
and after removing $\cR_x$, $\mu_{\Pi(x)}$ is $(d_i + 2 \epsilon_2, s_{i+1})$-good. Then 
\[ \mu\left( \bigcup_{x \in \fQ_1} \cR_x\right) \le \beta^{-9} \times \beta^{20} = \beta^{11}. \]
Let 
\[ \cR_{\diamond} = \bigcup_{x \in \fQ_1} \cR_x \cup \bigcup_{x \in \fQ_2} \Pi(x). \]
Then 
\[ \mu(\cR_{\diamond}) \le \beta^{6} + \beta^{11} \le \beta^2. \]
Let $\mu'$ denote the measure constructed by removing $\cR_{\diamond}$ from $\mu$. Let us prove that $\mu'$ is 
$(d_i + \epsilon_2, s_{i+1})$-good. In fact, for any $x' \in X$, let us denote 
\[ Q(x') = Q_{s_{i+1}}(\beta) B^{A_0 H}(\beta) x'. \]
Then we have
\[ \mu'( Q(x')) \le \sum_{x \in \fQ} \mu'_{\Pi(x)} (Q(x')),  \]
where $\mu'_{\Pi(x)}$ denotes the measure constructed by removing $\cR_{\diamond}$ from $\mu_{\Pi(x)}$. Note that for any $x \in \fQ$, $\mu'_{\Pi(x)}$ is $(d_i + 2 \epsilon_2, s_{i+1})$-good, we have 
\[ \mu'_{\Pi(x)} (Q(x')) \le e^{-(d_i + 2 \epsilon_2) s_{i+1}}. \]
Thus, 
\[ \mu'( Q(x')) \le \beta^{-9} \times e^{-(d_i + 2 \epsilon_2) s_{i+1}} \le e^{-(d_i + \epsilon_2) s_{i+1}}.  \]
This completes the proof.
\end{proof}
\par Let us denote $\epsilon_1 = 2 \epsilon_2$. By Proposition \ref{prop:reduction}, it suffices to prove that for any full playground $\Pi(x)$, the measure 
$\mu_{\Pi(x)}$ is $(d_i+ \epsilon_1, s_{i+1})$-good after removing a $\beta^{20}$-small proportion. Therefore, in the following several sections, we will focus on a fixed full playground $\Pi(x)$ and study the measure $\tilde \mu = \mu_{\Pi(x)}$ on $\Pi(x)$.

\section{Construction of a Kakeya-type model}
\label{sec-kakeya-model}
\par From now on we will fix a full playground $\Pi(x)$ and denote $\tilde \mu = \mu_{\Pi(x)}$.
\par In this section we will introduce a Kakeya-type model which enables us to study the dimension of the orbit measure.

\par Let us introduce some notation.
\par Let us denote
\begin{equation} \label{eq:omega-neighborhood}\Omega_{s'} := Q^{s'}_{3s'}(2\beta) Q^{A_0 H}(2s', 2\beta) ,\end{equation}
and
\begin{equation}\label{eq:theta-neighborhood}
\Theta_{s'} :=  B^{A_0 A}(\beta) B^{U^\ast}(\beta e^{-2s'})\exp\left( B_{\vv r^-}(\beta e^{-s'}) \right)\exp\left( B_{\vv r^+}(\beta e^{-s'}) \right) B^U(\beta).\end{equation}
For $s' >0$, let us denote
\begin{equation}
  \label{eq:sigma-neighborhood}
  \Sigma_{s'} :=  B^{A_0 A}(\beta) B^{U^\ast}(\beta e^{-2s'}) B_{\vv r^-}(\beta e^{-s'})B_{\vv r^+}(\beta e^{s'}) B^U(\beta).
\end{equation}
For $x \in X$, let us denote
\[ \Sigma(x) := \Sigma_{s'} x,\]
\[ \Theta(x) := \Theta_{s'} x,  \]
and 
\[ \Omega(x) := \Omega_{s'} x. \]
\par We will need the following lemmata:
\begin{lemma}
    \label{lm:two-theta}
\par Let us denote 
\[ \hat{\Theta}_{s'} := Q_{s'}(1.1\beta) Q^{A_0 H}(2s', 1.1\beta). \]
Then 
\[ \Theta_{s'} \subset \hat{\Theta}_{s'}. \]
\end{lemma}
\begin{proof}
\par Let us take
$$g = a_0 a u^\ast (r^\ast) \exp(\vv w^-) \exp(\vv w^+) u(r) \in \Theta_{s'}.$$
By Lemma \ref{lm:product-w}, 
\[\exp(\vv w^-) \exp(\vv w^+) = \exp(\bar{\vv w} ) \exp( \bar{\vv h}), \]
where $\bar{\vv w} \in B_{\fr_1 + \fr_2}(1.01\beta e^{-s'})$ and $\bar{\vv h} \in B_{\R \fa_0 + \fh} (0.01\beta e^{-2s'})$.
Therefore, we have 
\[ g = a_0 a u^\ast (r^\ast) \exp(\bar{\vv w}) \exp(\bar{\vv h}) u(r).\]
By rearranging the equation above, we have that $g \in \hat{\Theta}_{s'}$.
\end{proof}
\begin{lemma}
\label{lm:product-theta}
\par Let us denote
    \[ \bar{\Theta}_{s'} := Q_{s'}(2.21\beta) Q^{A_0 H}(2s', 2.21\beta). \]
    Then for any $\epsilon_i \in \{\pm 1\}$ for $i=1,2$,
    \[ \Theta^{\epsilon_1}_{s'}\Theta^{\epsilon_2}_{s'}  \subset \bar{\Theta}_{s'}. \]
\end{lemma}
\begin{proof}
\par We will only prove 
\[ \Theta_{s'}^{-1} \Theta_{s'}  \subset \bar{\Theta}_{s'}. \]
Other cases can be proved similarly.
\par By Lemma \ref{lm:two-theta}, it suffices to show that 
\[ \hat{\Theta}_{s'}^{-1} \hat{\Theta}_{s'}  \subset \bar{\Theta}_{s'},\]
where $\hat{\Theta}_{s'}$ is defined as in Lemma \ref{lm:two-theta}.
\par Let us take $g = \exp(\vv w) \exp(\vv h) \in \hat{\Theta}_{s'}$ and $g' = \exp(\vv w') \exp(\vv h') \in \hat{\Theta}_{s'}$. Then
\[ g^{-1} g' = \exp(-\vv h) \exp(-\vv w) \exp(\vv w') \exp(\vv h'). \]
By Lemma \ref{lm:product-w}, we have 
\[ \exp(-\vv w) \exp(\vv w') = \exp(\bar{\vv w}) \exp (\bar{\vv h}), \]
where $\bar{\vv w} \in B_{\fr_1 + \fr_2}(2.201\beta e^{-s'})$, and $\bar{\vv h} \in B_{\R \fa_0 + \fh}(0.001\beta e^{-2s'})$.
 Thus, we have 
\[ g^{-1} g' = \exp(-\vv h) \exp(\bar{\vv w}) \exp(\bar{\vv h}) \exp(\vv h'). \]
By rearranging the equation above, we conclude that $g^{-1} g'  \in \bar{\Theta}_{s'}$. This completes the proof.
\end{proof}
Using Lemma \ref{lm:product-theta}, it is easy to prove the following:
\begin{cor}
    \label{cor:product-theta}
    \par Let us denote 
    \[ \tilde{\Theta}_{s'} := Q_{s'}(4.5\beta) Q^{A_0 H}(2s', 4.5\beta). \]
    Then for any $\epsilon_i \in \{\pm 1\}$ for $i=1,2,3,4$,
    \[ \Theta^{\epsilon_1}_{s'}\Theta^{\epsilon_2}_{s'} \Theta^{\epsilon_3}_{s'} \Theta^{\epsilon_4}_{s'} \subset \tilde{\Theta}_{s'}. \]
\end{cor}
\begin{proof}
\par We will only prove 
\[\Theta_{s'}\Theta_{s'} \Theta_{s'} \Theta_{s'} \subset \tilde{\Theta}_{s'}. \]
Other cases can be proved similarly. In fact, by Lemma \ref{lm:product-theta}, we have 
\[ \Theta_{s'}\Theta_{s'} \subset \bar{\Theta}_{s'}, \]
where $\bar{\Theta}_{s'}$ is defined as in Lemma \ref{lm:product-theta}. By repeating the proof of Lemma \ref{lm:product-theta}, we get 
\[ \bar{\Theta}_{s'} \bar{\Theta}_{s'} \subset \tilde{\Theta}_{s'}. \]
Combining the two inclusions above, we conclude that 
\[\Theta_{s'}\Theta_{s'} \Theta_{s'} \Theta_{s'} \subset \tilde{\Theta}_{s'}.  \]
\end{proof}
\par We will need the following lemma concerning different coordinate systems in $G$:
\begin{lemma}
    \label{lm:different-coordinate-1}
    For $\vv w^+ \in \vv r^+$ and $ \vv w^- \in \vv r^-$ with $\|\vv w^+\|, \|\vv w^-\| \le 10 \beta e^{-\ell}$, we have that
    \[ \exp(\vv w^+ + \vv w^-) =  a_0 a u^{\ast} (r^\ast) \exp(\bar{\vv w}^-) \exp(\bar{\vv w}^+) u(r), \]
    where $\bar{\vv w}^+ \in \vv r^+$ with $\|\bar{\vv w}^+ - \vv w^+\| \le 0.01\beta e^{-3\ell}$, $\bar{\vv w}^- \in \vv r^-$ with $\|\bar{\vv w}^- - \vv w^-\| \le 0.01\beta e^{-3\ell}$, $a_0 a \in B^{A_0 A}(0.01\beta e^{-2\ell})$, 
    and $r^\ast, r \in [0.01\beta e^{-2\ell}]$.
    \par Similarly, we have 
    \[ \exp(\vv w^+ + \vv w^-) =  a_0 a u (r) \exp(\bar{\vv w}^+) \exp(\bar{\vv w}^-) u^\ast(r^\ast), \]
    where $\bar{\vv w}^+ \in \vv r^+$ with $\|\bar{\vv w}^+ - \vv w^+\| \le 0.01\beta e^{-3\ell}$, $\bar{\vv w}^- \in \vv r^-$ with $\|\bar{\vv w}^- - \vv w^-\| \le 0.01\beta e^{-3\ell}$, $a_0 a \in B^{A_0 A}(0.01\beta e^{-2\ell})$, 
    and $r^\ast, r \in [0.01\beta e^{-2\ell}]$.
\end{lemma}
\begin{proof}
\par We only prove the first statement. The second statement can be proved using the same argument.
\par Let us write 
\[ \exp(\vv w^+ + \vv w^-) =  a_0 a u^{\ast} (r^\ast) \exp(\bar{\vv w}^-) \exp(\bar{\vv w}^+) u(r), \]
where $a_0 a \in B^{A_0A} (10\beta e^{-\ell})$, $r^\ast, r \in [10\beta e^{-\ell}]$, 
$\bar{\vv w}^- \in B_{\vv r^-}(10 \beta e^{-\ell})$, and $\bar{\vv w}^+ \in B_{\vv r^+}(10 \beta e^{-\ell})$.
Noting that 
\[ \|[\bar{\vv w}^-, \bar{\vv w}^+]\| \le \|\bar{\vv w}^-\| \cdot \|\bar{\vv w}^+\| \le 0.001 \beta e^{-2\ell}, \]
by Lemma \ref{lm:product-w}, we have 
\[ \exp(\bar{\vv w}^-) \exp(\bar{\vv w}^+) = \exp(\tilde{\vv w}^- + \tilde{\vv w}^+) \exp( \vv h), \]
where $\tilde{\vv w}^- \in \vv r^-$ with $\| \tilde{\vv w}^- - \bar{\vv w}^-\| \le 0.002\beta e^{-3\ell}$, $\tilde{\vv w}^+ \in \vv r^+$ with $\| \tilde{\vv w}^+ - \bar{\vv w}^+\| \le 0.002\beta e^{-3\ell}$, and 
$\vv h \in B_{\R \fa_0 +\fh}(0.002\beta e^{-2\ell})$. Thus, we get 
\begin{align*} 
\exp(\vv w^+ + \vv w^-) &=  a_0 a u^{\ast} (r^\ast) \exp(\bar{\vv w}^-) \exp(\bar{\vv w}^+) u(r) \\
                        &=  a_0 a u^{\ast} (r^\ast) \exp( \tilde{\vv w}^- + \tilde{\vv w}^+ ) \exp(\vv h) u(r) \\
                        &= \exp( \Ad_{h} (\tilde{\vv w}^- + \tilde{\vv w}^+) ) a_0 a u^{\ast} \exp(\vv h) u(r),
\end{align*}
where $h = a_0 a u^{\ast} (r^\ast)$. Then we have 
\[ \vv w^+ + \vv w^- = \Ad_{h} (\tilde{\vv w}^- + \tilde{\vv w}^+),\]
and 
\[ a_0 a u^{\ast}(r^\ast) \exp(\vv h) u(r) = \id. \]
Noting that $\|\vv h\| \le 0.003 \beta e^{-2\ell} $, we get $a_0 a \in B^{A_0 A}(0.01\beta e^{-2\ell})$, $r^\ast , r \in [0.01\beta e^{-2\ell}]$. Note that 
\[ \| \Ad_{h} (\tilde{\vv w}^- + \tilde{\vv w}^+) - (\tilde{\vv w}^- + \tilde{\vv w}^+) \| \le 0.001\beta e^{-3\ell}. \]
Combining the estimates from above, we get 
\[ \|\vv w^+ - \bar{\vv w}^+\|, \|\vv w^- - \bar{\vv w}^-\| \le 0.01\beta e^{-3\ell}. \]
This completes the proof.
\end{proof}
For $r \in [\beta e^{2s'}]$, let us denote 
\begin{equation} \Sigma^r = \Sigma(u(r)x) \subset \Pi(x). \end{equation}
 $\Sigma^r$ can be regarded as a cross-section of $\Pi(x)$ along the $U$-orbits in the following sense: 
 \begin{lemma}
  \label{lm:cross-section}
  For $y \in \Pi(x)$ and $r \in [\beta e^{2s'}]$, let 
  \[J_r(y) :=\left\{r' \in \left[2\beta e^{2s'}\right]: u(r')y \in \Sigma^r\right\}. \]
  Then there exists intervals $J_1 \subset J_r(y) \subset J_2$ with size 
  $ 0.5 \beta \le |J_1| \le |J_2| \le 4 \beta$.
 \end{lemma}
 \par We will postpone the proof after Definition \ref{def:kakeya-model}.
\par For $\bar x \in \Sigma^r \subset  \Pi(x)$, let us denote 
\[ \Pi_{s', \ell} (\bar x) = \Pi_{s', \ell} \bar x,  \]
and 
\[\Sigma_{s', \ell}(\bar x) = \Pi_{s',\ell} (\bar x) \cap \Sigma^r.\]
Then $\Sigma_{s', \ell} (\bar x)$ can also be regarded as a cross-section of $\Pi_{s', \ell} (\bar x)$ along the $U$-orbits. 
\par We first prove the following lemma:
\begin{lemma}
  \label{lm:initial-dimension-under-a-t}
   For any $\delta_3 t \le \ell \le \delta_1 t$ and any $\bar x \in \Pi(x)$, we have
  \[ \tilde \mu(\Sigma_{s', \ell} (\bar x)) \le e^{-2s'} e^{-d_0 \ell}.\]
\end{lemma}
\begin{proof}
  \par Note that $\Sigma_{s', \ell} (\bar x)$ is a cross-section of $\Pi_{s', \ell}(\bar x)$ along the $U$-orbits, 
  and $\tilde \mu$ is Lebesgue along the $U$-orbits.
   We have 
  \[ \tilde \mu(\Sigma_{s', \ell} (\bar x) ) \ll e^{-2s'} \tilde \mu (\Pi_{s', \ell} (\bar x) ).\]
  Note that by the definition of $\tilde \mu$, we have
  \[ \tilde \mu (\Pi_{s', \ell} (\bar x)) \le \mu_{\cF_i} (a(-2s')\Pi_{s', \ell} (\bar x) ), \]
  and 
  \[ a(-2s')\Pi_{s',  \ell} (\bar x) = B^{A_0 A}(\beta) B^{U^\ast}(\beta) \exp\left( B_{\vv r^-}(\beta e^{-\ell}) \right) \exp\left( B_{\vv r^+}(\beta e^{-\ell}) \right) B^{U}(\beta) x' \]
  where $x' = a(-2s') \bar x$. By Lemma \ref{lm:different-coordinate-1}, we have that 
  \[ B^{A_0 A}(\beta) B^{U^\ast}(\beta) \exp\left( B_{\vv r^-}(\beta e^{-\ell}) \right) \exp\left( B_{\vv r^+}(\beta e^{-\ell}) \right) B^{U}(\beta) \subset  Q_\ell(2\beta) B^{A_0 H }(2\beta).\]
  By repeating the proof of Proposition \ref{prop:initial-bound-dimension} to $\cF_i$ we have 
  \[ \mu_{\cF_i}(Q_\ell(2\beta) B^{A_0 H }(2\beta) x') \le e^{-d_0 \ell}, \]
  which concludes the proof.
\end{proof}

\par We then introduce the Kakeya-type model to analyze the behavior of $U$-orbits.

\begin{defn}
  \label{def:kakeya-model}
 \par Let $s'>0$ and $\tilde \mu$ be a probability measure on a playground $\Pi(x) \subset X$ defined by taking the normalized 
 measure on a subset consisting of $U$-orbits of length $\beta e^{2s'}$.  
 Let us fix a sequence of cross-sections 
 $$\left\{\Sigma^r = \Sigma_{s'} u(r)x: r \in \left[\beta e^{2s'}\right]\cap\left(0.1\beta \Z\right) \right\} .$$
 \par We first define the projection map 
 \[ \bP: \Pi(x) \to \R^3 \]
 as follows: For $x' \in \Pi(x)$, we can write
 \[x' =a_{0, x'} a_{x'} u^\ast(r^\ast_{x'} ) \exp(\vv w^-_{x'}) \exp(\vv w^+_{x'}) u(r_{x'} ) x\]
 where $a_{0, x'} \in B^{A_0}(\beta)$, $a_{x'} \in B^A(\beta)$, $\vv w^-_{x'} =  v_{2, x'} \vv v_2 + w_{2, x'} \vv w_2 \in \vv r^- $ 
 and $\vv w^+_{x'} = v_{1, x'} \vv v_1 + w_{1, x'} \vv w_1 \in \vv r^+$, $ r^{\ast}_{x'} \in [\beta e^{-2s'}]$, and $r_{x'}\in [\beta e^{2s'}]$.
 Then 
 $$\bP(x') := (r_{x'}, v_{1,x'}, w_{1, x'}).$$
 Then 
 $$\bP(\Pi(x)) = \left[\beta e^{2s'}\right] \times \left[\beta e^{s'}\right]^2 \subset \R^3.$$ 
 It is easy to see that for $r \in \left[\beta e^{2s'}\right]$, 
 \[ \bP(\Sigma^r) = (r, 0, 0) + [\beta] \times \left[\beta e^{s'}\right]^2. \]
 \par Let $\cI$ be a finite subset of $\Pi(x)$ defined as follows:
 \begin{equation} 
 \label{eq:cI}
 \cI := \left\{ \exp(v_1 \fd \vv v_1 + w_1 \fd \vv w_1 ) u(r \fd') x: r, v_1, w_1 \in \left[10 e^{2s'}\right]\cap \Z  \right\},\end{equation}
 where $\fd = 0.1 \beta e^{-s'}, \fd' = 0.1 \beta$. For $z \in \cI$, let $\Theta'(z) := \Theta(z) \cap \Pi(x)$.
 
 For every $z \in \cI$, let $\cI_z$ be a finite subset of $\Theta'(z)$ defined as follows:
 \begin{equation}
 \label{eq:cIz}
 \cI_z := \left\{ \exp(v_2 \fd'' \vv v_2 + w_2 \fd'' \vv w_2 ) z : v_2, w_2 \in \left[10 e^{2s'}\right]\cap \Z \right\},
 \end{equation}
 where $\fd'' = 0.1 \beta e^{-3s'}$. For $y \in \cI_z$, let $\Omega'(y):= \Omega(y) \cap \Theta'(z)$. Let us denote 
 $$\fJ := \bigcup_{z \in \cI} \cI_z.$$
 \par For any $y \in \Pi(x)$, let us define 
 \begin{equation}
 \label{eq:fly}
 \fl(y) := \left\{ r \in \left[2\beta e^{2s'}\right]: u(r) y \in \Pi(x) \right\}.
 \end{equation}
 \par For each $z \in \cI$ and $y \in \cI_z$, let us write 
 \[ y = \exp\left(\vv w^-_y\right) \exp\left(\vv w^+_y\right) u\left(r_y \right) x,\]
 where  $$\vv w^-_y =  v_{2, y} \vv v_2 + w_{2, y} \vv w_2 \in B_{\vv r^-}\left(\beta e^{-s'}\right),$$ 
 $$\vv w^+_y = v_{1, y} \vv v_1 + w_{1, y} \vv w_1 \in  B_{\vv r^+}\left(\beta e^{s'}\right),$$ 
 and $r_y \in \left[\beta e^{2s'}\right]$. 
 Let us define the curve $\cL(y)$ as follows:
 \[ \cL(y) := \left\{ \left(f_y(r), f_{v, y}(r), f_{w, y}(r) \right) = \bP\left(u(r) y\right) : r \in \fl(y) \right\}. \]
 It is straightforward to check that 
 \begin{equation}
  \label{eq:fvy}
  f_{v, y}(r) = \frac{2r v_{2, y}}{2 + r v_{2, y} w_{2, y}} + v_{1, y}, 
 \end{equation}
 \begin{equation}
  \label{eq:fwy}
  f_{w, y}(r) = \frac{-2r w_{2, y}}{2 - r v_{2, y} w_{2, y}} + w_{1, y},
 \end{equation}
 and 
 \begin{equation}
  \label{eq:fy}
  \begin{array}{rl}
  f_y(r)  & = r_y + \frac{2r + 2r w_{1,y} v_{2, y}}{2 + r v_{2, y} w_{2, y}}  + \frac{w_{1,y}v_{1,y} - f_{v, y}(r)f_{w, y}(r)}{2} \\
          &= r_y + r +  \frac{r w_{2,y} v_{1,y}}{2 - r w_{2,y} v_{2,y}} + \frac{r w_{1,y} v_{2,y}}{2 + r w_{2,y} v_{2,y}} + \frac{ r (r w_{2,y} v_{2,y})^2 }{4 - (r w_{2,y} v_{2,y})^2}
  \end{array}
 \end{equation}
 \par From \eqref{eq:fy} it is easy to see that $|f_{y} (r) - r_y - r| \le \beta |r|$.
 \par Let $\cT(y)$ denote the $2\beta e^{-s'}$-neighborhood of $\cL(y)$. We then assign to $\cT(y)$ a weight 
 \[\cW(y) = \tilde{\mu}(\Omega'(y)).\]
 \par For $y_1, y_2 \in \Pi(x)$, let us define $\Delta(y_1, y_2)$ as follows:
Let us write
 \[ y_2 = a a_0 u^\ast(r^\ast) \exp(\vv w^-) \exp(\vv w^+) u(r) y_1,\]
where $a a_0 \in B^{A A_0} (2\beta)$, $r^\ast \in \left[2\beta e^{-2s'}\right]$, $\vv w^- = v_2 \vv v_2 + w_2 \vv w_2 \in B_{\vv r^-} \left(2\beta e^{-s'}\right)$,
and $\vv w^+ = v_1 \vv v_1 + w_1 \vv w_1 \in B_{\vv r^+} \left(2\beta e^{s'}\right)$. Then we define
 \begin{equation}
  \label{eq:delta-y_1-y_2}
  \Delta(y_1, y_2) = (v_2, w_2 ). 
 \end{equation} 
 Let us call the collection of tubes $\{\cT(y): y \in \fJ\}$ with weights $\{\cW(y)\}$ the Kakeya-type model for $(\tilde \mu, \Pi(x))$.
\end{defn}

\par Now let us prove Lemma \ref{lm:cross-section}:
\begin{proof}[Proof of Lemma \ref{lm:cross-section}]
  \par Without loss of generality, we can assume $r=0$. We first prove $J_0(y) \neq \emptyset$. For any $y \in \Pi(x)$, let us write 
  \[ y = a_y a_{0, y} u^\ast(r^\ast_y) \exp(\vv w_y^-) \exp(\vv w_y^+) u(r_y) x \]
  where $a_y a_{0,y} \in B^{A A_0} (\beta)$, $r_y^\ast \in \left[\beta e^{-2s'}\right]$, 
  $$\vv w_y^- = v_{2,y} \vv v_2 + w_{2,y} \vv w_2 \in B_{\vv r^-} \left(\beta e^{-s'}\right),$$
  $$\vv w_y^+ = v_{1,y} \vv v_1 + w_{1,y} \vv w_1 \in B_{\vv r^+} \left(\beta e^{s'}\right),$$
  and $r_y \in \left[\beta e^{2s'}\right]$. Then for any $r \in \left[\beta e^{2s'}\right]$,
  \[ u(r) y = a'_y a'_{0, y} u^{\ast} (\tilde r^\ast_y) u(\tilde r) \exp(\vv w_y^-) \exp(\vv w_y^+) u(r_y) x,\]
  where $a'_y a'_{0,y} \in B^{A A_0}(\beta)$, $ \tilde r^{\ast}_{y} \in  \left[\beta e^{-2s'}\right]$ 
  and $\tilde r = \tilde r(r)$ is a smooth function of $r$ satisfies that $|r-\tilde r| \le \beta|r|$. Moreover, by \eqref{eq:fy}, we have 
  \[ u(\tilde r) \exp(\vv w_y^-) \exp(\vv w_y^+) u(r_y) x = a_{0,y,\tilde r} a_{y,\tilde r} u^\ast (r^\ast_{y,\tilde r}) \exp(\vv w^-_y(\tilde r)) \exp(\vv w^+_y(\tilde r)) u(f_y(\tilde r)) x,  \]
  where $|f_y(\tilde r) - r_y - \tilde r| \le \beta |\tilde r|$. This implies that 
  $$|f_y(\tilde r) - r_y - r| \le 2 \beta |r|.$$
  To prove that there exists $r$ such that $u(r) y \in \Sigma(x)$, it suffices to show that there exists $r$ 
  such that $f_y(\tilde r) = 0$. Since $f_y(\tilde r)$ depends on $r$ continuously, it is enough to show that there exist $r_1, r_2$ 
  such that $f_y(\tilde r(r_1)) < 0$ and $f_y(\tilde r(r_2)) >0$. By choosing $r$ large (and small, respectively) enough, we can make 
  \[ r+ r_y > 2\beta|r| (\text{ and } r+ r_y < -2\beta |r|, \text{ respectively}),\] 
  which implies that $f_y(\tilde r(r)) >0$ (and $f_y(\tilde r(r)) <0$, respectively). This shows that there exists $r$ such that $u(r)y \in \Sigma(x)$. 
  \par Now fix $r_0$ such that 
  \[ z = u(r_0)y = a_z a_{0, z} u^\ast \left(r^\ast_z\right) \exp\left(\vv w_z^-\right) \exp\left(\vv w_z^+\right) u(0) x. \]
  We first prove that $r_0 + [\beta/2] \subset J_0(y)$, namely, for $r \in \left[\beta/2\right]$, 
  \[ u(r) z = a_{z,r} a_{0, z, r} u^\ast\left(r^\ast_{z,r}\right) \exp\left(\vv w_z^-(r)\right) \exp\left(\vv w_z^+(r)\right) u\left(f_z(r)\right) x, \]
  with $|f_z(r)| \le \beta$. In fact, the above argument shows that $|f_z(r) - r| \le \beta |r|$. Therefore, for $|r| \le  \beta/4$,
  we have $|f_z(r)| \le (1+\beta)|r| \le \beta$. This shows that $r_0 +[\beta/2] \subset J_0(y)$.
 \par Similar argument shows that if $|f_z(r)| \le \beta$, then $|r| \le (1-\beta)^{-1} |f_z(r)| \le 2\beta$. 
 This shows that $J_0(y) \subset r_0 + [4\beta]$.
 \par This completes the proof.
\end{proof}
\par We will need the following lemmata:

\begin{lemma}
\label{lm:partition-1}
We have 
\[ \Pi(x) = \bigcup_{z \in \cI} \Theta'(z). \]
Moreover, for any $z' \in \Pi(x)$, there are at most $N_2$ different $z \in \cI$ such that 
\[ \Theta(z') \cap \Theta(z) \neq \emptyset. \]
\end{lemma}
\begin{proof}
\par We first prove that for any $z' \in \Pi(x)$, there exists $z \in \cI$ such that $z' \in \Theta(z)$. In fact, let us write 
\[ z' = a_{z'} a_{0, z'} u^{\ast} (r^\ast_{z'}) \exp(\vv w^-_{z'}) \exp(v_{1, z'} \vv v_1 + w_{1, z'} \vv w_1) u(r_{z'}) x, \]
where $a_{z'} \in B^A(\beta)$, $a_{0, z'} \in B^{A_0}(\beta)$, $r^\ast_{z'} \in \left[\beta e^{-2s'}\right]$, $\vv w^-_{z'} \in B_{\vv r^-}\left(\beta e^{-s'}\right)$, and
$$(r_{z'} , v_{1, z'}, w_{1, z'}) \in \left[\beta e^{2s'}\right] \times \left[\beta e^{s'}\right]^2.$$
Let us choose $z \in \cI$ such that
\[ z =  \exp(v_1 \vv v_1 + w_1 \vv w_1 ) u(r_z) x \]
such that $|r_z - r_{z'}| \le 0.2 \beta$, and 
\[ \max \{ |v_{1, z} - v_1|, |w_{1,z} - w_1| \} \le 0.2 \beta e^{-s'}. \]
We claim that $z' \in \Theta(z)$. In fact, by direct calculation, we have
\[ z' = a_{z'} a_{0, z'} u^{\ast} (r^\ast_{z'}) \exp(\vv w^-_{z'}) \exp(v'_1 \vv v_1 + w'_1 \vv w_1) u(r') z, \]
where $v'_1 = v_{1, z'} - v_1$, $w'_1 = w_{1, z'} - w_1$, and $r' = r_{z'} - r_z - \frac{v'_1 w_1 - v_1 w'_1}{2}$.
Note that 
$$\max \left\{ |v'_1|, |w'_1| \right\} = \max \left\{ | v_{1, z'} - v_1|, |w_{1, z'} - w_1| \right\} \le 0.2 \beta e^{-s'},$$
and 
\[ |r'| \le |r_{z'} - r_z| + |v'_1 w_1| + |v_1 w'_1| \le 0.1 \beta + 2 \times 0.2 \beta e^{-s'} \times \beta e^{s'} \le 0.3\beta. \]
Thus, we have $z' \in \Theta(z)$.
\par We then prove that there are at most $N_2$ different $z \in \cI$ such that 
\[ \Theta(z') \cap \Theta(z) \neq \emptyset.  \]
In fact, let us take
\[ \tilde z \in \Theta(z') \cap \Theta(z). \]
Then we have 
\[ z' \in  \Theta^{-1}_{s'} \tilde z \subset \Theta^{-1}_{s'} \Theta_{s'}  z.  \]
By Lemma \ref{lm:product-theta}, we have 
\[ z' \in \bar{\Theta}_{s'}z, \]
namely,
\[ z' =  \exp(\bar{\vv w}) \exp(\bar{\vv h}) z, \]
where $\bar{\vv w} \in B_{\fr_1 + \fr_2}(2.21\beta e^{-s'})$ and $\bar{\vv h} \in \fq^{A_0 H}(2s', 2.21\beta)$. By rearranging the equation above, we get 
\[ z' =  \exp(b \fa + b_0 \fa_0 + r^\ast \fu^\ast )\exp(\vv w'') u(r) z, \]
where $b, b_0, r \in [2.22\beta]$, $r^\ast \in [2.22\beta e^{-2s'}]$, and $\vv w'' \in B_{\fr_1 + \fr_2} (2.22\beta e^{-s'})$. 
By Lemma \ref{lm:different-coordinate-1}, we have 
\[ \exp(\vv w'' ) = \exp(\tilde b \fa + \tilde b_0 \fa_0 + \tilde r^\ast \fu^\ast) \exp(\vv w^-) \exp(\vv w^+) u(\tilde r), \]
where $\vv w^- \in B_{\vv r^-}(2.3 \beta e^{-s'})$, $\vv w^+ \in B_{\vv r^+}(2.3 \beta e^{-s'})$, $\tilde b, \tilde b_0 , \tilde r^\ast, \tilde r \in [0.1 \beta e^{-2s'}]$. Thus, 
\begin{equation}
\label{eq:z-z'}
z' = a_0 a u^\ast (\bar r^\ast) \exp(\vv w^-) \exp(\vv w^+) u(\bar r) z, \end{equation}
where $a_0 a \in B^{A_0 A}(3.1 \beta)$, $\bar r^\ast \in [3.1 \beta e^{-2s'}]$, $\bar r \in [3.1 \beta]$, $\vv w^- \in B_{\vv r^-}(2.3 \beta e^{-s'})$, and $\vv w^+ \in B_{\vv r^+}(2.3 \beta e^{-s'})$. Let us consider $\bP(z)$ and $\bP(z')$. From \eqref{eq:z-z'}, we have that
\[ \bP(z) \in \bP(z') + \left[3.1\beta\right] \times \left[2.3 \beta e^{-s'}\right]^2. \]
From the definition of $\cI$, we have that there are at most $N_2$ different $z \in \cI$ satisfying the condition above.

\end{proof}
\begin{lemma}
    \label{lm:partition-2}
For any $z \in \cI$,
\[ \Theta'(z) = \bigcup_{y \in \cI_z} \Omega'(y). \]
Moreover, for any $y' \in \Theta(z)$, there are at most $N_2$ different $y \in \cI_z$ such that 
\[ \Omega(y) \cap \Omega(y') \neq \emptyset. \]
\end{lemma}
\par Before proving Lemma \ref{lm:partition-2}, let us introduce a coordinate system in $\Theta(z)$:
\begin{defn}
\label{def:coordinate-theta}
\par  Given $z \in \cI$ and $ y \in \Theta'(z)$, let us define the coordinate of $y$ in $\Theta'(z)$, denoted by $\fD_z(y) \in \R^2$ as follows: Let us write
\[ y = a_0 a u(r_y) \exp(\vv w^+) \exp( \vv w^-)u^\ast (r^\ast_y) z, \]
where $a_0 a \in A_0 A$, $\vv w^+ \in \vv r^+$ and $\vv w^- = v_2 \vv v_2 + w_2 \vv w_2 \in \vv r^-$. Then we define 
\[ \fD_z(y) := (v_2, w_2). \]
\end{defn}
\par Let us prove Lemma \ref{lm:partition-2}:
\begin{proof}[Proof of Lemma \ref{lm:partition-2}]
\par We first prove that for any $z' \in \Theta'(z)$, there exists $y \in \cI_z$ such that $z' \in \Omega(y)$. By Lemma \ref{lm:two-theta}, $\Theta_{s'} \subset \hat{\Theta}_{s'}$. Therefore, we can write
\[ z' =  a_0 a u(r) \exp(\vv w^+ + \vv w^-)u^\ast (r^\ast)  z, \]
where $a_0 a \in B^{A_0 A} (1.1 \beta)$, $r \in [1.1\beta]$, $r^\ast \in [1.1 \beta e^{-2s'}]$, $\vv w^+ \in B_{\vv r^+}(1.1 \beta e^{-s'})$ and $\vv w^- = v_2 \vv v_2 + w_2 \vv w_2 \in B_{\vv r^-}(1.1\beta e^{-s'})$. 

 Let $y \in \cI_z$ such that 
\[ y = \exp(\bar{\vv w}^-) z, \]
with $\| \bar{\vv w}^- - \vv w^- \|\le 0.2 \beta e^{-3s'}$. 
\par We will prove that $z' \in \Omega(y)$. In fact, we have 
\begin{equation}
\label{eq:z'-y}
z' = a_0 a u(r) \exp(\vv w^+ + \vv w^-) \exp(-\bar{\vv w}^-) u^\ast (r^\ast)y. \end{equation}
By Lemma \ref{lm:product-w}, we have that 
\[  \exp(\vv w^+ + \vv w^-) \exp(-\bar{\vv w}^-) = \exp(\vv w^+ + \vv w^- - \bar{\vv w}^- + \tilde{\vv w} ) \exp(\tilde{\vv h} ), \]
where $ \tilde{\vv w} \in B_{\fr_1 + \fr_2}(0.1\beta e^{-3s'})$, and $\tilde{\vv h} \in B_{\R \fa_0 + \fh}(0.1\beta e^{-2s'})$. 
Let us denote 
\[ \vv w =  \vv w^+ + \vv w^- - \bar{\vv w}^- + \tilde{\vv w} \in \fr_1 + \fr_2,\]
and write
\[\exp(\vv w^+ + \vv w^-) \exp(-\bar{\vv w}^-) =  \exp(\vv w ) \exp( \tilde{\vv h}). \]
 Then the above estimates imply that $\|p_{-} (\vv w)\| \le \beta e^{-3s'}$. By plugging the above equation into \eqref{eq:z'-y}, we get that $z' \in \Omega(y)$. 
\par Let us prove that there are at most $N_2$ different $y \in \cI_z$ such that 
\[ \Omega(z') \cap \Omega(y) \neq \emptyset. \]
Let us take 
\[ y' \in \Omega(z') \cap \Omega(y). \]
Then 
\[ y' = \exp(\vv w^+ + \vv w^-) \exp(\vv h) y = \exp(\tilde{\vv w}^+ + \tilde{\vv w}^-) \exp(\vv h') z', \]
where $$\vv w^+, \tilde{\vv w}^+ \in B_{\vv r^+}\left(2\beta e^{-s'}\right),$$ 
$$\vv w^-, \tilde{\vv w}^- \in B_{\vv r^-}\left(2\beta e^{-3s'}\right),$$ 
and 
$$\vv h, \vv h' \in \fq^{A_0 H}\left(2s', 2\beta\right).$$
Then
\[ z' =  \exp(-\vv h') \exp( - \tilde{\vv w}^+ - \tilde{\vv w}^-) \exp(\vv w^+ + \vv w^-) \exp(\vv h) y.\]
By Lemma \ref{lm:product-w}, 
\[ \exp( - \tilde{\vv w}^+ - \tilde{\vv w}^-) \exp(\vv w^+ + \vv w^-) = \exp(\tilde{\vv w}^+ + \tilde{\vv w}^- ) \exp( \tilde{\vv h}), \]
where $\tilde{\vv w}^+ \in B_{\vv r^+}(4.1 \beta e^{-s'})$, $\tilde{\vv w}^- \in B_{\vv r^-}(4.1 \beta e^{-3s'})$, 
$\tilde{\vv h} \in B_{\R \fa_0 + \fh} (0.1 \beta e^{-2s'})$. Thus,
\[ z' = \exp(-\vv h') \exp(\tilde{\vv w}^+ + \tilde{\vv w}^-) \exp(\tilde{\vv h}) \exp(\vv h)y. \]
By rearranging the above equation, we have 
\[
z' = a_0 a u(r) \exp( \vv v^+ + \vv v^- ) u^{\ast}(r^\ast) y, \]
where $a_0 a \in B^{A_0 A}(2.3\beta)$, $r \in [2.3 \beta]$, $r^\ast \in [2.3 \beta e^{-2s'}]$, $\vv v^+ \in B_{\vv r^+} (4.3\beta e^{-s'})$, and $\vv v^-\in B_{\vv r^-} (4.3\beta e^{-3s'})$. By Lemma \ref{lm:different-coordinate-1}, 
\[  \exp( \vv v^+ + \vv v^- ) = \tilde{a}_0 \tilde{a} u(\tilde r) \exp(\tilde{\vv v}^+)\exp(\tilde{\vv v}^-) u^\ast(\tilde{r}^\ast), \]
where $\tilde{a}_0 \tilde{a} \in B^{A_0 A}(0.1 \beta e^{-2s'})$, 
$\tilde r, \tilde{r}^\ast \in [0.1\beta e^{-2s'}]$, $\tilde{\vv v}^+ \in B_{\vv r^+}(4.4 \beta e^{-s'})$, 
and $\tilde{\vv v}^- \in B_{\vv r^-}(4.4 \beta e^{-3s'})$. This implies that 
\[ z' = a_0 a u(r)\tilde{a}_0 \tilde{a} u(\tilde r) \exp(\tilde{\vv v}^+)\exp(\tilde{\vv v}^-) u^\ast(\tilde{r}^\ast) u^{\ast}(r^\ast) y. \]
Considering the coordinates of $z'$ and $y$ in $\Theta(z)$, we have 
\[ \left\| \fD_z(z') - \fD_z(y) \right\| = \|\vv v^-\| \le 4.4 \beta e^{-3s'}. \]
This shows that there are at most $N_2$ different $y \in \cI_z$ satisfying this condition. 
\par This completes the proof.
\end{proof}
\begin{remark}
    \label{rmk:partition-2}
    \par Let us denote 
    \begin{equation}
    \label{eq:fC}
    \fC := \left[1.2 \beta e^{-s'}\right]^2.\end{equation}
    From Lemma \ref{lm:partition-2}, we can see that for any $z' \in \Theta(z)$, 
    \[ \fD_z(z') \in \fC. \]
\end{remark}
\begin{lemma}
  \label{lm:wv2r}
  \par Suppose 
  \[ y_2 = \exp(\vv w^-) \exp(\vv w^+) u(r_0) y_1,\]
  where $|r_0| \le \beta e^{2s'}$, $\vv w^- = v_2 \vv v_2 + w_2 \vv w_2 \in B_{\vv r^-} (\beta e^{-s'})$, and $\vv w^+ = v_1 \vv v_1 + w_1 \vv w_1 \in B_{\vv r^+} (\beta e^{s'})$.
  For $r \in [\beta e^{2s'}]$, let us write $u(r) y_2$ as 
  \[ u(r) y_2 = a a_0 u^\ast \exp(\vv w^-(r)) \exp(\vv w^+(r)) u(f(r)) y_1, \]
  where $a \in A$, $a_0 \in A_0$, $u^\ast \in U^\ast$, $\vv w^+ (r) \in \vv r^+$, and 
  \[ \vv w^-(r) = v_2(r) \vv v_2 + w_2 (r) \vv w_2 \in \vv r^-. \]
  Then 
  \begin{equation}
    \label{eq:v2r}
    v_2(r) = \frac{2 v_2}{2 - r w_2 v_2},
  \end{equation} 
  and 
  \begin{equation}
    \label{eq:w2r}
    w_2(r) = \frac{2 w_2}{2 + r w_2 v_2}.
  \end{equation} 
  In particular, for any $r' \in [\beta e^{2s'}]$, if we denote $y'_1 = u(r') y_1$, $y'_2 = u(r) y_2$, we have 
  \[ \|\Delta(y'_1, y'_2)\| \asymp \|\Delta(y_1, y_2)\|. \]
\end{lemma}
\begin{proof}
  This can be verified by direct calculation.
\end{proof}

\par The following lemma tells that $\Omega(y)$ is invariant along the $U$-orbits, modulo necessary time change:
\begin{lemma}
\label{lm:omega-u}
Let $\tilde \mu$ and $\Pi(x)$ be as in Definition \ref{def:kakeya-model}.  Let us denote 
  \begin{equation} 
  \label{eq:hat-omega}
  \hat{\Omega}_{s'} := Q^{s'}_{3s'}(2.1\beta) Q^{A_0 H} (2s', 2.1\beta). \end{equation}
 For $y \in \fJ$, let us denote 
 \begin{equation} 
 \label{eq:u-omega}
 U\Omega'(y) : = \left\{ u(r') y' : y' \in \Omega'(y), r' \in \fl(y') \right\}. \end{equation}
 For $r \in \fl(y)$, let us denote 
 \begin{equation} 
 \label{eq:omega-y-r}
 \Omega(y, r) := U\Omega'(y) \cap \Sigma^{\tilde{r}}, \end{equation}
 where $\tilde r$ satisfies that $u(r)y \in \Sigma^{\tilde{r}}$. Then 
 \[ \Omega(y, r) \subset \hat{\Omega}(u(r)y) := \hat{\Omega}_{s'} u(r) y. \]
 In particular, we have
 \[ \tilde \mu(\Omega'(y)) \ll \tilde \mu \left(\hat{\Omega} (u(r) y)\right). \]
\end{lemma}
\begin{proof}
  \par For $y' \in \Omega'(y)$, we have 
  \[ y = \exp(\vv w^+ + \vv w^-) \exp(\vv h) y', \]
  where $\vv w^+ \in B_{\vv r^+}(2.01\beta e^{-s'})$, $\vv w^- \in B_{\vv r^-}(2.01\beta e^{-3s'})$, and 
  $\vv h \in \fq^{A_0 H}(2s', 2\beta)$. Then
  \begin{align*}
      u(r) y &= u(r) \exp(\vv w^+ + \vv w^-) \exp(\vv h) y' \\ 
             &= \exp(\Ad_{u(r)} (\vv w^+ + \vv w^-)) u(r)\exp(\vv h) y' \\
             &= \exp(\vv w^+ + \tilde{\vv w}^+ + \vv w^-) \exp(\tilde{\vv h}) u(r') y',
  \end{align*} 
where $\tilde{\vv w}^+ \in \vv r^+$ with $\|\tilde{\vv w}^+\| = |r| \|\vv w^-\|$, and $\tilde{\vv h} \in \fq^{A_0 H}(2s', 2.01\beta)$. Note that $|r| \le \beta e^{2s'}$, and $\|\vv w^-\| \le 2.01 \beta e^{-3s'}$. We have $\|\tilde{\vv w}^+\| \le 0.01 \beta e^{-s'}$. This implies that 
\[ u(r') y' \in \hat{\Omega}(u(r) y).  \]
Also note that $u(r') y'$ and $u(r) y$ are in the same cross-section $\Sigma^{\tilde{r}}$. By Lemma \ref{lm:cross-section}, $u(\fl(y'))y' \cap \Sigma^{\tilde{r}}$ is contained in an interval of length $\asymp \beta$. This implies that for any $y' \in \Omega(y)$,
\[ u(\fl(y'))y' \cap \Sigma^{\tilde{r}} \subset \hat{\Omega}(u(r) y). \]
Thus, we have 
\[ \Omega(y, r) \subset \hat{\Omega}(u(r) y). \]
Note that $\Omega'(y)$ and $\Omega(y, r)$ are both cross-sections of $U\Omega'(y)$ along $U$-direction. Thus, we have
\[ \tilde \mu(\Omega(y,r)) \asymp \tilde \mu(\Omega'(y)) \asymp e^{-2s'} \tilde \mu(U\Omega'(y)). \]
This implies that 
\[ \tilde \mu(\Omega(y)) \ll \tilde \mu\left(\hat{\Omega}(u(r) y)\right).\]

\end{proof}
\begin{lemma}
  \label{lm:time-leave}
  \par Let $s' >0$ and $\Pi(x) $ be as in Definition \ref{def:kakeya-model}. Let us fix a cross-section $\Sigma^r \subset \Pi(x)$ defined as above. Then for any $y_1, y_2 \in \fJ \cap \Sigma^r$, suppose there exist $L_1, L_2$ and $z \in \cI$ such that 
  \[ \Omega(y_i, L_i) \cap \Theta'(z) \neq \emptyset, \text{ for } i=1,2, \]
  where $\Omega(y_i, L_i)$ is defined as in \eqref{eq:omega-y-r}.
  \par  Let us denote $\tilde{y}_i = u(L_i) y_i$ and write 
  \[ \tilde{y}_1 = a_0 a u^\ast(r^\ast) \exp(\vv w^-) \exp(\vv w^+) u(r_0) \tilde{y}_2,\]
  where $a_0 a \in A_0 A$,
  $$\vv w^- = v_2 \vv v_2 + w_2 \vv w_2 \in \vv r^-,$$
  and 
  $$\vv w^+ = v_1 \vv v_1 + w_1 \vv w_1 \in \vv r^+.$$
  Then $\vv w^+ \in B_{\vv r^+}(4.6 \beta e^{-s'})$. Moreover, let 
  \begin{equation}
  \label{eq:Iy1y2}
  I (y_1, y_2) := \left\{ r_1 \in \fl(y_1) : \exists r_2 , z' \in \cI,  \text{ s.t. }  \Omega (y_i, r_i) \cap \Theta'(z')\neq \emptyset, i=1,2 \right\}.\end{equation}
  Then there exists an interval $I$, such that $ I(y_1, y_2) \subset I$  with 
  $$|I| \asymp \min\left\{ \beta e^{2s'}, \beta e^{-s'} \|\vv w^-\|^{-1}\right\}.$$
\end{lemma}
\begin{proof}
\par By Lemma \ref{lm:omega-u}, we have $\Omega(y_i, L_i) \subset \hat{\Omega}(\tilde{y}_i)$. Therefore, we have $ \hat{\Omega}(\tilde{y}_1) \cap \Theta(z) \neq \emptyset$ which implies that 
\[ \tilde{y}_1 \in \hat{\Omega}_{s'}^{-1} \Theta_{s'} z \subset \Theta_{s'}^{-1} \Theta_{s'} z.  \]
Similarly, we have $ \hat{\Omega}(\tilde{y}_2) \cap \Theta(z) \neq \emptyset$ implies that 
\[ z \in \Theta_{s'}^{-1} \Omega_{s'} \tilde{y}_2 \subset \Theta_{s'}^{-1} \Theta_{s'} \tilde{y}_2.  \]
Therefore, 
\[ \tilde{y}_1 \in \Theta_{s'}^{-1} \Theta_{s'}\Theta_{s'}^{-1} \Theta_{s'} \tilde{y}_2. \]
By Corollary \ref{cor:product-theta},
\[ \tilde{y}_1 \in  \tilde{\Theta}_{s'} \tilde{y}_2.\]
Then we have
\[ \tilde{y}_1 = \tilde{a}_0 \tilde{a} u^{\ast}(\tilde{r}^\ast) \exp(\tilde{\vv w}^+ + \tilde{\vv w}^- ) u(\tilde{r}) \tilde{y}_2,  \]
where $\vv w^+ \in B_{\vv r^+}(4.5 \beta e^{-s'})$. By Lemma \ref{lm:different-coordinate-1}, if we write 
\[ \tilde{y}_1 = a_0 a u^\ast(r^\ast) \exp(\vv w^-) \exp(\vv w^+) u(r_0) \tilde{y}_2, \]
we will have $\|\vv w^+ - \tilde{\vv w}^+\| \le 0.01\beta e^{-3s'}$. This implies that $\vv w^+ \in B_{\vv r^+}(4.6 \beta e^{-s'})$.

  \par For any $r_1 \in \left[2\beta e^{2s'}\right]$, we have that 
  \begin{align*} 
    u(r_1) \tilde{y}_1 &=u(r_1) a a_0 u^\ast(r^\ast) \exp(\vv w^-) \exp(\vv w^+) u(r_0) \tilde{y}_2 \\
                &= a' a'_0 u^{\ast}_1 u(r'_1) \exp(\vv w^-) \exp(\vv w^+) u(r_0) \tilde{y}_2
  \end{align*} 
  where $a' a'_0 \in B^{A_0 A}(\beta)$, $u_1^{\ast} \in B^{U^\ast}(\beta e^{-2s'})$, and $|r_1 - r'_1| \le \beta |r_1|$.
  By \eqref{eq:fvy}, \eqref{eq:fwy}, and \eqref{eq:fy}, we have that 
  \begin{align*} 
    u(r_1) \tilde{y}_1 &= \bar a \bar a_0 \bar u^{\ast }  \exp(\vv w^-(r'_1)) \exp(\vv w^+(r'_1)) u(f(r'_1)) \tilde{y}_2
  \end{align*} 
  with $\vv w^+ (r'_1) = f_v(r'_1) \vv v_1 + f_w(r'_1) \vv w_1$, where 
  \[ f_v(r'_1) = v_1 + \frac{2v_2 r'_1}{2 + r'_1 v_2 w_2}, \]
  and 
  \[ f_w(r'_1) = w_1 + \frac{-2 w_2 r'_1}{2 - r'_1 v_2 w_2}. \]
By our previous argument we can see that $r_1 \in I(y_1, y_2)$ implies that $\|\vv w^+ (r'_1)\| \le 4.6\beta e^{-s'}$.
\par Since 
\[ \left| \frac{2v_2 r'_1}{2 + r'_1 v_2 w_2} - v_2 r'_1 \right| \le \beta |v_2 r'_1|. \]
If $|f_v(r'_1)| \le 4.6\beta e^{-s'}$, then $|v_2 r'_1| \le 10 \beta e^{-s'}$ 
implying that $|r'_1| \le 10 \beta e^{-s'} |v_2|^{-1}.$
\par By repeating the argument to $f_w(r'_1)$, we have that $|f_w(r'_1)| \le 4.6\beta e^{-s'}$ implies that $|r'_1| \le 10 \beta e^{-s'}|w_2|^{-1}$.
\par This completes the proof.
\end{proof}
\par Recall that in Definition \ref{def:coordinate-theta}, we define a coordinate system in $\Theta'(z)$ for any $z \in \cI$. By Remark \ref{rmk:partition-2}, for any $z' \in \Theta'(z)$, its coordinate $\fD_z(z')$ is contained in 
\[ \fC = \left[1.2 \beta e^{-s'}\right]^2.\]
The following lemma is crucial for us to study the structure of the Kakeya-type model:
\begin{lemma}
  \label{lm:concentrated-surface}
  \par Recall that $I(y_1, y_2)$ is defined as in \eqref{eq:Iy1y2}. Let us fix a cross-section $\Sigma^r$. Let $z \in \cI\cap \Sigma^r, y \in \fJ\cap \Sigma^r$. Then the set of all coordinates $\fD_z (\bar z)$ of the elements $\bar z \in \cI_z$ such that $I(\bar z , y) \neq \emptyset$ and in addition there is some $r \in I(\bar z, y)$ with $|r| \ge \beta^{100} e^{2s'}$ is contained in the $ \beta^{-100} e^{-3s'}$ neighborhood of the hyperbola
  \begin{equation}
  \label{eq:hyperbola}
  \left\{ (X, Y) \in \fC:  (X - X_0)(Y-Y_0) + \frac{X - X_0}{v_1} + \frac{Y - Y_0}{w_1} = 0  \right\}, 
  \end{equation}
  where $v_1, w_1, X_0, Y_0$ are constants depending on $y, z$ with bounds 
  $(X_0, Y_0) \in \fC$ and $(v_1, w_1) \in [\beta e^{s'}]^2$.  The set of such coordinates $\fD_z(\bar z)$ is also contained in the box
  \begin{equation}
  \label{eq:hyperbola-box}
  \fB(\cT(y), z) := (X_0, Y_0) + \left[\beta^{-100} e^{-2s'} \|(v_1, w_1)\|\right]^2.\end{equation}
  Moreover, if there exists $\bar z \in \cI_z$ such that there is some $r \in I(\bar z, y)$ with $|r| \ge \beta^{100} e^{2s'}$ and 
  \[ \|\Delta(\bar z , y)\| \le \beta e^{-s'-j}, \]
  then 
  \[ \|(v_1, w_1)\| \le \beta^2 e^{s' -j},\]
  and the size of $\fB(\cT(y), z)$ is $\le \beta^{-98} e^{s'-j}$.
\end{lemma}

\begin{proof}
\par There exist $y' \in \Omega(y)$ and $z' \in \Theta(z)$ such that 
  \[ z' = \exp(\vv w^+) u(r_{y'}) y',\]
  where $\vv w^+ =  v_1 \vv v_1 + w_1 \vv w_1 $, and 
  \[ z' = \exp(\vv w_0^-) z, \]
  where $\vv w_0^- = X_0 \vv v_2 + Y_0 \vv w_2$. We can also make sure that 
  \[ \min\{|v_1|, |w_1|\} \ge \beta e^{-s'}. \]
  \par  Without loss of generality, let us assume $|w_1| \ge |v_1|$. The opposite case can be handled similarly.
\par Suppose $\bar z = \exp(\vv w^-) z' $ (where $\vv w^- = v_2 \vv v_2 + w_2 \vv w_2$) 
satisfies that $I(\bar z, y) \neq \emptyset$. Then by \eqref{eq:fvy}, \eqref{eq:fwy}, and \eqref{eq:fy}, we have that 
\[ u(r) \bar z = a_{0,r} a_r u^{\ast} (f^\ast(r)) \exp(\vv f_-(r)) \exp(\vv f_+(r)) u(f(r))y',  \]
where $\vv f_+(r) =  f_v(r) \vv v_1 + f_w (r) \vv w_1$, 
\[f_v(r) = \frac{2r v_{2}}{2 + r v_{2} w_{2}} + v_{1},\]
and 
\[ f_w(r)  = \frac{-2r w_{2}}{2 - r v_{2} w_{2}} + w_{1}.\]
By Lemma \ref{lm:time-leave}, $ r \in I(\bar z, y)$ implies that $|f_v(r)|, |f_w(r)| \le 4.6 \beta e^{-s'}$. 
Let us denote $\Delta = v_2 w_2/2$. Then we have 
\begin{equation}
  \label{eq:2}
  r^{-1} + \Delta = - \frac{v_2}{v_1} (1+ O(\beta e^{-s'}v_1^{-1})),\end{equation}
and
\begin{equation} 
  \label{eq:1}
  r^{-1} - \Delta = \frac{w_2}{w_1} ( 1+ O(\beta e^{-s'} w_1^{-1}) ).\end{equation}
Note that $|v_1| \ge \beta e^{-s'}$. So we have 
\[ |\beta e^{-s'} v_1^{-1}| \le 1.\]
Note that 
\begin{equation} 
\label{eq:3}
|r^{-1}| \ge \beta^{-1} e^{-2s'} \ge \beta^2 e^{-2s'} \ge |\Delta|.\end{equation}
Since $|r| \ge \beta^{100}e^{2s'}$, we get 
\begin{equation}
  \label{eq:control-w2}
   |v_2/v_1|, |w_2/ w_1| \le \beta^{-100}e^{-2s'}.\end{equation}
   This implies that 
   $$\|(v_2, w_2)\| \le  \beta^{-100}e^{-2s'} \|(v_1, w_1)\| .$$ 
   Moreover, \eqref{eq:1} and \eqref{eq:2} imply that 
\[ v_2 w_2 + v_1^{-1} v_2 + w_1^{-1} w_2 = O(\beta^{-99} e^{-3s'} v_1^{-1}) + O(\beta^{-99} e^{-3s'} w_1^{-1}).  \]
 This implies that the coordinate of $\bar z$, $\fD_z(\bar z) =(v_2 + X_0, w_2+ Y_0)$ is contained in $\beta^{-100} e^{-3s'}$ neighborhood of the hyperbola 
\[  (X- X_0) (Y- Y_0) + v_1^{-1} (X- X_0) + w_1^{-1} (Y- Y_0) = 0, \]
and also contained in $(X_0, Y_0) + \left[\beta^{-100} e^{-2s'} \|(v_1, w_1)\|\right]^2$.
\par Now let us assume that $\|\Delta(\bar z , y)\| \le e^{-s'-j}$. Then we have 
\[ \bar z =  \exp(\vv w^-) \exp(\vv w^+) u(r_{y'}) y'\]
where $\vv w^- = v_2 \vv v_2 + w_2 \vv w_2 \in B_{\vv r^-}(e^{-s'-j})$. Then by \eqref{eq:2}, \eqref{eq:1}, and \eqref{eq:3}, we have that 
\[ |v_2/v_1|, |w_2/ w_1| \ge \beta^{-1}e^{-2s'}. \]
This implies that 
\[ \|(v_1, w_1)\| \le \beta e^{2s'} \|(v_2, w_2)\| \le \beta^2 e^{s' -j}. \]
By \eqref{eq:hyperbola-box}, we have that the size of $\fB(\cT(y), z)$ is $\le \beta^{-98} e^{s'-j}$.
This completes the proof.
\end{proof}
\begin{remark}
  \label{rmk:concentrated-surface}$\quad$
  
\begin{enumerate}
 \item  Note that $I(y_1, y_2)$ will change if we replace $y_1$ with $y'_1 \in \Omega(y_1)$. However, there exists a one-to-one map 
\[ P_{y_1, y'_1}: I(y_1, y_2) \to I(y'_1, y_2) \]
such that $|P_{y_1, y'_1} (r') - r'| \le \beta |r'|$ for every $r' \in I(y_1,y_2)$.
\item\label{item:line} If $|v_1/w_1| \le  e^{-s'}$ or $|w_1/v_1| \le  e^{-s'}$, then the neighborhood described in~\eqref{eq:hyperbola} agrees with the $ \beta^{-100} e^{-3s'}$ neighborhood of the straight line
  \[  v_1^{-1}(X - X_0) + w_1^{-1}(Y - Y_0) = 0.\] 
  \end{enumerate}
\end{remark}

\medskip

\par Let us denote the $\beta^{-100} e^{-3s'}$-neighborhood of the curve \eqref{eq:hyperbola} by \[\cC(\cT(y), z) \subset \fC.\]
By misuse of notation, we will write $\Omega (\bar z) \subset \cC(\cT(y), z)$ when we actually mean that $ \fD_{ z}(\bar z) \in \cC(\cT(y), z)$.  Also, note that if $\fD_{z} (\bar z) \in \cC(\cT(y), z)$ but 
\[ \fD_z(\bar z ) \not\in \fB(\cT(y), z), \]
then there does not exist $ r' \in I( \bar z, y)$ with $r' \ge \beta^{100} e^{2s'}$.
\section{Dimension Improvement --- unstructured part}
\label{sec-dimension-improvement-unstructured}
\par Recall that our goal is the following proposition:

\begin{prop}
  \label{prop:improve-dimension}
  \par  Suppose that $\cF_i$ is $(d_i, s_i)$-good.
  Let $\mu = a(s_i)_\ast\mu_{\cF_i}$. Let us fix a full playground $\Pi(x)$ and denote $\tilde \mu = \mu_{\Pi(x)}$.
  Then $\tilde \mu$ admits the following decomposition:
  \[ \tilde \mu = \nu_{\cR_i} \tilde \mu_{\cR_i} + \nu_{i+1} \tilde \mu_{i+1}, \]
  such that  $\nu_{\cR_i} \le \beta^{20}$, and $\mu_{i+1}$ is $(d_i+\epsilon_1, s_{i+1})$-good. Recall that $s_{i+1} 
  = s_i/2$.
\end{prop}

\noindent
The proof of this proposition occupies the rest of this section as well as \S \ref{sec-no-loss-of-dimension} and \S \ref{sec-structured-component}.

\medskip

\par Let us define the heavy part of $\Pi(x)$:
\begin{defn}
\label{def:heavy-part}
    \par For $z \in \cI$, we call $\Theta'(z)$ a \emph{heavy neighborhood} if 
\[ \tilde \mu (\Theta'(z)) > e^{-2s'} e^{- (d_i + \epsilon_1) s'}.  \]
$\Theta'(z)$ a \emph{extremely heavy} if 
\[ \tilde \mu (\Theta'(z)) > e^{-2s'} e^{- (d_i -C_1 \epsilon_1) s'}.  \]
\end{defn}
\begin{defn}
\par Let $\cR_i \subset \Pi(x)$ denote the union of heavy neighborhoods. Then $\cR_i$ will be said to be the \emph{heavy part} of $\Pi(x)$.
\end{defn}
\par To prove Proposition \ref{prop:improve-dimension}, it suffices to show that 
\begin{equation}
    \label{eq:heavy-part-is-small}
    \tilde \mu(\cR_i) \le \beta^{20}.
\end{equation}
 In fact, if this is the case, we can put $\tilde \mu_{\cR_i} = \tilde{\mu}^{-1}(\cR_i)\tilde{\mu}|_{ \cR_i}$, and $\tilde \mu_{i+1}=  \tilde \mu^{-1}(\Pi(x)\setminus\cR_i)\tilde \mu_{\fr}|_{\Pi(x)\setminus\cR_i}$. Then since $\Pi(x)\setminus\cR_i$ does not have any heavy neighborhood, we have for any $z \in \cI$
\[ \tilde \mu_{i+1}(\Theta'(z)) \le e^{-2s'} e^{- (d_i + \epsilon_1) s'}. \]
Since every neighborhood of the form $Q_{s'}(\beta) B^{A_0 H}(\beta) y$ contains at most $e^{2s'}$ different $\Theta'(z)$'s, we conclude that $\tilde \mu_{i+1}$ is $(d_i + \epsilon_1, s')$-good, which is what we wanted.
\par For the rest of the proof, we will focus on proving \eqref{eq:heavy-part-is-small}.
\par We first prove the following lemma:
\begin{lemma}
    \label{lm:omega-upper}
    Let us denote $s' = s_{i+1}$.  For any 
 $\Omega(y) = \Omega_{s'} y$, we have 
 \begin{equation} 
\label{eq:upper-omega}
\tilde \mu (\Omega'(y)) \le \beta^{-15}e^{-2 s'} e^{-2 d_i s'}.  \end{equation}
\end{lemma}
\begin{proof}
First note that 
 \[ \mu(\Omega(y))  = \mu_{\cF_i}(a(-2s') \Omega(y)),\]
 and
 \[ a(-2s') \Omega(y) = Q_{2s'}(2\beta) R^{A_0 H}(2s', 2\beta) y' \]
 where $y' = a(-2s') y$.
 Noting that $\mu_{\cF_i}$ is a normalized measure on $U$-orbits, we have 
 \[ \mu_{\cF_i} (Q_{2s'}(2\beta) R^{A_0 H}(2s', 2\beta) y') \le e^{-2s'}  \mu_{\cF_i} (Q_{2s'}(2\beta) B^{A_0 H }(2\beta) y'). \]
 By our hypothesis that $\mu_{\cF_i}$ is $(d_i, 2s')$-good, we have
 \[ \mu_{\cF_i} (Q_{2s'}(2\beta) B^{A_0 H }(2\beta) y') \le e^{-2d_i s'}, \]
 which implies the claim.
\par Since $\Pi(x)$ is a full playground, we have that for any $y \in \fJ$, \eqref{eq:upper-omega} holds.
\end{proof}
\par Let us further decompose $\cR_i$ as follows:
\begin{defn}
    \label{def:heavy-part-decomposition}
    \par For integer $j \ge 1$, let us denote
\[ \Xi_j(z') := \Xi_{s',j} z' \cap \Theta'(z'), \]
where 
\[ \Xi_{s', j} := Q^{s'}_{s' + j}Q^{A_0 H}(2s', 1.1\beta).\]
\par Given a heavy neighborhood $\Theta' (z)$, we call $\Theta'(z)$ \emph{strange} if there exists $ \Xi_{C_1 \epsilon_1 s'}(z') \subset \Theta'(z)$ such that 
\[ \tilde \mu_{\fr}(\Xi_{C_1 \epsilon_1 s'}(z')) \ge (1-\beta) \tilde \mu_{\fr}(\Theta'(z)). \]
Otherwise, we call it a \emph{normal neighborhood}. Let $\bar \cR_i \subset \cR_i$ denote the union of all the heavy neighborhoods that are also normal. Then we have the following decomposition:
\begin{equation}
    \label{eq:strange-normal-decomposition}
    \cR_i = \bar \cR_i \cup \bigcup_{j = C_1 \epsilon_1 s'}^{2s'} \cR_{i,j},
\end{equation}
where $\cR_{i,j}$ denotes the union of heavy neighborhoods $\Theta'(z)$ such that there exists $\Xi_j(z') \subset \Theta’(z)$ satisfying 
\begin{equation} 
\label{eq:strange-nbhd-condition}
\tilde \mu_{\fr}(\Xi_j(z')) \ge (1-\beta)^{j} \tilde \mu_{\fr}(\Theta'(z)), \end{equation}
and 
\begin{equation} 
\label{eq:strange-nbhd-condition-2}
\tilde \mu_{\fr}(\Xi_{j+1}(z'')) < (1-\beta) \tilde \mu_{\fr}(\Xi_j(z')) \end{equation}
for any $\Xi_{j+1}(z'') \subset \Xi_j(z')$. Here $j \le 2s'$ because for $j > 2s'$, 
\[ \Xi_j (z') \subset \Omega(z'), \]
and by \eqref{eq:upper-omega}, 
\[ \tilde \mu (\Omega(z')) \le \beta^{-15}e^{-2 s'} e^{-2 d_i s'}, \]
which implies (since $\Theta'(z)$ is heavy) that \eqref{eq:strange-nbhd-condition} can not hold for $\Xi_j(z')$ and~$\Theta'(z)$.
\end{defn}

\begin{defn}
  \label{def:structure-randomness}
  \par Let $s'>0$, $\Pi(x)$ and $\tilde \mu$ be as in Definition \ref{def:kakeya-model}. To fit our case, let us assume that for any $y \in \fJ$, 
  $$\cW(y) = \tilde \mu(\Omega'(y)) \le e^{-2 d_i s' } e^{-2s'}.$$
  By Lemma \ref{lm:concentrated-surface}, for every $y \in \fJ$ and $z \in \cI$ from the same cross-section $\Sigma^r$, there is a corresponding neighborhood $\cC(\cT(y), z) \subset \fC$.
  \par Let us consider the summation 
  \[ \fS(\cT(y), z) := \sum_{\substack{ y' \in \cI_z \\ \fD_z(y') \in \cC(\cT(y), z)\cap \fB(\cT(y), z) }} \cW(y'). \]
  Let $j$ be so that the size of the sides of $\|(v_1, w_1)\|$ are in $[e^{s'-j}, e^{s'-j+1}]$ (cf.~Lemma~\ref{lm:concentrated-surface}, particularly equation \eqref{eq:hyperbola-box}). Then it is easy to see that
  \begin{equation}
  \label{eq:fScT}
  \fS(\cT(y), z) \le N_2 \beta^{-100} \beta^{-100}e^{2s' - j}  e^{-2 d_i s' } e^{-2s'} \le  \beta^{-201} e^{-2 d_i s' - j},
  \end{equation}
  where $N_2$ is a large numeric constant (in \S\ref{subsec-constants},it was chosen to be $10^6$). The curve neighborhood
  $\cC(\cT(y), z)$ is called \emph{highly concentrated} if 
  $$ \fS(\cT(y), z) \ge  e^{-2 d_i s' -7C_1\epsilon_1 s'}.$$
  
\end{defn}

\begin{remark}
    Note that $\fS(\cT (y), z) $ is essentially the $\tilde\mu$-measure of a subset of $\Theta'(z)$ with $r$ so that $z,y \in \Sigma^r$ (in particular, up to edge effects, $\Theta'(z)$ is a subset of $\Sigma^r$). Thus 
    \[
    \fS(\cT(y),z) \lessapprox \tilde\mu(\Sigma^r) \leq e^{-2s'}.
    \]
    Thus, if $d_i<1$ there can be no highly concentrated curve neighborhoods. Moreover if $\Theta'(z)$ is not extremely heavy, then unless $d_i >2-O(\epsilon_1)$ then $ \cC(\cT(y),z)$ is not highly concentrated.
\end{remark}

\par We further decompose $\bar \cR_i$ as follows:
\begin{defn}
\label{def:structure-random-decompostion}
\par Given $z \in \cI$ such that $\Theta'(z)$ is heavy and normal, and $\Sigma^r$ such that $z \in \Sigma^r$, let us define:
\begin{equation}
   \label{eq:structure-part-theta-0}
   \HS(z,r) := \left\{ y \in \fJ\cap \Sigma^r: \cC(\cT(y), z) \subset \fC \text{ is highly concentrated} \right\},
   \end{equation}
   and 
   \begin{equation}
   \label{eq:structure-set-theta-0}
   \SS(z,r) := \bigcup_{y \in \HS(z, r)} \Omega'(y).
   \end{equation}
   If there exists $\Sigma^r$ containing $z$ such that 
   \[ \tilde \mu (\SS(z,r)) \ge e^{-2s' - 10 C_1 \epsilon_1 s'}, \]
we will call $\Theta'(z)$ \emph{highly structured}. Then we have the following decomposition for $\bar \cR_i$:
\begin{equation} 
\label{eq:normal-random-structure-decomposition}
\bar \cR_i = \bar \cR^{\fr}_{i} \cup \bar \cR^{\fs}_{i},\end{equation}
where $\bar \cR^{\fs}_i$ denotes the union of heavy, normal, and highly structured $\Theta'(z)$, and $\bar \cR^{\fr}_i$ denotes the union of heavy, normal $\Theta'(z)$ which are not highly structured.
\end{defn}
\par We first prove that $\tilde \mu (\bar \cR^{\fr}_{i})$ is small. The proof that  $\tilde \mu (\bar \cR^{\fs}_{i})$ is small is deferred to the next section.
\begin{prop}
    \label{prop:dimension-improvement-random}
    \[ \tilde \mu (\bar \cR^{\fr}_{i}) \le \beta^{22}. \]
\end{prop}
\begin{proof}

\par For a contradiction, let us assume $\tilde \mu(\bar \cR^{\fr}_i) \ge \beta^{22}$. 

\par For every $y \in \bar \cR^{\fr}_i$, if
\[ |u(\fl(y))y \cap \bar \cR^{\fr}_i| \ge \beta^{31} e^{2s'}, \]
we will call it a \emph{good point}. Otherwise, we call it a \emph{bad point}. Then it is easy to see that 
the collection of bad points has $\tilde \mu$-measure $\le \beta^{30}$. 
Therefore, the collection of good points in $\bar \cR^{\fr}_i$, denoted by $\cR'_{i}$, has $\tilde \mu$-measure $\ge \beta^{22.1}$.

\par For $r \in [\beta e^{2s'}]$, we say that the cross-section $\Sigma^r \subset \Pi(x)$ is a \emph{good cross-section} if 
\[ \tilde \mu_{\fr}(\cR'_i \cap \Sigma^r) \ge \beta^{30} \tilde \mu_{\fr}(\Sigma^r).\]
Otherwise, we call it a \emph{bad cross-section}. Note that the $\tilde \mu_{\fr}$-measure of $\cR'_i$ contained in bad cross-sections is $\le \beta^{30}$. 
Thus, under our assumption that ${\tilde\mu}(\mathcal R'_i) \geq \beta^{22.1}$, there must be at least one good cross-section.
\par Now let us fix a good cross-section $\Sigma^r$. For such a cross-section we define $\cS(\Sigma^r)$ as the number of pairs of tubes from it that intersect somewhere in the playground with appropriate weights as follows: 
\[
\cS(\Sigma^r):=\sum_{ (y_1,y_2) \in \diamond_r} \cW(y_1) \cW(y_2)
\]
where $(y_1, y_2) \in \diamond_r$ if
\begin{itemize}
\item $y_1, y_2 \in \fJ\cap \Sigma^r$, and 
\item $\exists r' \in I(y_1, y_2)$ with $|r'| \ge \beta^{100} e^{2s'}$, and
\item $(\Omega'(y_1) \cap \bar \cR^{\fr}_i) \cup (\Omega'(y_2) \cap \bar \cR^{\fr}_i)  \neq \emptyset$.
\end{itemize}
The summands $\cW(y_1) \cW(y_2)$ can be interpreted as the weight given to the intersection $\cT(y_1) \cap \cT(y_2)$.

\par  For every $y \in \fJ \cap \Sigma^r$ such that $\Omega'(y) \cap \cR'_i \neq \emptyset$, we can choose 
$\hat{y} \in \Omega'(y)\cap \cR'_i$. By definition of $\cR'_i$, we have that
\[ |u(\fl(\hat y))\hat y \cap \bar \cR^{\fr}_i | \ge \beta^{31} e^{2s'}. \]
For every $ y' = u(r_{y'}) \hat y \in u(\fl(\hat y))\hat y \cap  \bar \cR^{\fr}_i$ such that $|r_{y'}|\ge \beta^{100} e^{2s'}$, there exists $z \in \cI$ such that $y' \in \Theta'(z)$ 
and 
\[\tilde \mu(\Theta'(z)) \ge e^{-2s'} e^{-(d_i +\epsilon_1)s'}.\]
For every $\tilde{z} \in \cI_z$, let $r_{\tilde{z}}$ be such that $u(r_{\tilde{z}}) \tilde z \in \Sigma^r$. Let us define
\[ \fJ(\tilde z, r) := \left\{ \hat{z} \in \fJ\cap \Sigma^r: \Omega'(\hat z) \cap \Omega(\tilde z, r_{\tilde{z}}) \neq \emptyset \right\}, \]
where $\Omega(\tilde z, r_{\tilde{z}}) $ is as in \eqref{eq:omega-y-r}.
Then for every $\hat z \in \fJ(\tilde z, r)$, we have that $y, \hat z$ satisfies $\diamond_r$. Since 
\[ \Omega(\tilde z, r_{\tilde{z}}) \subset \bigcup_{\hat z \in \fJ(\tilde z, r)} \Omega'(\hat z), \]
and 
\[ \tilde \mu_{\fr}(\Omega'(\tilde z)) \asymp \tilde \mu_{\fr}(\Omega(\tilde z, r_{\tilde{z}})), \]
we have that
\[ \sum_{\hat z \in \fJ(\tilde z, r)} \cW(y) \cW(\hat z) \gg \cW(y) \cW(\tilde z).  \]
\par Let us only consider $\tilde z \in \cI_{z;y'}:=\cI_z\cap( \Theta'(z) \setminus \Xi_{C_1 \epsilon_1 s'}( y'))$.
Then $\tilde z \in \cI_{z;y'}$ implies that for any  $\hat z \in \fJ(\tilde z, r)$,
\begin{equation}
\label{eq:assumption-1}
 \| \Delta(y, \hat z) \| \gtrsim \beta e^{-(1 + C_1 \epsilon_1 ) s'}. 
\end{equation}
\par Noting that $\tilde \mu_{\fr}(\Theta'(z) \setminus \Xi_{C_1\epsilon_1 s'}( y')) \ge \beta \tilde \mu_{\fr}(\Theta'(z))$ by the definition of $\bar \cR^{\fr}_i$, 
by running over all $\tilde{z} \in \cI_{z;y'}$, we have that 
\begin{align*} 
\sum_{\tilde z \in \cI_{z;y'}}\sum_{\hat z \in \fJ(\tilde z, r)} \cW(y) \cW(\hat z) &\gg \sum_{\tilde z\in \cI_{z;y'}} \cW(y)\cW(\tilde z) \\ 
 &\ge \beta \cW(y) \tilde \mu(\Theta'(z)) \ge \beta \cW(y) e^{-2s'} e^{-(d_i +\epsilon_1)s'}.  
\end{align*}
Since 
\[ |u(\fl(\hat y))\hat y \cap \bar \cR^{\fr}_i | \ge \beta^{31} e^{2s'}, \]
we can find a set $\mathcal Y$ of at least $\beta^{30} e^{2s'}$ different $ y' \in u(\fl(\hat y))\hat y \cap  \cR_i$ which are $\beta$-separated. The corresponding summation is:
\begin{align*}
    \sum_{y' \in \mathcal Y} \sum_{\tilde z \in \cI_{z;y'}} \sum_{\hat z \in \fJ(\tilde z, r)} \cW(y) \cW(\hat z) &\gg \sum_{y' \in \mathcal Y} \beta \cW(y) e^{-2s'} e^{-(d_i +\epsilon_1)s'} \\
     & \ge \beta^{31} e^{2s'}  \cW(y) e^{-2s'} e^{-(d_i +\epsilon_1)s'} \\
     &= \beta^{31} \cW(y) e^{-(d_i +\epsilon_1)s'}.
\end{align*}
By \eqref{eq:assumption-1} and Lemma \ref{lm:time-leave}, 
 For every $ \hat z \in \fJ\cap\Sigma^r$ such that $(y, \hat z) \in \diamond_r$ and \eqref{eq:assumption-1}, we have that $I(\hat y, \hat z) \subset I$ where $I$ is an interval of length $\le e^{C_1 \epsilon_1 s'}$. Note that the summand $\cW(y) \cW(\hat z)$ appears in the summation for $y' \in \cY$ only if $y' = u(r_{y'}) \hat y$ where $r_{y'} \in I(\hat y, \hat z)$. Therefore, every summand $\cW(y) \cW(\hat z)$ from above is counted in summations for at most $\beta^{-1}e^{C_1\epsilon_1 s'}$ different $y' \in \cY$. Therefore, we have that the contribution to $\cS(\Sigma^r)$
from the above summation is at least $\beta^{32} \cW(y) e^{-(d_i +\epsilon_1)s' - C_1 \epsilon_1 s'}$.
By running over all $y \in \fJ \cap \Sigma^r$ such that $\Omega'(y) \cap \cR'_i \neq \emptyset$, we get 
\begin{equation} 
\label{eq:cS-lower-bound}
\cS(\Sigma^r) \ge \beta^{62} \tilde{\fm} e^{- (d_i + \epsilon_1) s' - C_1 \epsilon_1 s'} \ge \tilde{\fm} e^{-d_i s' -2 C_1 \epsilon_1 s'}, \end{equation}
where $\tilde{\fm} = \tilde \mu_{\fr}(\Sigma^r) \asymp e^{-2s'}$.
\par On the other hand, let us get an upper bound on $\cS(\Sigma^r)$.
\par For any fixed $y \in \fJ \cap \Sigma^r$ and $z\in \cI \cap \Sigma^r \cap \bar \cR^{\fr}_i$, let us consider all $z'\in \cI_z$ such that there exists $r' \in I(y, z')$ with $|r'| \ge \beta^{100} e^{2s'}$. By Lemma \ref{lm:concentrated-surface}, $\fD_{z}(z') \in \cC(\cT(y), z)$. 
\par Let us first consider the case $y \in \HS(z,r)$.  By \eqref{eq:fScT}, we have that 
\begin{align*}
    \sum_{z' \in \cI_z, (y, z') \in \diamond_r} \cW(y) \cW(z') &\le  \beta^{-101} \cW(y) e^{2s'} e^{-2d_i s'} e^{-2s'} \\  &= \cW(y) e^{-2d_i s' +\epsilon_1 s'}.
\end{align*}
From the definition of $\bar \cR^{\fr}_{i}$, we have that 
\begin{align*} 
\sum_{y \in \HS(z,r)} \cW(y) &\ll \tilde \mu (\SS(z,r)) \le e^{-2s' - 10C_1 \epsilon_1 s'}. 
\end{align*}
Thus, by summing over $y \in \HS(z,r)$, we have 
\begin{align*}
   \sum_{y \in \HS(z,r)} \sum_{z' \in \cI_z, (y, z') \in \diamond_r} \cW(y) \cW(z') &\le \sum_{y \in \HS(z,r)}  \cW(y) e^{-2d_i s' +\epsilon_1 s'} \\
                                                                                    &\le e^{-2s' - 10C_1 \epsilon_1 s'} e^{-2d_i s' +\epsilon_1 s'} \le e^{-2s' -2d_i s' - 9 C_1 \epsilon_1 s'}.
\end{align*}
\par We then consider the case $y \not\in \HS(z,r)$. Since in this case $\cC(\cT(y), z)$ is not highly concentrated, we have that
\begin{align*}
    \sum_{z' \in \cI_z, (y, z') \in \diamond_r} \cW(y) \cW(z') &\le \cW(y) e^{2s' - 7C_1\epsilon_1 s'} e^{-2d_i s'} e^{-2s'} \\  &= \cW(y) e^{-2d_i s' - 7C_1\epsilon_1 s'}.
\end{align*}
Note that 
\[ \sum_{y \not\in \HS(z, r)} \cW(y) \ll \tilde \mu (\Sigma^r) \ll e^{-2s'}.\]
By summing over $y \not\in \HS(z,r)$, we have 
\begin{align*}
   \sum_{y \not\in \HS(z, r)} \sum_{z' \in \cI_z, (y, z') \in \diamond_r} \cW(y) \cW(z') &\le \sum_{y \not\in \HS(z, r)} \cW(y) e^{-2d_i s' - 7C_1\epsilon_1 s'} \\      
                       &\le e^{-2s'} e^{-2d_i s' - 7C_1\epsilon_1 s'} = e^{-2s'  -2d_i s' - 7C_1\epsilon_1 s'}.
\end{align*}
Therefore, by summing over $y \in \fJ\cap \Sigma^r$, we have 
\begin{align*}
   \sum_{y \in \fJ\cap \Sigma^r} \sum_{z' \in \cI_z, (y, z') \in \diamond_r} \cW(y) \cW(z') &\le e^{-2s'  -2d_i s' - 9C_1\epsilon_1 s'}+ e^{-2s'  -2d_i s' - 7C_1\epsilon_1 s'} \\ &\le e^{-2s'  -2d_i s' - 6.9C_1\epsilon_1 s'}.
\end{align*}
Since $z \in \cR_i$, we have  
$$\tilde \mu_{\fr} (\Theta'(z)) \ge e^{-(d_i + \epsilon_1)s'} e^{-2s'}.$$
Thus, we have that there are at most $\tilde \fm e^{2s'} e^{d_i s' + \epsilon_1 s'}$ 
different $z$ from $ \cI \cap \Sigma^r \cap \cR_i$. By running over all $z \in \cI\cap \cR_i \cap \Sigma^r$, we have 

\begin{align*}
    \cS(\Sigma^r) &= \sum_{z \in \cI\cap \cR_i\cap \Sigma^r} \sum_{y \in \fJ\cap \Sigma^r} \sum_{z' \in \cI_z, (y, z') \in \diamond_r} \cW(y) \cW(z') \\ 
    &\le \sum_{z \in \cI\cap \cR_i\cap \Sigma^r}  e^{-2s'  -2d_i s' - 6.9C_1\epsilon_1 s'} \\
    &\le  \tilde \fm e^{2s'} e^{d_i s' + \epsilon_1 s'} e^{-2s'  -2d_i s' - 6.9C_1\epsilon_1 s'} = \tilde \fm e^{-d_i s' -  6C_1\epsilon_1 s'}.
\end{align*} 
This upper bound is smaller than the lower bound in \eqref{eq:cS-lower-bound}
which leads to a contradiction. This proves that $\tilde \mu (\bar \cR^{\fr}_i) < \beta^{22}$.
\end{proof}
\par Using arguments similar to the proof of Proposition \ref{prop:dimension-improvement-random}, we will prove the measure of $\cR_{i,j}$ is small:
\begin{prop}
    \label{prop:dimension-improvement-strange-part}
    For any $ C_1 \epsilon_1 s' \le j \le 2s' $, we have 
    \[ \tilde{\mu}(\cR_{i,j}) \le \beta^{22}. \]
\end{prop}
\begin{proof}
For each $\Theta'(\bar z) \subset \cR_{i,j}$, fix a subset $\Xi_j(z) \subset \Theta'(\bar z)$ as in \eqref{eq:strange-nbhd-condition} and let $\cR'_{i,j}$ denote the union of all of those $\Xi_j(z)$ for $\Theta'(\bar z) \subset \cR_{i,j}$. To prove $\tilde \mu_{\fr}(\cR_{i,j}) < \beta^{22} $, 
it suffices to show that $\tilde \mu_{\fr}(\cR'_{i,j}) < \beta^{22.1}$.
\par Assume in contradiction that $\tilde \mu_{\fr}(\cR'_{i,j}) \ge \beta^{22.1}$. 
Using similar arguments as in the proof of Proposition \ref{prop:dimension-improvement-random}, we can get 
a subset $\cR''_{i,j} \subset \cR'_{i,j} $ such that $\tilde \mu_{\fr}(\cR''_{i,j}) \ge \beta^{22.2}$ and for any $y \in \cR''_{i,j}$,
\[|u(\fl(y))y \cap \cR'_{i,j}| \ge \beta^{30} e^{2s'}, \]
and get a cross-section $\Sigma^r$ such that 
\[ \tilde \mu_{\fr}(\Sigma^r\cap \cR''_{i,j}) \ge \beta^{30} \tilde \mu_{\fr}(\Sigma^r). \]
\par Let us define $\cS_j(\Sigma^r)$ as follows:
\[  \cS_j(\Sigma^r) := \sum_{ (y_1, y_2) \in \diamond_{r,j}}\cW(y_1) \cW(y_2), \]
where $(y_1, y_2) \in \diamond_{r,j}$ if 
\begin{itemize}
\item $y_1, y_2 \in  \fJ\cap\Sigma^r$, and 
\item  $\exists r' \in I(y_1, y_2)$ with $|r'|\ge \beta^{100} e^{2s'}$, and 
\item $(\Omega'(y_1)\cap \cR'_{i,j})\cup (\Omega'(y_2)\cap \cR'_{i,j}) \neq \emptyset$, and 
\item $\|\Delta(y_1, y_2)\| \in [\beta e^{-s' - j-1}, \beta e^{-s' - j }]$.
\end{itemize}

\par Let us first get a lower bound on $\cS_j(\Sigma^r)$. For any fixed $y \in \fJ \cap \Sigma^r$ with $\Omega'(y)\cap \cR''_{i,j} \neq \emptyset$, we can find $\hat y \in \Omega'(y)\cap \cR''_{i,j}$. By the properties of $\cR''_{i,j}$, we have
\[ |u(\fl(\hat y))\hat y \cap \cR'_{i,j}| \ge \beta^{22} e^{2s'}.\]
For any $y' \in u(\fl(\hat y))\hat y \cap \cR'_{i,j}$, there exists $z\in \cI$ and $ \Xi_j(z') \subset \Theta'(z)$ with $y' \in \Xi_j(z')$ such that 
$$\tilde \mu_{\fr} (\Xi_j(z')) \ge (1-\beta)^j e^{-2s'} e^{- (d_i+\epsilon_1)s' }$$
and 
$$ \tilde \mu_{\fr} (\Xi_{j+1}(y')) < (1-\beta)\tilde \mu_{\fr} (\Xi_j(z')).$$

Let us denote $\cI_{y'\!, j} := \cI_z\cap(\Xi_j(z') \setminus \Xi_{j+1}(y'))$. Following the same argument as in the proof of Proposition \ref{prop:dimension-improvement-random}, for every $y'_2 \in  \cI_{y'\!, j}$ we can get a finite subset 
$\fJ_j(y'_2, r) \subset \fJ\cap\Sigma^r$ such that for any $y_2 \in \fJ_j(y'_2, r)$, $(y_2, y)$ satisfies $\diamond_{r,j}$, and 
\begin{align*}
    \sum_{y_2 \in \fJ_j(y'_2, r)} \cW(y) \cW(y_2) \gg \cW(y) \cW(y'_2).
\end{align*}
By running over all $y'_2 \in \cI_{y'\!,j}$, we get 
\begin{align*}
   & \sum_{y'_2 \in \cI_{y'\!, j} } \,\sum_{y_2 \in \fJ_j(y'_2,r)} \cW(y) \cW(y_2) \\
   &\gg \sum_{y'_2 \in \cI_{y'\!, j}} \cW(y) \cW(y'_2) \gg \cW(y)\,\tilde \mu_{\fr}(\Xi_j(z') \setminus \Xi_{j+1}(y')) \\ 
   &> \cW(y) \beta(1-\beta)^j e^{-2s'} e^{- (d_i+\epsilon_1)s' } >  \cW(y)e^{-2s'} e^{- (d_i+2\epsilon_1)s' }.
\end{align*}
Similarly to the proof of the lower bound for $\cS(\Sigma^r)$ in the proof of Proposition \ref{prop:dimension-improvement-random} (specifically, the discussion in p.~\pageref{eq:assumption-1}), since 
\[ |u(\fl(\hat y))\hat y \cap \cR'_{i,j}| \ge \beta^{30} e^{2s'}, \]
we can find a set $\cY_j$ of at least $\beta^{29} e^{2s'}$ different $y' \in u(\fl(\hat y))\hat y \cap \cR'_{i,j}$ which are $\beta$-separated. 
The corresponding summation is:
\begin{align}
   & \sum_{y' \in \cY_j} \, \sum_{y'_2 \in \cI_{y'\!, j} } \,\sum_{y_2 \in \fJ_j(y'_2, r)} \cW(y) \cW(y_2) \label{eq: triple sum} \\
   &\qquad \gg \sum_{y' \in \cY_j} \cW(y)e^{-2s'} e^{- (d_i+2\epsilon_1)s' } \nonumber\\
   &\qquad \gg \beta^{29} e^{2s'} \cW(y)e^{-2s'} e^{- (d_i+2\epsilon_1)s' } = \beta^{29} \cW(y) e^{- (d_i+2\epsilon_1)s' }.\nonumber
\end{align}
Since $\|\Delta(y, y_2)\| \ge \beta e^{-s' - j-1}$, by Lemma \ref{lm:time-leave}, for each $y_2 \in \fJ\cap \Sigma^r$ such that $(y, y_2) \in \diamond_{r,j}$, we have that  $I(y, y_2) \subset I$ where $I$ is an interval of length $\le e^{j+1}$. Note that the summand $\cW(y) \cW(y_2)$ appears in the summation \eqref{eq: triple sum} for $y' \in \cY_j$ only if $y' = u(r_{y'}) \hat y$ where $r_{y'} \in I(y, y_2)$. Therefore, every summand $\cW(y) \cW(y_2)$ appears in \eqref{eq: triple sum} for at most $\beta^{-1} e^{j+1}$ different $y' \in \cY_j$. 
This implies that for a fixed $y \in \fJ \cap \Sigma^r\cap \cR''_{i,j}$,
\begin{align*}
    \sum_{y_2:\ (y,y_2) \in \diamond_{r,j}}\cW(y) \cW(y_2) &\geq \beta^{30} \cW(y) e^{- (d_i+2\epsilon_1)s'-j-1 } \\
    &> \cW(y) e^{- (d_i + 3 \epsilon_1)s' - j }.
\end{align*} 
By summing over all $y \in \fJ \cap \Sigma^r$ with $\Omega'(y) \cap \cR''_{i,j} \neq \emptyset$, we have
\begin{equation}
\label{eq:j lower bound}
\cS_j(\Sigma^r) \ge \tilde \mu(\cR''_{i,j}\cap \Sigma^r) e^{- (d_i + 3 \epsilon_1)s' - j} \ge \beta^{30} \tilde \fm e^{- (d_i + 3 \epsilon_1)s' - j},
\end{equation}
where $\tilde \fm = \tilde \mu(\Sigma^r)$.

\medskip 

\par Now let us get an upper bound on $\cS_j(\Sigma^r)$.
To do this, we need the following lemma, for a fixed $y_2 \in \fJ\cap \Sigma^r$, and $ z \in \cI$ such that $\Theta'( z)$ contains a $ \Xi_j (z') \subset \Sigma^r \cap \cR'_{i,j} $:
\begin{lemma}\label{lem:sigma j lemma}
    If there is at least one $y_1\in \cI_{ z} \cap \Xi_j (z')$ so that $(y_2, y_1) \in \diamond_{r,j}$ then
\[ z' \in \Sigma_{s', j-1} (y_2). \]
\end{lemma}
\begin{proof}
Since $z', y_2 \in \Sigma^r$, it suffices to show that 
\[ z' \in \Pi_{s', j-1}(y_2). \]
Write
\[ y_1 = a_0 a u^\ast (r^\ast_{y_1})\exp(\vv w^-) \exp(\vv w^+) u(r_{y_1}) y_2, \]
where $\vv w^- \in \vv r^-$ satisfies that $\|\vv w^-\| \in [\beta e^{-s' - j-1}, \beta e^{-s' - j }]$. Then by Lemma~\ref{lm:concentrated-surface}, we have $\|\vv w^+\| \le \beta e^{s' - j}$. Let $\hat y_1 = a(-2t) y_1, \hat y_2 = a(-2t) y_2$. Then we have
\[ \hat y_1 = a_0 a u^\ast (\hat r^\ast) \exp (\hat{ \vv w}^-)\exp (\hat{ \vv w}^+) u(\hat r) \hat y_2,\]
where $\hat r^\ast, \hat r \in [\beta]$, $\hat{\vv w}^- \in B_{\vv r^-}(\beta e^{-j})$, and  $\hat{\vv w}^+ \in B_{\vv r^+}(\beta e^{-j})$. By Lemma \ref{lm:product-w}, 
\[ \exp (\hat{ \vv w}^-)\exp (\hat{ \vv w}^+) = \exp(\hat{\vv w})\exp(\hat{\vv h}), \]
where $\hat{\vv w} \in B_{\fr_1 + \fr_2} (1.1 \beta e^{-j})$, and $\hat{\vv h} \in B_{\R \fa_0 + \fh} (0.1 \beta e^{-2j})$. 
Therefore, 
\begin{equation}
\label{eq:hat-y1-y2}
\hat y_1 = \exp(\tilde{\vv w}) \exp(\tilde{\vv h}) \hat y_2, \end{equation}
where $\tilde{\vv w} \in  B_{\fr_1 + \fr_2} (1.3 \beta e^{-j})$
and $\tilde{\vv h} \in B_{\R \fa_0 + \fh}(1.1\beta)$. Note that $y_1 \in \Xi_j (z')$, we have 
\[ y_1 = \exp(\vv w^+_1 + \vv w^-_1) \exp(\vv h_1) z',  \]
where $\vv w^+_1 \in B_{\vv r^+}(\beta e^{-s'})$, $\vv w^-_1 \in B_{\vv r^-}(\beta e^{-s' - j})$, and $\vv h_1 \in \fq^{A_0 H}(2s', \beta)$. Let us denote $\hat z' = a(-2s') z'$, then we have 
\begin{equation} 
\label{eq:hat-y1-z}
\hat y_1 = \exp(\tilde{\vv w}_1) \exp(\tilde{\vv h}_1) \hat z', \end{equation}
where $\tilde{\vv w}_1 \in B_{\fr_1 + \fr_2}(\beta e^{-j})$, and $\tilde{\vv h}_1 \in B_{\R \fa_0 + \fh}(\beta)$. \eqref{eq:hat-y1-y2} and \eqref{eq:hat-y1-z} imply that 
\[ \hat y_2 = \exp(-\tilde{\vv h}) \exp(-\tilde{\vv w})\exp(\tilde{\vv w}_1) \exp(\tilde{\vv h}_1) \hat z'.  \]
By Lemma \ref{lm:product-w}, we have 
\[ \hat y_2 = \exp(\tilde{\vv w}_2) \exp(\tilde{\vv h}_2) \hat z', \]
where $\tilde{\vv w}_2 \in B_{\fr_1 + \fr_2}(2.1 \beta e^{-j})$ and 
$\tilde{\vv h}_2 \in B_{\R \fa_0 + \fh}(2.1 \beta)$. Noting that $\hat y_2 = a(-2s') y_2$, and $\hat z' = a(-2s') z'$, we have 
\[ z' \in \Pi_{s', j-1} (y_2). \]
This completes the proof of the lemma.
\end{proof}

We now continue with the proof of Proposition~\ref{prop:dimension-improvement-strange-part}. Recall that $y_2 \in \fJ\cap \Sigma^r$, and $ z \in \cI$ such that $\Theta'( z)$ contains a $ \Xi_j (z') \subset \Sigma^r \cap \cR'_{i,j} $. 
Using similar arguments to those used in the proof of Lemma~\ref{lem:sigma j lemma}, we can prove that 
\[ \Xi_j(z') \subset \Sigma_{s', j-2}(y_2). \]
Define 
\begin{equation}\label{eq:cIy2j}
\begin{aligned}
\cI(y_2, j) := \Bigl\{ z \in \cI \cap \Sigma^r \cap \cR_{i,j}\!: &\,\text{the heavy slice $\Xi_j(z') \subset \Theta'(z)$} \\
&\text{ satisfies } \Xi_j(z') \subset \Sigma_{s', j-2}(y_2) \cap \cR'_{i,j} \Bigr\}.
\end{aligned}
\end{equation}
Fix $z \in \cI(y_2, j)$. By Lemma \ref{lm:concentrated-surface}, the coordinates $\fD_{ z}(y_1)$ of all $y_1\in \cI_z$ such that $(y_2,y_1) \in \diamond_{r,j}$
are 
contained in 
$$\cC(\cT(y_2),  z) \cap \fB(\cT(y_2),  z).$$ 
By $\diamond_{r,j}$ and Lemma \ref{lm:concentrated-surface}, we have that if there is at least one such $y_1$ the size of $\fB(\cT(y_2),  z)$ is $\le \beta^{-98} e^{-s'-j}$.


Using \eqref{eq:fScT} we may conclude that 
\[ \sum_{\fD_{ z}(y_1) \in \cC(\cT(y_2),  z)\cap \fB(\cT(y_2),  z)} \cW(y_1)\cW(y_2) \ll  \cW(y_2) \beta^{-198} e^{-j} e^{-2d_i s'}. \]
By Lemma \ref{lm:initial-dimension-under-a-t}, we have 
\[ \mu(\Sigma_{s', j-2}(y_2)) \le e^{-2s'} e^{- d_0 (j-2)}.\]
Since $\Pi(x)$ is a full playground, we have 
\[ \tilde \mu(\Sigma_{s', j-2}(y_2)) \le \beta^{-15}e^{-2s'} e^{- d_0 (j-2)}.\]
Noting that for $z'$ as in \eqref{eq:cIy2j} 
\[
\tilde \mu_{\fr}(\Xi_j(z')) \ge (1-\beta)^j e^{-2s'} e^{-(d_i + \epsilon_1)s'},
\]
we have that 
\begin{align*}
|\cI(y_2, j)| & \le \frac{\beta^{-15}e^{-2s'} e^{-  d_0 (j-2)}}{ (1-\beta)^j e^{-2s'} e^{-(d_i + \epsilon_1)s'}}\\
&= \beta^{-15}(1-\beta)^{-j} e^{- d_0 (j-2)} e^{ (d_i+\epsilon_1) s' } .
\end{align*}
 Thus, for a fixed $y_2 \in \fJ\cap \Sigma^r$, the corresponding 
sum
\begin{align*}
 \sum_{ z \in \cI(y_2, j)} \,&\sum_{\substack{\fD_{ z}(y_1) \in \,\cC(\cT(y_2), z) \\ \quad\cap \fB(\cT(y_2),  z)}} \cW(y_1)\cW(y_2) \\ \ll& \cW(y_2) \beta^{-198} e^{-j} e^{-2d_i s'}  \beta^{-15}(1-\beta)^{-j} e^{- d_0 (j-2)} e^{ (d_i+\epsilon_1) s' } \\
\le& 2 \cW(y_2) \beta^{-213} e^{- d_0 (j-2)} e^{ (d_i+\epsilon_1) s' }e^{-j} e^{-2d_i s'}.
\end{align*}
By summing over all $y_2 \in \fJ\cap \Sigma^r$, we get
\[ \cS_j(\Sigma^r) \le 2\tilde \fm \beta^{-213} e^{- d_0 (j-2)} e^{ (d_i+\epsilon_1) s' }e^{-j} e^{-2d_i s'} \le  \tilde \fm e^{2 \epsilon_1 s'} e^{- d_0 j} e^{-d_i s' - j}. \]
This upper bound is smaller than the lower bound in \eqref{eq:j lower bound} because $ d_0 j \ge d_0 C_1 \epsilon_1 s' \ge 10 \epsilon_1 s'$. This leads to a contradiction, completing the proof of Proposition~\ref{prop:dimension-improvement-strange-part}.
\end{proof}
\section{No loss of dimension in a coarser scale}
\label{sec-no-loss-of-dimension}
\par Using similar arguments to those in the proofs of Proposition \ref{prop:dimension-improvement-random} and Proposition \ref{prop:dimension-improvement-strange-part}, we can also prove that compared to $\mu_i$, the dimension of $\tilde \mu$ does not drop significantly:
\begin{prop}
  \label{prop:dimension-does-not-drop}
  Let us keep the notation as in Proposition \ref{prop:improve-dimension}. Then there exists $\mathfrak{Z} \subset \Pi(x)$ with $\tilde \mu(\mathfrak Z) \le \beta^{50}$ such that 
for $x \in \cI\setminus\mathfrak{Z}$,
\begin{equation} 
  \label{eq:upper-theta}
  \tilde \mu(\Theta'(x)) \le e^{-2s'} e^{-(d_i - C_1 \epsilon_1)s'}. \end{equation}
\end{prop}
\begin{proof}[Sketch of the proof]
    \par For $z \in \cI$, let us call $\Theta'(z)$ \emph{extremely heavy} if 
    \[ \tilde \mu(\Theta'(z)) > \tilde M_{s'} \qquad\text{for $ \tilde M_{s'}= e^{-2s'} e^{-(d_i - C_1 \epsilon_1)s'}$}. \]
Let $\mathfrak Z$ denote the union of extremely heavy $\Theta'(z)$'s. Then it suffices to show that 
\[ \tilde \mu (\mathfrak Z ) \le \beta^{50}. \]
Similarly to \eqref{eq:strange-normal-decomposition}, let us decompose
\[ \fZ = \fZ' \cup \bigcup_{ C_1 \epsilon_1 s' \le j \le 2s'} \fZ_j. \]
Here $\fZ'$ denotes the union of extremely heavy $\Theta'(z)$ such that for any 
$$\Xi_{C_1 \epsilon_1 s'} (y) \subset \Theta'(z),$$ 
we have 
\[ \tilde \mu(\Xi_{C_1 \epsilon_1 s'} (y)) \le (1- \beta) \tilde \mu (\Theta'(z)); \]
and $\fZ_j$ denotes the union of extremely heavy $\Theta'(z)$ such that there exists 
$\Xi_j(y) \subset \Theta'(z)$ such that 
\[ \tilde \mu (\Xi_j(y)) > (1-\beta)^j \tilde \mu(\Theta'(z)), \]
and for any $\Xi_{j+1} (y') \subset \Xi_j(y)$, 
\[ \tilde \mu (\Xi_{j+1}(y')) \le  (1-\beta) \tilde \mu(\Xi_j(y)). \]
Then it suffices to show that $\tilde \mu(\fZ') \le \beta^{50.1}$, and 

$$\tilde \mu(\fZ_j) \le \beta^{50.1},  \text{ for any } j \in [C_1 \epsilon_1 s', 2s'].$$
\par We first sketch the proof of 
$$\tilde \mu(\fZ') \le \beta^{50.1},$$
which follows the proof of Proposition \ref{prop:dimension-improvement-random}.
\par Assume in contradiction that 
\[ \tilde \mu(\fZ') > \beta^{50.1}. \]
By repeating the argument in the proof of Proposition \ref{prop:dimension-improvement-random}, 
we get $\fZ'' \subset \fZ'$ and a cross-section $\Sigma^r$ such that 
\begin{itemize}
    \item $\tilde \mu(\fZ'') \ge \beta^{51}$;
    \item $\tilde \mu (\fZ'' \cap \Sigma^r) \ge \beta^{52} \tilde \mu(\Sigma^r)$;
    \item For any $y \in \fZ''$, $\left|u(\fl(y)) y \cap \fZ' \right| \ge \beta^{52} e^{2s'}$.
\end{itemize}
Now let us consider the intersection number $\cZ(\Sigma^r)$ as follows:
\[ \cZ(\Sigma^r) := \sum_{(y_1, y_2) \in \diamond'_r} \cW(y_1)\cW(y_2), \]
where $(y_1, y_2) \in \diamond'_r$ if 
\begin{itemize}
\item $y_1, y_2 \in \fJ\cap \Sigma^r$, and 
\item $\exists r' \in I(y_1, y_2)$ with $|r'| \ge \beta^{100} e^{2s'}$, and
\item $(\Omega'(y_1) \cap \fZ') \cup (\Omega'(y_2) \cap \fZ')  \neq \emptyset$.
\end{itemize}
We can repeat the argument in the proof of Proposition \ref{prop:dimension-improvement-random} to get lower and upper bounds on $\cZ(\Sigma^r)$ (similarly to $\cS(\Sigma^r)$): 
\par For lower bound, we get that for any $y_1 \in \fJ \cap \Sigma^r$ such that 
\[ \Omega'(y_1) \cap \fZ'' \neq \emptyset, \]
the summation 
\begin{align}\label{eq: y_2 sum for diamond'}
    \sum_{y_2: (y_1, y_2) \in \diamond'_r} \cW(y_1) \cW(y_2) &\gg \beta^{52} e^{2s' - C_1 \epsilon_1 s'} \beta \tilde{M}_{s'} \cW(y_1) ,
\end{align}
since the trajectory $u(\fl(y)) y \cap \fZ'$ encounters $\gtrapprox \beta ^{51} e^{2s'}$ many very heavy $\Theta'$ neighbourhoods sufficiently far from $\Sigma^r$. For each such $\Theta'(z)$, 
\begin{equation}\label{eq:theta sum'}
\tilde M_{s'} \leq \tilde \mu (\Theta'(z)) \approx \sum_{y_2}\!{\vphantom{\sum}}' \cW(y_2)
\end{equation}
where the sum is over all $y_2 \in \fJ \cap \Sigma^r$ so that $(y_1,y_2) \in \diamond_r'$ and there is a $|r'|\leq \beta e^{2s'}$ and $y_2' \in \Omega(y_2)$ so $u(r') y_2' \in \Theta'(z)$.
However each such $(y_1,y_2)$ may appear in the sum \eqref{eq:theta sum'} up to $|I(y_1,y_2)|\beta^{-1}$ many times. Throwing away all $y_2$ in for which the corresponding $u(r')$ trajectory meets $\Theta'(z)$ in one $\Xi_{C_1\varepsilon_1 s'}$-subneighbourhood we can ensure that $|I(y_1,y_2)|\lessapprox e^{C_1\epsilon_1s'}$.
Combining everything one gets \eqref{eq: y_2 sum for diamond'}.

By summing up all $y_1 \in \fJ \cap \Sigma^r$ with $\Omega'(y_1) \cap \fZ'' \neq \emptyset$, we have 
\begin{align*} \cZ(\Sigma^r) &\gg \beta^{52} \tilde \fm \beta^{52} e^{2s' - C_1 \epsilon_1 s'} \beta \tilde{M}_{s'} \\
                            &= \beta^{110} \tilde \fm e^{-d_i s'} \ge \tilde \fm e^{-d_i s' -\epsilon_1 s'},
\end{align*}
where $\tilde \fm = \tilde \mu (\Sigma^r)$.
\par For the upper bound, by Lemma \ref{lm:concentrated-surface}, for any $y_1 \in \fJ \cap \Sigma^r$ and $z \in \cI \cap \Sigma^r \cap \fZ'$, we have that for any $y_2 \in \cI_z$ such that $(y_1, y_2) \in \diamond'_r$, we have 
\[ \fD_z (y_2) \in \cC(\cT(y_1), z). \]
By \eqref{eq:fScT}, we get
\begin{align*}
    \sum_{y_2 \in \cI_z: (y_1, y_2) \in \diamond'_r} \cW(y_1) \cW(y_2) &\le  \cW(y_1) \beta^{-101} e^{-2d_i s'}.
\end{align*}

Note that for any $z \in \cI \cap \Sigma^r \cap \fZ'$, 
\[ \tilde \mu (\Theta'(z)) > e^{-2s'} e^{-(d_i - C_1 \epsilon_1 )s'}. \]
Thus, we have 
\[ |\cI \cap \Sigma^r \cap \fZ'| < \tilde \fm e^{2s'} e^{(d_i - C_1 \epsilon_1)s'} \le e^{(d_i - C_1 \epsilon_1)s'}.\]
Therefore, by summing over all $y_1 \in \fJ\cap \Sigma^r$ and $z \in \cI \cap \Sigma^r \cap \fZ'$, we have 
\begin{align*}
    \cZ(\Sigma^r) &\le \tilde \fm |\cI \cap \Sigma^r \cap \fZ'| \beta^{-101} e^{-2d_i s'} \\
                  &\le \tilde \fm e^{(d_i - C_1 \epsilon_1)s'} \beta^{-101} e^{-2d_i s'} \le \tilde \fm e^{(-d_i - C_1 \epsilon_1 + \epsilon_1)s'}.
 \end{align*} 
 The upper bound is smaller than the lower bound, leading to a contradiction. This proves that 
 \[\tilde \mu(\fZ') \le \beta^{50.1}. \]
 \par We then sketch the proof of 
 \[\tilde \mu(\fZ_j) \le \beta^{50.1},  \text{ for any } j \in [C_1 \epsilon_1 s', 2s'], \]
 which follows the proof of Proposition \ref{prop:dimension-improvement-strange-part}.
 \par For every $z \in \cI \cap \Sigma^r \cap \fZ_j$, let us fix a $\Xi_j(z') \subset \Theta'(z)$ such that 
 \[ \tilde \mu (\Xi_j(z')) > (1-\beta)^j e^{-2s'} e^{-(d_i - C_1 \epsilon_1)s'}, \]
 and for any $\Xi_{j+1} (z'') \subset \Xi_j(z')$, 
 \[ \tilde \mu (\Xi_{j+1} (z'')) \le (1-\beta) \tilde \mu (\Xi_j(z')).  \]
 Let $\fZ'_j$ denote the union of these $\Xi_j(z')$'s. Then it suffices to show that 
 \[ \tilde \mu(\fZ'_j) \le \beta^{50.1}. \]
 Let us assume in contradiction that 
 \[\tilde \mu(\fZ'_j) > \beta^{50.1}. \]
 Repeating the argument as in the proof of Proposition \ref{prop:dimension-improvement-random}, we will get $\fZ''_j \subset \fZ'_j$ and a cross-section $\Sigma^r$ satisfying
 \begin{itemize}
     \item $\tilde \mu (\fZ''_j) \ge \beta^{51}$;
     \item $ \tilde \mu (\fZ''_j \cap \Sigma^r) \ge \beta^{52} \tilde \mu(\Sigma^r) $;
     \item For any $y \in \fZ''_j$, we have $|u(\fl(y))y \cap \fZ'_j| \ge \beta^{52} e^{2s'}$.
 \end{itemize}
Now let us consider the intersection number $\cZ_j (\Sigma^r)$ as follows:
\[\cZ_j(\Sigma^r) := \sum_{(y_1, y_2) \in \diamond'_{r,j}} \cW(y_1)\cW(y_2), \]
where $(y_1, y_2) \in \diamond'_{r,j}$ if 
\begin{itemize}
\item $y_1, y_2 \in \fJ\cap \Sigma^r$, and 
\item $\exists r' \in I(y_1, y_2)$ with $|r'| \ge \beta^{100} e^{2s'}$, and
\item $\|\Delta(y_1, y_2)\| \in  [\beta e^{-s' - j-1}, \beta e^{-s' - j }]$, and
\item $(\Omega'(y_1) \cap \fZ') \cup (\Omega'(y_2) \cap \fZ')  \neq \emptyset$.
\end{itemize}
We can repeat the argument in the proof of Proposition \ref{prop:dimension-improvement-strange-part} to get lower and
upper bounds on $\cZ_j(\Sigma^r)$ (similarly to $\cS_j(\Sigma^r)$):
\par For lower bound, we will get that for any $y_1 \in \fJ\cap \Sigma^r$ with 
$$\Omega'(y_1) \cap \fZ''_j \neq \emptyset,$$
the summation 
\begin{align*}
    \sum_{y_2: (y_1, y_2) \in \diamond'_{r,j}} \cW(y_1) \cW(y_2) &\gg \cW(y_1) \beta^{51} e^{2s' - j} \beta (1-\beta)^j \tilde{M}_{s'} \\ 
     &\ge \cW(y_1) \beta^{53} e^{2s' - j} e^{-2s'} e^{-(d_i - C_1 \epsilon_1) s'}  \\ 
     &\ge \cW(y_1) \beta^{53} e^{(-d_i + C_1 \epsilon_1) s' - j}.
\end{align*}
By summing over all $y_1 \in \fJ\cap \Sigma^r$ with 
$$\Omega'(y_1) \cap \fZ''_j \neq \emptyset,$$
we get 
\begin{align*}
    \cZ_j(\Sigma^r) \ge \beta^{52} \tilde \fm \beta^{53} e^{(-d_i + C_1 \epsilon_1) s' - j} \ge \tilde \fm e^{(-d_i + C_1 \epsilon_1 -\epsilon_1) s' - j}.
\end{align*}
For upper bound, for $y_1 \in \fJ \cap \Sigma^r$ and $z \in \cI\cap \Sigma^r$ such that $\Theta'(z)$ contains a 
\[ \Xi_j(z') \subset \fZ'_j, \]
if there exists $y_2 \in \cI \cap \Xi_j(z')$ with $(y_1, y_2) \in \diamond'_{r,j}$, we will get 
\[ \Xi_j(z') \subset \Sigma_{s', j-2} (y_1). \]
Let us define
\begin{equation}
\begin{aligned}
\cI'(y_1, j) := \Bigl\{ z \in \cI \cap \Sigma^r \cap \fZ'_{j}\!: &\,\text{the heavy slice $\Xi_j(z') \subset \Theta'(z)$} \\
&\text{ satisfies } \Xi_j(z') \subset \Sigma_{s', j-2}(y_1) \cap \fZ'_{j} \Bigr\}.
\end{aligned}
\end{equation}
For $z \in \cI'(y_1, j)$, let us define 
\[ \fJ(y_1, z, j) := \left\{ y_2 \in \cI_z \cap \Sigma_{s', j-2}(y_1) \cap \fZ'_j: (y_1, y_2) \in \diamond'_{r,j} \right\}. \]
Then for every $z \in \cI'(y_1, j)$, by Lemma \ref{lm:concentrated-surface}, for any $y_2 \in \fJ(y_1, z, j)$, we have
\[ \fD_z(y_2) \in \cC(\cT(y_1), z) \cap \fB(\cT(y_1), z). \]
Therefore, 
\[ \sum_{y_2 \in \fJ(y_1, z, j)} \cW(y_1) \cW(y_2) \le \cW(y_1) \beta^{-198} e^{-j} e^{-2d_i s'}. \]
Moreover, since $\tilde \mu (\Sigma_{s', j-2} (y_1)) \le e^{-2s'} e^{- d_0 (j-2)}$, we have 
\begin{align*} 
|\cI'(y_1, j)| &\le e^{-2s'} e^{- d_0 (j-2)} (1-\beta)^{-j} e^{2s'} e^{(d_i - C_1 \epsilon_1)s'}  \\
                &\le e^{(d_i - C_1 \epsilon_1 + \epsilon_1)s' - d_0 j}.
\end{align*}
By summing up all $y_1 \in \fJ\cap \Sigma^r$ and $z \in \cI'(y_1, j)$, we get 
\begin{align*}
    \cZ_j(\Sigma^r) &\le \sum_{y_1 \in \fJ\cap \Sigma^r} \cW(y_1) |\cI'(y_1, j)|  \beta^{-198} e^{-j} e^{-2d_i s'} \\
                    &\le \sum_{y_1 \in \fJ\cap \Sigma^r} \cW(y_1) e^{(d_i - C_1 \epsilon_1 + \epsilon_1)s' - d_0 j} \beta^{-198} e^{-j} e^{-2d_i s'} \\
                     &\le \tilde \fm e^{(-d_i - C_1 \epsilon_1 + 2\epsilon_1)s' - j - d_0 j}.
\end{align*}
The upper bound is smaller than the lower bound, leading to a contradiction. This proves that 
 \[\tilde \mu(\fZ'_j) \le \beta^{50.1}. \]
\end{proof}
}


\section{Dimension Improvement --- structured part}
\label{sec-structured-component}
\par Let us keep notation in Definition \ref{def:structure-random-decompostion}. This section is devoted to the study of the structured part $\bar \cR^{\fs}_i$.

\par We will prove the following statement:
\begin{prop}
  \label{prop:dimension-improvement-structured-component}
  Suppose $ d_i \le 4 - 10^{11} C_1 \epsilon_1 $.   We have 
   \[ \tilde \mu (\bar \cR^{\fs}_i) \le \beta^{22}. \]
\end{prop}
\par Note that by \eqref{eq:strange-normal-decomposition} and \eqref{eq:normal-random-structure-decomposition}, Proposition \ref{prop:dimension-improvement-random}, \ref{prop:dimension-improvement-strange-part} and \ref{prop:dimension-improvement-structured-component} will imply \eqref{eq:heavy-part-is-small} and thus Proposition \ref{prop:improve-dimension}.

\par Note that up to now, from the results of \S\ref{subsec-initialize} we have only used Proposition~\ref{prop:initial-bound-dimension}. Proposition~\ref{prop:dimension-improvement-structured-component} uses in an essential way Proposition~\ref{prop:sl2xr2-closing-lemma}, which essentially says $\tilde \mu$ cannot concentrate along $\SL(2,\R)\ltimes\R^2$-orbits.

\par Let us first fix some notation.
\begin{defn}
\label{def:full-omega}
For $z \in \fJ$, if
\begin{equation} 
  \label{eq:full-omega}
  \tilde \mu(\Omega'(z)) \ge e^{-2s'} e^{-2d_i s' -  8 C_1 \epsilon_1 s'},\end{equation}
we will call $\Omega'(z)$ \emph{full}. Otherwise, we call it \emph{almost empty}.
\end{defn}
\begin{defn}
\label{def:comparable}
    For $z \in \cI$, we say that two curve neighborhoods $\cC = \cC(\cT(y), z)$ and $\cC' = \cC(\cT(y'), z)$ are \emph{comparable} if $\cC'$ is contained in a $e^{-3s' + 11 C_1 \epsilon_1 s'}$-neighborhood of $\cC$. Let us denote by
   \[ \cC' \approx \cC, \]
   if $\cC$ and $\cC'$ are comparable.
\end{defn}
\par We need the following lemma:
\begin{lemma}
    \label{lm:hyperbola-plane}
    Let $C \subset \fC =  \left[1.2 \beta e^{-s'}\right]^2$ 
    be a hyperbola with the following equation:
    \[ X = \frac{-Y}{w_1 Y + k}, \]
    where $|w_1| \in [e^{s' - 10 C_1 \epsilon_1 s'}, \beta e^{s'}]$, and $|k| \in [e^{\ell}, e^{\ell+1}]$ for some $\ell \in \Z_{\ge 0}, \ell \le s'$. Let $C' \subset \fC$ be another hyperbola with the following equation:
    \[ X =   \frac{-Y}{w'_1 Y + k'},\]
    where $|w'_1| \in [e^{s' - 10 C_1 \epsilon_1 s'}, \beta e^{s'}]$. 
    \par (1) For $0< A < e^{2s'-C_1\epsilon_1 s'}$, if $C'$ is contained in the $A e^{-3s'}$ neighborhood of $C$, then we have 
    \begin{equation}\label{eq:w1w'1}
    |w_{1} - w'_{1}| \le Ae^{-s'+2\ell}, \text{ and } |k - k'| \le Ae^{-2s' +2\ell}.
    \end{equation}
    Conversely, if \eqref{eq:w1w'1} holds, then $C'$ is contained in the $A e^{-3s'}$ neighborhood of $C$.
     \par (2) Suppose $C'$ is not contained in the $e^{-3s' + 11 C_1 \epsilon_1 s'}$ neighborhood of $C$. For $0< A < e^{2s'-C_1\epsilon_1 s'}$, if we have
     \begin{equation}\label{eq:w1-w1'}
     |w_{1} - w'_{1}| \in \left[Ae^{-s'+2\ell} , Ae^{-s'+2\ell+1}\right],
     \end{equation}
     and there exists $ \vv p' \in C'$ such that $\|\vv p'\| \ge e^{-s' -10 C_1 \epsilon_1 s'}$ and $\vv p'$ is contained in the $\beta^{-100} e^{-3s'}$ neighborhood of $C$, then $C'$ is contained in the $\beta^{-100}A e^{-3s'}$ neighborhood of $C$. Moreover, if $\vv p'_1, \vv p'_2\in C'$ satisfy that $\|\vv p'_1\| , \|\vv p'_2\| \ge e^{-s' -10 C_1 \epsilon_1 s'}$ and $\vv p'_1, \vv p'_2$ are both contained in the $\beta^{-100} e^{-3s'}$ neighborhood of $C$, then 
     $$\|\vv p'_1 - \vv p'_2\| \le A^{-1}  e^{-s' + 11 C_1 \epsilon_1 s'}.$$
\end{lemma}
\begin{figure}
\centering
\begin{tikzpicture}
\draw [gray!50]  (0,0) -- (0,11) -- (10,11) -- (10,0)  -- cycle;
\draw [blue] plot [smooth] coordinates {(0,0) (2.5, 4) (5,7) (7.5, 9) (10,10) };
\draw [red] plot [smooth] coordinates {(0,0.5) (2.5, 4.5) (5,7.5) (7.5, 9.5) (10,10.5) } node[above left]{$C$};
\draw [blue] plot [smooth] coordinates {(0,1) (2.5, 5) (5,8) (7.5, 10) (10,11) };
\draw [red] plot [smooth] coordinates {(0,0.5) (2.5, 4.9) (5,7.9) (7.5, 9.5) (10,10.2) }node[below left]{$C'$};
\node at (7.5,9.7) { $\vv p'$};
\node at (-0.6,0.5) { $(0,0)$};

\end{tikzpicture}
\caption{Lemma \ref{lm:hyperbola-plane}, part (2) --- if there is a $\vv p' \in C'$, not too close to the origin, near $C$ then all of $C'$ is within a comparable distance to~$C$.}
\end{figure}
\begin{proof}

\par Let us first prove (1). Our assumption is that for all $Y \in [\beta e^{-s'}]$,
  \[ \left|\frac{Y}{w_{1} Y + k} - \frac{Y}{w'_{1}Y + k'}  \right| \le  Ae^{-3s'}.\]
  This implies that 
  \[ |Y| \cdot |w_{1} - w'_{1}| \frac{\left| Y + \frac{k' - k}{w_{1} - w'_{1}} \right|}{|w_{1} Y + k| \cdot |w'_{1} Y + k'|} \le  Ae^{-3s'}. \]
  Then it is easy to get that $|k'| \in [e^{\ell-1}, e^{\ell+2}]$.
  \par Let us consider the following two subcases:
  \par \textbf{Subcase 1}: $|k - k'| < e^{-s'}|w_{1} - w'_{1}|$;
  \par \textbf{Subcase 2}: $|k - k'| \ge e^{-s'}|w_{1} - w'_{1}|$.
  \par For \textbf{Subcase 1}, we have 
  \[ |w_{1} - w'_{1}| e^{-2s'  - 2\ell} \le A e^{-3s' }, \]
  hence
  \[ |w_{1} - w'_{1}| \le Ae^{-s' +2\ell}. \]
  Moreover
  \[ |k - k'| <  e^{-s'}|w_{1} - w'_{1}| \le  A e^{-2s'+2\ell},\]
  establishing \eqref{eq:w1w'1}.

  For \textbf{Subcase 2}, we have 
 \[ |k - k'| e^{-s'  - 2 \ell} \le Ae^{-3s'}, \]
 implying
 \[ |k - k'| \le A e^{-2s'+2\ell}. \]
 Then 
 \[ |w_{1} - w'_{1}| \le e^{s'} |k - k'| \le Ae^{-s'+2\ell}. \]
 This proves \eqref{eq:w1w'1} and hence the first part of (1).
 \par Now let us assume \eqref{eq:w1w'1}, we want to prove that for any $Y \in [\beta e^{-s'}]$, 
  \[ \left|\frac{Y}{w_{1} Y + k} - \frac{Y}{w'_{1}Y + k'}  \right| \le  Ae^{-3s'}.\]
  This is equivalent to 
  \[ |Y| \cdot |w_{1} - w'_{1}| \frac{\left| Y + \frac{k' - k}{w_{1} - w'_{1}} \right|}{|w_{1} Y + k| \cdot |w'_{1} Y + k'|} \le  Ae^{-3s'}. \]
Note that the left-hand side is bounded from above by 
\begin{equation} 
\label{eq:y-y1-lhs}
e^{-s'- 2\ell} |w_{1} - w'_{1}| \cdot  \left| Y + \frac{k' - k}{w_{1} - w'_{1}} \right|.  \end{equation}
  Let us consider the 2 subcases given as above. 
 \par For \textbf{Subcase 1}, we have
 \begin{align*}  
 |w_{1} - w'_{1}| \cdot  \left| Y + \frac{k' - k}{w_{1} - w'_{1}} \right| &\le e^{-s'}|w_{1} - w'_{1}|  \\ 
                          &\le e^{-s'}Ae^{-s' +2\ell} = A e^{-2s' + 2\ell}.  \end{align*}
  For \textbf{Subcase 2}, we have 
  \begin{align*}  
 |w_{1} - w'_{1}| \cdot  \left| Y + \frac{k' - k}{w_{1} - w'_{1}} \right| &\le |k - k'|  \le A e^{-2s' + 2\ell}.  \end{align*}
This shows that for both subcases we have \eqref{eq:y-y1-lhs} is bounded from above by $A e^{-3s'}$, which implies the second statement of (1).
 \par Let us prove (2). By the assumptions, we have that there exists 
 $$\tilde Y \in [e^{-s' - 10 C_1 \epsilon_1 s'}, \beta e^{-s'}]$$ 
 such that 
  \[ \left|\frac{\tilde Y}{w_{1} \tilde Y + k} - \frac{\tilde Y}{w'_{1}\tilde Y + k'}  \right| \le  \beta^{-100} e^{-3s'},\]
  which implies
  \[ |\tilde Y| \cdot |w_{1} - w'_{1}| \frac{\left| \tilde Y + \frac{k' - k}{w_{1} - w'_{1}} \right|}{|w_{1} \tilde Y + k| \cdot |w'_{1} \tilde Y +  k'|} \le  \beta^{-100}e^{-3s'}. \]
  It is easy to see that $|k'| \in [e^{\ell -1}, e^{\ell +1}]$. 
 \par We claim that 
   \begin{equation}\label{eq:k over w claim}
   \left| \frac{k' - k}{w'_{1} - w_1} \right| \le e^{-s'}.
   \end{equation}
   In fact, if this is not the case, we have, 
   \[ \left|  \frac{k' - k}{w'_{1} - w_1} \right| \ll \left| \tilde Y + \frac{k' - k}{w'_{1} - w_1} \right|, \]
   which implies that 
   \[ \frac{| \tilde Y|\cdot |k'-k|}{|w_1 \tilde Y + k| \cdot |w'_{1} \tilde Y + k'|} \le \beta^{-100} e^{-3s'}. \]
   Thus we have 
   \[ |k' - k| \le e^{-2s' +2\ell + 11 C_1 \epsilon_1 s'}. \]
   Then for any $Y' \in [\beta e^{-s'}]$, we have 
   \begin{align*}
       \left|\frac{Y'}{w_1 Y' + k} - \frac{Y'}{w'_{1}Y' + k'}  \right| &\le \frac{|Y'|\cdot |k'-k|}{|w_1 Y' + k| \cdot |w'_{1} Y' + k'|} \\ 
       &\le e^{-s'} e^{-2s' +2\ell + 11 C_1 \epsilon_1 s'} e^{-2\ell} = e^{-3s' + 11 C_1 \epsilon_1 s'}.
   \end{align*}
   This implies that $C'$ is contained in the $e^{-3s' + 11 C_1 \epsilon_1 s'}$ neighborhood of $C$, which contradicts our assumption. This proves the claim~\eqref{eq:k over w claim}.
  \par From \eqref{eq:w1-w1'} and \eqref{eq:k over w claim} we obtain that for any $Y \in [\beta e^{-s'}]$
  \[ |Y| \cdot |w_{1} - w'_{1}| \frac{\left| Y + \frac{k' - k}{w_{1} - w'_{1}} \right|}{|w_{1} Y + k| \cdot |w'_{1} Y + k'|} \le  A \beta^{-100} e^{-3s'}.
  \]
   This proves that $C'$ is in the $\beta^{-100}A e^{-3s'}$ neighborhood of $C$.
   \par Finally, if $\tilde Y' \in [e^{-s' - 10 C_1 \epsilon_1 s'}, \beta e^{-s'}]$ also satisfies that 
   \[ \left|\frac{\tilde Y'}{w_{1} \tilde Y' + k} - \frac{\tilde Y'}{w'_{1}\tilde Y' + k'}  \right| \le  \beta^{-100} e^{-3s'},\]
   then we have 
   \[ \left| \tilde Y + \frac{k' - k}{w_{1} - w'_{1}} \right| \le  A^{-1} \beta^{-100} e^{-s' + 10 C_1 \epsilon_1 s'},   \]
   and 
   \[ \left| \tilde Y' + \frac{k' - k}{w_{1} - w'_{1}} \right| \le  A^{-1} \beta^{-100} e^{-s' + 10 C_1 \epsilon_1 s'}.  \]
   This implies that 
   \[ | \tilde Y - \tilde Y'| \le A^{-1}  e^{-s' + 11 C_1 \epsilon_1 s'}. \]
   This completes the proof.
 
\end{proof}
\par Let us apply Lemma \ref{lm:hyperbola-plane} to prove the following lemmata:
\begin{lemma}
  \label{lm:integrable-surface}
  For $y \in \fJ\cap \Sigma^r , z \in \cI\cap\Sigma^r$, let us denote 
  \[\cC = \cC(\cT(y), z)\subset \fC =  \left[1.2 \beta e^{-s'}\right]^2.\] 
  Suppose that $\cC$ has equation 
  \[ (X- X_0) (Y- Y_0) + v_1^{-1} (X- X_0) + w_1^{-1} (Y- Y_0) = 0, \]
  with $\| (v_1 , w_1) \| \ge e^{s' - 10 C_1 \epsilon_1 s'}$, and 
  $$\max\{|v_1/w_1|, |w_1/v_1|\} \in [e^\ell, e^{\ell + 1}].$$
  Then for $\ell < s'$, there exist at most $e^{2s' + 2\ell + 60C_1\epsilon_1 s'}$ different $y' \in  \fJ\cap\Sigma^r$ such that 
  \[\cC(\cT(y'), z) \approx \cC.\]
  \par For $\ell \ge s' $, there are at most $e^{4s' + 60 C_1 \epsilon_1 s'}$ different $y' \in \fJ \cap \Sigma^r$ such that 
  \[ \cC(\cT(y'), z) \approx \cC. \]
\end{lemma}

\begin{proof}

 \par We first consider the case $\ell < s'$. For $A > 1$, let us denote by $\cC[A]$ the $Ae^{-3s'}$ neighborhood of $\cC$.
  \par Let us fix $z_0 \in \cI_z$ such that $\fD_z(z_0) = (X'_0, Y'_0) \in \cC[e^{11 C_1 \epsilon_1 s'}]$, and consider all $y' \in \fJ\cap \Sigma^r$ such that 
  \[ \cC \approx \cC(\cT(y'), z). \] 
  Assume in addition that $y'$ is such that there are coefficients $v_{1, y'},w_{1, y'}$ of size $\leq \beta e^{s'}$ and a point $y'_0 \in \Omega(y')$ such that
  \[ z_0 = \exp(v_{1, y'} \vv v_1 + w_{1, y'} \vv w_1) u(r_{y'}) y'_0. \]
  Let us denote by $\cM(\cC, z_0)$ the collection of those $y'$'s. 
  \par Fix a $y \in \cM(\cC,z_0)$ and consider the curve neighborhood $\cC'$ with equation  
  \[ (X- X'_0) (Y- Y'_0) + v_{1,y}^{-1} (X- X'_0) + w_{1,y}^{-1} (Y- Y'_0) = 0. \]
  Then it is easy to see that for any $y' \in \cM(\cC, z_0)$, \[\cC(\cT(y'), z) \subset \cC'[2e^{11 C_1 \epsilon_1 s'}].\]
   \par Without loss of generality, let us assume that $(X'_0, Y'_0) = \vv 0$. Then the equation reads
  \[ XY + v_{1,y}^{-1} X + w_{1,y}^{-1} Y = 0. \]
  Then for any $y' \in \cM(\cC, z_0)$, if $y_0'$ is the corresponding point satisfying
  \[ z_0 = \exp(v_{1, y'} \vv v_1 + w_{1, y'} \vv w_1) u(r_{y'}) y_0', \]
  we get that the curve
  \begin{equation}
  \label{eq:prime curve}
  XY + v_{1, y'}^{-1} X + w_{1, y'}^{-1} Y = 0
  \end{equation}
 is contained in $ 2e^{-3s' + 10 C_1 \epsilon_1 s'}$ neighborhood of the curve 
  \begin{equation}
  \label{eq:not prime curve}XY + v_{1,y}^{-1} X + w_{1,y}^{-1} Y = 0.
  \end{equation}
  Let us assume that $|w_{1,y}| \ge |v_{1,y}|$. The opposite case can be handled similarly. 
  \par Let us denote 
  \[ k = \frac{w_{1,y}}{v_{1,y}}, \text{ and } k' = \frac{w_{1, y'}}{v_{1, y'}}. \]
  Then it is easy to see that $k \in [e^\ell, e^{\ell +1}]$ and $ k' \in [e^{\ell-1}, e^{\ell +2}]$. Note that the equations of the curves~\eqref{eq:prime curve} and~\eqref{eq:not prime curve} can be written as 
  \[ X = \frac{-Y}{w_{1,y} Y + k}, \text{ and } X = \frac{-Y}{w_{1,y'}Y + k'} \]
  respectively.
  Then we have for all $Y \in [\beta e^{-s'}]$,
  \[ \left|\frac{Y}{w_{1,y} Y + k} - \frac{Y}{w_{1,y'}Y + k'}  \right| \le  2e^{-3s'+11C_1\epsilon_1 s'}.\]
  By Lemma \ref{lm:hyperbola-plane}(1) with $A =2 e^{11C_1\epsilon_1 s'}$, we have that 
  \[ |w_{1,y} - w_{1, y'}| \le 2e^{-s' + 11C_1 \epsilon_1 s'+2\ell}, \text{ and } |k - k'| \le 2e^{-2s' + 11C_1 \epsilon_1 s'+2\ell}. \]

Then there are at most $ e^{40 C_1 \epsilon_1 s' + 2\ell}$ different $y' \in \fJ \cap \Sigma^r$
  contained in $\cM(\cC, z_0)$. By running over all $z_0 \in \cI_z$ such that $\fD_z(z_0) \in \cC[e^{11C_1 \epsilon_1 s'}]$, we have that 
 there are at most 
 \[ e^{2s' + 2\ell + 60 C_1 \epsilon_1 s'} \] 
 different $y' \in \fJ\cap \Sigma^r$ such that $\cC(\cT(y'), z) \approx \cC$. This proves the statement for $\ell < s'$.
\par Now let us turn to the case $\ell \ge s'$. In this case, by Remark \ref{rmk:concentrated-surface}, we can consider the corresponding straight lines 
\[ X = -Y/k, \text{ and } X = - Y/k'. \]
Then we have for any $y \in [\beta e^{-s'}]$,
\[ \left| Y/k - Y/k' \right| \le 2e^{-3s'+11C_1 \epsilon_1 s'}, \]
which implies that 
\[  |k - k'| \cdot |kk'|^{-1}  \le  2e^{-2s' + 11C_1\epsilon_1 s'}. \]
This implies that for any fixed $w_{1, y'}$, the desired $v_{1, y'}$ is lying in 
$$ \frac{w_{1, y'}}{k} + \left[2e^{-s' + 11C_1\epsilon_1 s'}\right].$$ 
Noting that $w_{1, y'} \in [\beta e^{s'}]$, we have that there are at most $ e^{40 C_1 \epsilon_1 s' + 2s'}$ different $y' \in \fJ \cap \Sigma^r$
  contained in $\cM(\cC, z_0)$. By running over all $z_0 \in \cI_z$ such that $\fD_z(z_0) \in \cC[e^{11C_1 \epsilon_1 s'}]$, we have that 
 there are at most 
 \[ e^{4s' + 60 C_1 \epsilon_1 s'} \] 
 different $y' \in \fJ\cap \Sigma^r$ such that $\cC(\cT(y'), z) = \cC$.
 \par This completes the proof.
\end{proof}

\begin{lemma}
    \label{lm:determine-curve-nbhd}
    \par Let us fix constants $\ell, m \ge 1$. 
    Let $\cC \subset \fC$ be the $\beta^{-100} e^{-3s'}$-neighborhood of the curve 
\[ (X- X_0) (Y- Y_0) + v_1^{-1} (X- X_0) + w_1^{-1} (Y- Y_0) = 0, \]
with $\|(v_1, w_1)\| \ge e^{s' - 10C_1 \epsilon_1 s'}$ and 
\[ \max\{|v_1/w_1|, |w_1/v_1|\} \in [e^\ell, e^{\ell + 1}].\]
Let $\cC'$ be the $\beta^{-100} e^{-3s'}$-neighborhood of the curve with the following equation
\[(X- X_0) (Y- Y_0) +  {v'_1}^{-1}(X- X_0) +  {w'_1}^{-1}(Y- Y_0) = 0.\]
Suppose we have that $\cC$ and $\cC'$ are not comparable. Moreover, suppose we have the following:
\[ \max\{|v'_1/w'_1|, |w'_1/v'_1|\} \in [e^\ell, e^{\ell + 1}],\]
\[ \max\{ |v'_1 - v_1|, |w'_1 - w_1| \} \in [e^{s'-m}, e^{s'-m+1}], \]
and that there exists $(X_1, Y_1) \in \cC \cap \cC'$ with 
\[ \max\{ |X_0 - X_1|, |Y_0 - Y_1| \} \ge e^{-s' - 10 C_1 \epsilon_1 s'}. \]
Then $m + 2 \ell \le 2s' - 10 C_1 \epsilon_1 s'$, and $\cC'$ is contained in the $e^{-s'-m-2\ell }$ neighborhood of $\cC$.
\end{lemma}
\begin{proof}
   \par Without loss of generality, we can assume that $(X_0, Y_0) = \vv 0$ and $|w_1/v_1| \in [e^\ell, e^{\ell+1}]$. Then we have $|w'_1/v'_1| \in [e^\ell, e^{\ell+1}]$, $|w_1 - w'_1| \in [e^{s'-m}, e^{s'-m+1}]$, and there exists $(X_1, Y_1) \in \cC \cap \cC'$ with $|Y_1| \ge e^{-s'-10C_1 \epsilon_1 s'}$.
   \par Let us rewrite the equation of $\cC$ as 
   \[ X = \frac{-Y}{w_1 Y + k}, \text{ where } k = w_1/ v_1, \]
   and the equation of $\cC'$ as 
   \[ X = \frac{-Y}{w'_1 Y + k'}, \text{ where } k' = w'_1/ v'_1. \]
   Since $(X_1, Y_1) \in \cC \cap \cC'$, we have 
   \[ \left|\frac{Y_1}{w_1 Y_1 + k} - \frac{Y_1}{w'_{1}Y_1 + k'}  \right| \le \beta^{-100}e^{-3s'}. \]
   
  By Lemma \ref{lm:hyperbola-plane}(2) with $A = e^{2s' - 2\ell -m}$ we have that $\cC'$ is contained in the $e^{-s' - m - 2\ell}$-neighborhood of $\cC$. Recalling that by assumption $\cC'$ and $\cC$ are not comparable, we conclude that
  \[ m+ 2 \ell \le 2s' - 10 C_1 \epsilon_1 s'. \]
\end{proof}
\begin{lemma}
\label{lm:counting-curves-nbhd}
\par Let us fix constants $m, \ell \ge 1$ with $m+2\ell \le 2s' - 10 C_1 \epsilon_1 s'$.
Let $\cC \subset \fC$ be the $\beta^{-100} e^{-3s'}$-neighborhood of the curve 
\[ (X- X_0) (Y- Y_0) + v_1^{-1} (X- X_0) + w_1^{-1} (Y- Y_0) = 0, \]
with $\|(v_1, w_1)\| \ge e^{s' - 10C_1 \epsilon_1 s'}$ and 
\[ \max\{|v_1/w_1|, |w_1/v_1|\} \in [e^\ell, e^{\ell + 1}].\]
Let $\fN( \cC, (X_0, Y_0), m)$ denote a collection of non-comparable $\beta^{-100} e^{-3s'}$-curve neighborhoods with equations of the form
\[(X- X_0) (Y- Y_0) +  {v'_1}^{-1} (X- X_0) +  {w'_1}^{-1}(Y- Y_0) = 0,\]
and contained in the $e^{-s'-m-2\ell}$ neighborhood of $\cC$. Then 
\[ |\fN( \cC, (X_0, Y_0), m) | \le e^{4s' - 2m - 4 \ell + 100 C_1 \epsilon_1 s'}.\]
Moreover, for $\cC' \in \fN( \cC, (X_0, Y_0), m)$ with corresponding parameters $v'_1,w'_1$, let $m'$ denote the number satisfying
\[ \max \{|w'_1 - w_1|, |v'_1 - v_1| \} \in [e^{s'-m'}, e^{s'-m'+1}]. \]
Then the intersection
$$\cC\cap \cC' \cap \left\{ (X, Y): \|(X, Y) - (X_0, Y_0)\| \ge e^{-s' - 10 C_1 \epsilon_1 s'} \right\}$$
contains at most $e^{m'+ 2\ell + 11 C_1 \epsilon_1 s'}$ different points from the set 
\[ \left\{ \fD_z(z') : z' \in \cI_z \right\}. \]

\end{lemma}
\begin{proof}
\par Without loss of generality, we can assume that $(X_0, Y_0)= \vv 0$, and $|w_1| \ge |v_1|$.
\par Let us write the equation of $\cC$ as
\[ X = \frac{-Y}{w_1 Y + k}, \text{ where } k = w_1/ v_1. \]
Then $|k| \in [e^\ell, e^{\ell +1}]$. Suppose $\cC'$ with equation
\[ X = \frac{-Y}{w'_1 Y + k'}  \text{ where } k' = w'_1/v'_1 \]
is in $\fN(\cC, \vv 0, m)$. Then we have for all $Y \in [ \beta e^{-s'}]$,
\[ \left|\frac{Y}{w_1 Y + k} - \frac{Y}{w'_{1}Y + k'}  \right| \le  e^{-s' - m - 2\ell}.\]
By Lemma \ref{lm:hyperbola-plane} (1) with $A = e^{2s' - m - 2 \ell}$, we have
\begin{equation} 
\label{eq:counting-conditions}
|w_1 - w'_{1}| \le  e^{s' -m }, \text{ and } |k - k'| \le e^{-m }. \end{equation}

Note that by Lemma \ref{lm:hyperbola-plane} (1) with $A = e^{11C_1 \epsilon_1 s'}$, if 
$$|w'_1 - \bar w'_1| \le  e^{-s'+ 11C_1 \epsilon_1 s'+2\ell}$$ 
and 
$$|k' - \bar k'| \le  e^{-2s'+ 11C_1 \epsilon_1 s'+2\ell},$$
the corresponding curves 
\[ X = \frac{-Y}{w'_1 Y + k'} \]
and 
\[ X =  \frac{-Y}{\bar w'_1 Y + \bar k'} \]
are comparable.
Thus, the conditions \eqref{eq:counting-conditions} give at most 
\[ e^{4s' - 2m - 4\ell} \]
different curve neighborhoods, which are not comparable to each other. 
\par Let us prove the last statement. By Lemma \ref{lm:hyperbola-plane} (2) with $A = e^{2s' -2\ell -m'}$, we have that there exists 
$(X_1, Y_1) \in \cC$ such that the intersection
\[\cC\cap \cC' \cap \left\{ (X, Y): \|(X, Y) - (X_0, Y_0)\| \ge e^{-s' - 10 C_1 \epsilon_1 s'} \right\}\]
is contained in the $e^{-3s' + m' +2\ell + 11 C_1 \epsilon_1 s'}$-neighborhood of $(X_1, Y_1)$, denoted by $B(X_1, Y_1)$.
Then the statement follows from the fact that the intersection 
\[ \cC \cap B(X_1, Y_1) \]
contains at most $e^{m' + 2\ell + 11 C_1 \epsilon_1 s'}$ different elements from 
\[ \left\{ \fD_z(z') : z' \in \cI_z \right\}. \]
This completes the proof.
\end{proof}
We also need the following lemma:
\begin{lemma}
    \label{lm:curve-in-N}
    Given $(X_0, Y_0) \in \fC = \left[1.2\beta e^{-s'}\right]^2$, and $1 < \ell \le s'$, let us denote 
    \[ \cN^\ell = \cN^{\ell}(X_0, Y_0) := (X_0, Y_0) + [\beta e^{-s' - \ell}] \times [\beta e^{-s'}]. \]
Suppose we have a cross-section $\Sigma^r$ and $z \in \cI \cap \Sigma^r$, $y \in \fJ\cap\Sigma^r$ such that 
$$\cC = \cC(\cT(y),z) \subset \cN^{\ell}.$$ 
Then there exists $z' \in \Theta(z)$ only depending on $\cN^\ell$ (and not on $\cC$) such that 
$\fD_z(z') \in \cN^{\ell}$ and 
     \[ \Omega'(y) \subset a(2s')Q(\fr_1, \ell-1, \beta ) B^{A_0 H}(\beta) a(-2s') z'. \]
\end{lemma}

\noindent 
Cf.~ equation \eqref{eq:def-thin} for the definition of~$Q(\fr_1, \ell, \beta )$.

\begin{proof}
    \par First, we note that there are $\tilde y \in \Omega(y)$ and $\tilde z \in \Theta(z)$ such that 
    \begin{equation} 
    \label{eq:tildey-tildez}
    \tilde y =u(r) \exp(\vv w^+)  \tilde z, \end{equation}
    and 
    \[ \tilde z = \exp(\vv w^-) z \]
    where $\vv w^+ = v_1 \vv v_1 + w_1 \vv w_1 \in B_{\vv r^+}(\beta e^{s'})$, $r \in [\beta e^{2s'}]$, and 
    $$\vv w^- =  X_1 \vv v_2 + Y_1 \vv w_2 \in \vv{r}^-.$$
    By Lemma \ref{lm:concentrated-surface}, we have that $\cC = \cC(\cT(y), z)$ is the $\beta^{-100} e^{-3s'}$-neighborhood of the following curve ($v_1,w_1$ as above):
    \[ (X-X_1)(Y - Y_1) + v_1^{-1} (X-X_1) + w_1^{-1} (Y-Y_1) = 0.\]
    Since $\cC \subset \cN^\ell$, we have that 
    \[ (X_1, Y_1) \in \cN^\ell.\]
    Moreover, for any $Y \in [\beta e^{-s'}]$, we have
    \[ |X - X_1| = \frac{|Y-Y_1|}{|w_1 (Y-Y_1) + k|} \le \beta e^{-s' - \ell},   \]
    where $k = w_1/v_1$. Note that there exists $Y \in [\beta e^{-s'}]$ such that $|Y- Y_1| \ge \beta e^{-s'}$, we get 
    \[ |w_1 (Y-Y_1) + k| \ge e^{\ell}. \]
    The above bounds on $w_1$, $|Y-Y_1|$ imply that $|w_1(Y-Y_1)| \le \beta^2$, hence
    \[ |k| = |w_1/v_1| \ge e^{\ell}. \]
    Therefore, we have 
    \begin{equation}\label{eq:pv1}
    |v_1| \le e^{-\ell} |w_1| \le \beta e^{s' - \ell}. 
    \end{equation}
    \par Let $z' \in \Theta(z)$ be the following element:
    \[ z' = \exp(\vv w_0^-) z,\]
    where $\vv w_0^- = X_0 \vv v_2 + Y_0 \vv w_2$. Then we have that 
    \begin{equation}
    \label{eq:tildez-z'}
    \tilde z =  \exp(\tilde{\vv w}^-) u^\ast (r^\ast)z' \end{equation}
    where $|r^\ast| \le \beta e^{-2s'}$ and $\tilde{\vv w}^- = \tilde X \vv v_2 + \tilde Y \vv w_2 $ with $(\tilde X, \tilde Y) \in \cN^\ell (\vv 0)$. 
    \par By \eqref{eq:tildey-tildez} and \eqref{eq:tildez-z'}, we get 
    \[ \tilde y =u(r) \exp(\vv w^+) \exp(\tilde{\vv w}^-) u^\ast (r^\ast)z'  \]
    where $\vv w^+ \in B_{\vv r^+}(\beta e^{s'})$ with $\|p_{\vv v_1} (\vv w^+)\| \le \beta e^{s' - \ell}$ by \eqref{eq:pv1}, and $\tilde{\vv w}^- \in B_{\vv r^-}(\beta e^{-s'})$ with $\|p_{\vv v_2} (\tilde{\vv w}^-)\| \le \beta e^{-s' - \ell}$ since $(\tilde X, \tilde Y) \in \cN^\ell (\vv 0)$. Then we can write 
    \[ \tilde y = a(2s' ) u(\bar r) \exp(\bar{\vv w}^+) \exp(\bar{\vv w}^-) u^\ast (\bar r^\ast) a(-2s')z', \]
    where $|\bar r|, |\bar r^\ast| \le \beta$, $\bar{\vv w}^+ \in B_{\vv r^+}(\beta)$ with $\|p_{\vv v_1} (\bar{\vv w}^+)\| \le \beta e^{-\ell}$, and  $\bar{\vv w}^- \in B_{\vv r^-}(\beta)$ with $\|p_{\vv v_2} (\bar{\vv w}^-)\| \le \beta e^{-\ell}$. Note that these estimates imply that 
    \[ \|[\bar{\vv w}^-, \bar{\vv w}^+]\| \le \beta^2 e^{-\ell}. \]
By Lemma \ref{lm:product-w}, we have 
\[ \exp(\vv w^+) \exp(\tilde{\vv w}^-) = \exp(\bar{\vv w}) \exp(\bar{\vv h}), \]
where $\bar{\vv w} \in B_{\fr_1 + \fr_2} (\beta)$ with $\|p_{\vv r_1}(\bar{\vv w})\| \le \beta e^{-\ell}$, and $\bar{\vv h } \in B_{\fh + \R \fa_0} (\beta e^{-\ell})$. 
Thus, we have
\begin{align*} 
\tilde y &=  a(2s' ) u(\bar r) \exp(\bar{\vv w}) \exp(\bar{\vv h}) u^\ast (\bar r^\ast) a(-2s')z' \\
         &\in a(2s')Q(\fr_1, \ell, \beta ) B^{A_0 H}(\beta) a(-2s') z',
\end{align*}
which implies that 
\[ \Omega'(y) \subset a(2s')Q(\fr_1, \ell-1, \beta ) B^{A_0 H}(\beta) a(-2s') z'. \]
\end{proof}
\begin{remark}
    \label{rmk:sl2r2}
Lemma \ref{lm:curve-in-N} roughly says that given $z \in \cI\cap \Sigma^r$ and 
$$ \cN^\ell = \cN^\ell(X_0, Y_0) \subset \fC$$ 
as above, 
all $y \in \fJ\cap \Sigma^r$ satisfying that 
\[ \cC(\cT(y), z) \subset \cN^\ell\] 
is contained in a $a(2s')$-translate of a $\beta e^{-\ell+1}$-neighborhood of a piece of some $\SL(2,\R)\ltimes \R^2$-orbit of size $\beta$.
\end{remark}

\par Now we are now ready to prove Proposition \ref{prop:dimension-improvement-structured-component}
\begin{proof}[Proof of Proposition \ref{prop:dimension-improvement-structured-component}]
\par Assume in contradiction that 
  \[ \tilde \mu (\bar \cR^{\fs}_i) > \beta^{22}. \]
   Let $\fZ \subset \Pi(x)$ be the subset from Proposition \ref{prop:dimension-does-not-drop}. 
   By replacing $\bar \cR^{\fs}_i$ with $\bar \cR^{\fs}_i \setminus \mathfrak{Z}$, we can assume 
   that \eqref{eq:upper-theta} holds for any $z \in \cI \cap \bar\cR^{\fs}_i$. 
   
   Thus we may assume:
    \begin{equation} 
    \begin{gathered}
        \cI \cap \bar \cR^{\fs}_i \neq \emptyset \\       \tilde \mu ( \Theta'(z)) \le  e^{-2s'} e^{-(d_i - C_1 \epsilon_1) s'} \qquad \text{for any $z \in \cI \cap \bar \cR^{\fs}_i$}
    \end{gathered}
    \label{eq:upper-theta-2}
 \end{equation} 
 {
   
  \par Let us fix a $z \in \cI\cap \bar \cR^{\fs}_i$. By the definition of $\bar \cR^{\fs}_i$ (see Definition \ref{def:structure-random-decompostion}), we have that $\Theta'(z)$ is highly structured. Thus, there exists $r=r(z)$ such that $z \in \Sigma^r$ and 
  \begin{equation}\label{eq:SSzr}
      \tilde \mu (\SS(z,r)) \ge  e^{-2s' - 10 C_1 \epsilon_1 s'}
  \end{equation} 
  Recall that by \eqref{eq:upper-omega}, for any $y$,
  \[
      \tilde \mu(\Omega'(y)) \le \beta^{-15}e^{-2 s'} e^{-2 d_i s'}.
  \]
Recall also that $\SS(z,r)$ is the union of $\Omega'(y)$ with $y \in \HS(z,r)$, i.e.\ those $y\in\fJ\cap\Sigma^r$ for which $\cC(\cT(y), z)$ is highly concentrated. Then by \eqref{eq:SSzr} and \eqref{eq:upper-omega}, we have the following:
\begin{itemize}
    \item[$\clubsuit$] there exist at least $e^{2d_i s' -15C_1 \epsilon_1 s'}$ elements $y$ in $\HS(z,r)$.
\end{itemize}

 \noindent
  We say that a highly concentrated curve neighborhood $\cC=\cC(\cT(y), z)$ \emph{covers}~$\Omega'(z')$ ($z'\in \cI_z$) if $\fD_z(z') \in \cC$. 
  \par To get a contradiction to \eqref{eq:upper-theta-2}, it is enough to show that if $\clubsuit$ holds, the union
  \begin{equation}\label{eq:HSunion}
  \bigcup_{y \in \HS(z,r)}\cC(\cT(y),z)
  \end{equation}
  (which by definition corresponds to a subset of $\Theta'(z)$) must cover 
  at least $e^{d_i s' + 20 C_1 \epsilon_1 s'}$ different full $\Omega'(z')$ with $z' \in\cI_z$. 
  Indeed, if this holds, then by the definition of full neighborhoods in~\eqref{eq:full-omega}, we have that 
  \[ \tilde \mu(\Theta'(z)) \ge e^{-2s'} e^{-(d_i - 12 C_1 \epsilon_1) s'}, \]
  which contradicts \eqref{eq:upper-theta-2}.
  \medskip
  
  \par Now we aim to show that under the assumption $\clubsuit$, 
  \begin{equation}\label{eq:nowweaim}
       \bigcup_{y \in \HS(z,r)}\!\!\!\cC(\cT(y),z)\ \ \text{covers $\geq e^{d_i s' + 20 C_1 \epsilon_1 s'}$ different full $\Omega'(z')$,  $z' \in \cI_z$.}
  \end{equation}
  
   Given a highly concentrated $\cC$ with equation
  \[ (X- X_0) (Y- Y_0) +   v_1^{-1}(X- X_0) +  w_1^{-1}(Y- Y_0) = 0, \]
  let us denote $k(\cC) := \max\{ |w_1/v_1|, |v_1/w_1| \}$.
  By dyadic pigeonholing, we have that there is some $\ell$ for which there are at least $e^{2d_i s' -16C_1 \epsilon_1 s'}$ different $y \in \fJ\cap\Sigma^r$ such that $\cC = \cC(\cT(y), z)$ satisfies $k(\cC) \in [e^\ell, e^{\ell+1}]$.  We divide the proof into two cases.

  \subsection*{The case \texorpdfstring{$\ell\geq s'$}{l<=s'}} 
By Remark \ref{rmk:concentrated-surface}.\eqref{item:line}, in this case the highly concentrated curve neighborhoods are neighborhoods of straight lines. Without loss of generality, we can assume that $|w_1/v_1| \in [e^\ell, e^{\ell+1}]$, in which case these highly concentrated curve neighborhoods are neighborhoods of almost vertical straight lines. 
Let $\cC=\cC(\cT(y_0), z)$ be a highly concentrated curve; the definition of highly concentrated implies that $\|(v_1,w_1)\|\geq e^{s'-10C_1\epsilon_1s'} $, hence Lemma \ref{lm:integrable-surface} can be applied. Thus there are at most $e^{4s' + 60 C_1 \epsilon_1 s'}$ different $y \in \fJ\cap \Sigma^r$ such that 
  \[ \cC(\cT(y), z) \approx \cC.\] 
  It follows that we can find a collection $\mathfrak{CT}$ of highly concentrated curve neighborhoods in $\Theta'(z)$ with at least 
  
  $$e^{2d_i s' -16C_1 \epsilon_1 s'}e^{-4s' -60 C_1 \epsilon_1 s'} =   e^{ 2 d_i  s' - 4 s'-76C_1 \epsilon_1 s'}$$ 
  elements such that any two curve neighborhoods from $\mathfrak{CT}$ are not comparable. 

  \par Now let us estimate the number of different full $\Omega'(z')$'s (with $z' \in \cI_z$) covered by highly concentrated curve neighborhoods from $\mathfrak{CT}$. 
  Our goal is to show that there are at least 
  \[e^{d_i s' + 20 C_1 \epsilon_1 s'}\]
  such $\Omega'(z')$'s. 
  \par Note that every highly concentrated curve neighborhood must cover at least $e^{2s' - 9 C_1 \epsilon_1 s'}$ full $\Omega'(z')$'s. Thus, counting with multiplicity, the curve neighborhoods from $\mathfrak{CT} $ cover at least 
  \[ e^{ 2 d_i  s' - 2 s' -85C_1 \epsilon_1 s'} \]
  full $\Omega'(z') \subset \Theta'(z)$. If every full $\Omega'(z')$ is counted with multiplicity at most 
  $$e^{d_i s' - 2 s' -105 C_1 \epsilon_1 s'},$$ 
  then we are done. 
  
  Therefore assume that:
  \begin{itemize}
      \item There exists a full 
  $\Omega'(z_1) \subset \Theta'(z)$ which is counted with multiplicity
  $\geq e^{d_i s' - 2 s' -105 C_1 \epsilon_1 s'}$; let $(X_0, Y_0)=\fD_z(z_1)$.
  \end{itemize}
  Let $\mathfrak{LL}\subset \mathfrak{CT}$ denote the set of highly concentrated curve neighborhoods passing through $(X_0, Y_0)$. Then 
  \[ |\mathfrak{LL}| \ge  e^{d_i s' - 2 s' -105 C_1 \epsilon_1 s'}.\]
  Recalling that in the case we are currently considering, namely of $\ell\geq s'$, these highly concentrated curve neighborhoods are neighborhoods of straight lines, as well as the definition of comparable in Definition~\ref{def:comparable}, we have that any two non-comparable curve neighborhoods $\cC$ and $\cC'$ passing through $(X_0, Y_0)$ do not intersect at any $(X, Y)$ with $|Y - Y_0| \ge e^{-s'-11C_1\epsilon_1 s'}$. Thus, the highly concentrated curve neighborhoods from $\mathfrak{LL}$ cover 
  at least 
  \[
  e^{d_i s' -120 C_1 \epsilon_1 s'}
  \]
  different full $\Omega'(z'')$'s. 
  Moreover, all these highly concentrated curve neighborhoods are contained in 
  $$\cN^{s'}_1=\cN^{s'}(X_0, Y_0) = (X_0, Y_0) + [\beta e^{-2s'}]\times [\beta e^{-s'}].$$
  \par Similarly to the definition of $\HS(z,r)$ and $\SS(z,r)$ (cf.\ 
  \eqref{eq:structure-part-theta-0} and \eqref{eq:structure-set-theta-0}), 
  let us define 
  \begin{equation}
  \label{eq:def-hszrn}
  \HS(z, r; \cN^{s'}_1) := \left\{ y \in \fJ\cap \Sigma^r: \cC(\cT(y), z) \subset \cN^{s'}_1 \ \text{and is highly concentrated}\right\}, \end{equation}
  and 
  \begin{equation} 
  \label{eq:def-sszrn}
  \SS(z, r; \cN^{s'}_1) := \bigcup_{y \in \HS(z, r; \cN^{s'}_1 )} \Omega'(y). \end{equation}
  We claim that 
  \begin{equation} 
  \label{eq:sszrN-small}
  \tilde \mu(\SS(z,r; \cN^{s'}_{1})) \le e^{-2s'}e^{-800 C_1 \epsilon_1 s'}.  \end{equation}
  Indeed, by Lemma \ref{lm:curve-in-N}, there exists $z' \in \Theta'(z)$ so that the following holds: for any $y \in  \HS(z, r; \cN_1^{s'} )$,  we have
  \[ \Omega'(y) \subset a(2s')Q(\fr_1, s'-1, \beta ) B^{A_0 H}(\beta) \tilde z,  \]
  where $\tilde z = a(-2s')z'$.
  Therefore, 
  \[ \SS(z, r; \cN_{1}^{s'}) \subset \left(a(2s')Q(\fr_1, s'-1, \beta ) B^{A_0 H}(\beta) \tilde z \right)\cap \Sigma^r. \]
  By our discussion at the beginning of \S \ref{sec-reduction}, Proposition \ref{prop:sl2xr2-closing-lemma} holds for $\mu_{\cF_i}$. Thus, we have 
  \begin{align*} 
  \tilde \mu\left(a(2s')Q(\fr_1, s'-1, \beta ) B^{A_0 H}(\beta) \tilde z\right) &\le \beta^{-15} \mu_{\cF_i} \left(Q(\fr_1, s'-1, \beta) B^{A_0 H}(\beta) \tilde z\right) \\ 
  &\le \beta^{-15}e^{-d_0 (s'-1) /2} < e^{-800 C_1 \epsilon_1 s'}.\end{align*}
  Note that 
  \[  \left(a(2s')Q(\fr_1, s'-1, \beta ) B^{A_0 H}(\beta) \tilde z\right) \cap \Sigma^r \]
  is a cross-section of $a(2s')Q(\fr_1, s'-1, \beta ) B^{A_0 H}(\beta) \tilde z$ along the $U$-direction. It follows that 
  \begin{align*}
      \tilde \mu\left(\SS(z,r; \cN_{1}^{s'})\right) &\le \tilde \mu\left(\left(a(2s')Q(\fr_1, s'-1, \beta ) B^{A_0 H}(\beta) \tilde z\right) \cap \Sigma^r\right) \\ 
        &\le e^{-2s'} \tilde \mu\left(a(2s')Q(\fr_1, s'-1, \beta ) B^{A_0 H}(\beta) \tilde z\right) \\ &\le e^{-2s'}e^{-800 C_1 \epsilon_1 s'}.
  \end{align*}
  This proves \eqref{eq:sszrN-small}. Thus, $\clubsuit$ holds with $\HS(z,r)$ replaced by 
  $$\HS(z,r) \setminus \HS(z, r;\cN^{s'}_{1}).$$ 
  By repeating the above argument with $\HS(z,r)$ replaced by 
  $$\HS(z,r) \setminus \HS(z, r;\cN^{s'}_{1}),$$ 
  we will get $\cN_{2}^{s'}$ disjoint with $\cN_{1}^{s'}$, and at least $e^{d_i s' -120 C_1 \epsilon_1 s'}$ 
  different $\Omega(z'')$'s with $ \fD_z(z'')\in \cN_{2}^{s'}$. This argument can be repeated for at least $e^{400 C_1 \epsilon_1 s'}$ times.  This gives at least 
  $$e^{d_i s'-120 C_1 \epsilon_1 s'} e^{400 C_1 \epsilon_1 s'} \ge  e^{d_i s' + 20 C_1 \epsilon_1 s'}$$
  different full $\Omega'(z')$'s. This establishes \eqref{eq:nowweaim}, completing the proof for the case $\ell \ge s'$.

\subsection*{The case \texorpdfstring{$\ell < s'$}{l<s'}} 
By Lemma \ref{lm:integrable-surface}, we have that for every highly concentrated curve neighborhood $\cC$, there exist at most 
  $$e^{2s' + 2\ell + 60C_1 \epsilon_1 s'}$$ 
  different $y \in \fJ\cap \Sigma^r$ such that 
  \[ \cC(\cT(y), z) \approx \cC. \]
  Therefore, we can choose a collection $\mathfrak{CT}$ of highly concentrated curve neighborhoods in $\Theta'(z)$ such that  
  \[ |\mathfrak{CT}| \ge e^{ 2 d_i  s'-16C_1 \epsilon_1 s'}  e^{-2s' - 2\ell - 60C_1 \epsilon_1 s'} = e^{(2d_i - 2)  s' - 2\ell - 76 C_1 \epsilon_1 s'},\]
  and any two elements from $\mathfrak{CT}$ are not comparable. Note that every highly concentrated curve neighborhood covers at least $$e^{2s' - 9C_1 \epsilon_1 s'}$$ 
  different full $\Omega'(z')$'s. Thus, counting multiplicity, highly concentrated curve neighborhoods from $\mathfrak{CT}$ cover at least $$e^{2d_i s' - 2 \ell -85 C_1 \epsilon_1 s'}$$ 
  full $\Omega'(z')$'s. 
  \par Our goal is to show 
  \begin{itemize}
      \item[(GOAL)] The curve neighborhoods from $\mathfrak{CT}$ cover at least $e^{d_i s' + 20 C_1 \epsilon_1 s'}$
  different full $\Omega'(z')$. 
  \end{itemize}

  \par If every full $\Omega'(z')$ is covered by at most 
  \[ e^{ d_i s' - 2 \ell-105C_1 \epsilon_1 s'} \]
  different highly concentrated curve neighborhoods, then we are done. Now let us assume:
  \begin{itemize}
      \item there exists a full $\Omega'(z_0)$ which is covered by 
  at least 
  $  e^{ d_i s' - 2 \ell-105C_1 \epsilon_1 s'}$
  different highly concentrated curve neighborhoods, and let 
  $(X_0, Y_0)=\fD_{z}(z_0)$.
  \end{itemize}
  Let $\mathfrak{LL}\subset \mathfrak{CT}$ denote the set of highly concentrated curve neighborhoods passing through $(X_0, Y_0)$.
  \par Now let us consider the following process:
  \par Step $0$: We start with $\mathfrak{L}_0 = \emptyset$ and $\mathfrak{P}_0 = \emptyset$.
  \par Step $j+1$: Suppose $\mathfrak{L}_j$ and $\mathfrak{P}_j$ have been constructed in Step $j$. 
  We pick a highly concentrated curve neighborhood $\cC \in \mathfrak{LL}$ such that 
  the number of full $\Omega'(z')$ contained in $\cC$ but not contained in $\mathfrak{P}_j$ is 
  $\ge \frac12 e^{2s' - 9C_1\epsilon_1 s'}$. Then we define 
  \[ \mathfrak{L}_{j+1} = \mathfrak{L}_{j} \cup \{\cC\}, \text{ and } \mathfrak{P}_{j+1} = \mathfrak{P}_j \cup \cC. \] 
  If we can not pick such a curve neighborhood, the process stops. 
  \par When the process stops, say at step $n$, 
  let us take any $\cC \in \mathfrak{LL} \setminus \mathfrak{L}_n$. 
  Then by how the $\mathfrak L_j$ and $\mathfrak P_j$ were constructed, we have that at least 
  $$\tfrac12 e^{2s' - 9 C_1 \epsilon_1 s'}$$
  full $\Omega'(z')$ in $\cC$ are covered by curve neighborhoods from $\mathfrak{L}_n$. Let the equation of $\cC$ be 
  \[ (X - X_0 )(Y - Y_0) + v_1^{-1} (X - X_0) + w_1^{-1} (Y - Y_0) = 0. \]
  Using dyadic pigeonholing, we can assume that there exists $m_{\cC} >0$, such that at least 
  $$e^{2s' - 9.1\, C_1 \epsilon_1 s'}$$ 
  full $\Omega'(z')$ in $\cC$ are covered by curve neighborhoods from 
  $\mathfrak{L}_n$ whose equation 
  \[ (X - X_0) (Y - Y_0) + {v'_1}^{-1} (X - X_0) + {w'_1}^{-1} (Y - Y_0) = 0, \]
  satisfies that 
  \begin{equation}\label{eq:Lnm}
      |w_1 - w'_1| \in [e^{s' - m_{\cC}}, e^{s' - m_{\cC}+1}].
  \end{equation} 
  We
  take $m_{\cC}$ to be the minimal natural number satisfying this condition. Let us denote the collection of the curve neighborhoods satisfying \eqref{eq:Lnm} by $\mathfrak{L}_n(\cC, m_{\cC})$. We also remove $\cC'$ from $\mathfrak{L}_n(\cC, m_{\cC})$ if the intersection $\cC' \cap \cC$ is contained in the $e^{-s'-10C_1 \epsilon_1 s'}$-neighborhood of $(X_0, Y_0)$.
  By Lemma \ref{lm:determine-curve-nbhd}, we have that the curve neighborhoods from $\mathfrak{L}_n(\cC, m_{\cC})$ are contained 
  in the $e^{-s' - m_{\cC} - 2\ell}$-neighborhood of $\cC$. 
  Let us call this curve neighborhood the 
  \emph{$(\ell, m_{\cC})$-neighborhood of $\cC$}. By the last statement of Lemma \ref{lm:counting-curves-nbhd}, we have that each curve neighborhood from $\mathfrak{L}_n(\cC, m_{\cC})$ can cover at most 
  $$e^{m_{\cC} + 2 \ell + 11C_1 \epsilon_1 s'}$$ 
  different full $\Omega'(z')$ contained in $\cC$. 
  Therefore, we have that
  $$ | \mathfrak{L}_n(\cC, m_{\cC})| \ge e^{2s' - m_{\cC} - 2 \ell - 21 C_1 \epsilon_1 s'}.$$ 
   By the construction of $\mathfrak L_n$, every $ \cC' \in \mathfrak{L}_n(\cC, m_{\cC})$ covers at least 
   \[ e^{2s' - 10 C_1 \epsilon_1 s'} \]
   different full $\Omega'(z')$ which are not covered by any other $\cC'' \in \mathfrak{L}_n(\cC, m_{\cC})$.
   Thus, curve neighborhoods from $\mathfrak{L}_n(\cC, m_{\cC})$ cover at least 
  \begin{equation}\label{eq:curve neighborhoods cover}
      e^{4s' - m_{\cC} - 2 \ell - 31 C_1 \epsilon_1 s'}
  \end{equation}
  different full $\Omega'(z')$ in the $(\ell, m_{\cC})$-neighborhood of $\cC$. 
  We say that a $(\ell, m')$-neighborhood is \emph{$\eta$-almost filled} if it covers at least 
  $\eta e^{4s' - m' - 2 \ell}$
  different full $\Omega'(z')$. 
  By repeating the above argument for every curve neighborhood $\cC \in \mathfrak{LL} \setminus \fL_n$, we have that the $(\ell, m_{\cC})$-neighborhood of every curve neighborhood $\cC \in \mathfrak{LL}\setminus \mathfrak{L}_n$ is $e^{-31C_1 \epsilon_1 s'}$-almost filled.
  
  Again by dyadic pigeonholing, we can find a $m$ such that for at least 
  $\beta$-proportion of the curve neighborhoods $\cC \in \mathfrak{LL}$ such that $m_{\cC} = m$. Let $\tilde \fN$ denote the collection of the $(\ell, m)$-neighborhoods obtained as above.
  \par Given a $(\ell, m)$-neighborhood $\tilde \cC$, let us denote $L(\tilde \cC)$ the 
  $$e^{-s' - m - 2\ell + C_1 \epsilon_1 s'}$$ 
  neighborhood of $\tilde \cC$ and call it the \emph{enlarged $(\ell, m)$-neighborhood} of $\tilde \cC$. 
  Note that if two $(\ell, m)$-neighborhoods $\tilde \cC_1, \tilde \cC_2$ both contain 
  a curve neighborhood $\cC$ from $\fL_n$, then we have
  \[ \tilde \cC_1 \subset L(\tilde \cC_2). \]
  Let us select a subset $\mathfrak{N} \subset \tilde \fN$ satisfying the following:
  \begin{itemize}
      \item For any $\tilde \cC_1, \tilde \cC_2 \subset \fN$,  \[ \tilde \cC_1 \not\subset L(\tilde \cC_2). \]
      \item For any $\tilde \cC \in \tilde \fN \setminus \fN$, there exists $\tilde \cC' \in \fN$ such that 
       \[ \tilde \cC \subset L(\tilde \cC'). \]
  \end{itemize}
  Let us consider the curve neighborhoods from $\fL_n$ which are contained in some $(\ell,m)$-neighborhood from $\fN$. We have that those curve neighborhoods cover at least 
  \[ |\mathfrak{N}| e^{4s' - m - 2 \ell - 31 C_1 \epsilon_1 s'} \]
different full $\Omega(z')$. If 
\[ |\mathfrak{N}| \ge e^{d_i s' - 4 s' + m + 2 \ell + 51 C_1 \epsilon_1 s'}, \]
then we are done as this establishes the goal above. Now suppose that 
\begin{equation} 
\label{eq:fN-upper}
|\mathfrak{N}| < e^{d_i s' - 4 s' + m + 2 \ell + 51 C_1 \epsilon_1 s'}. \end{equation}
On the other hand, by the definition of $\fN$, we have that at least $\beta$-proportional curve neighborhoods from $\mathfrak{LL}$ are contained in some enlarged $(\ell, m)$-neighborhood $L(\tilde \cC)$ where $\tilde\cC \in \fN$. By Lemma \ref{lm:counting-curves-nbhd}, every enlarged $(\ell, m)$-neighborhood contains at most 
\[ e^{4s' - 2m - 4\ell + 102 C_1 \epsilon_1 s'}\]
different curve neighborhoods from $\mathfrak{LL}$. We have 
\[ |\mathfrak{N}| e^{4s' - 2m - 4\ell + 102 C_1 \epsilon_1 s' } \ge e^{ d_i s' - 2 \ell - 105C_1 \epsilon_1 s'},\]
which implies that 
\[ |\mathfrak{N}| \ge e^{d_i s' - 4 s' + 2 m + 2 \ell - 207C_1 \epsilon_1 s'}. \]
By \eqref{eq:fN-upper}, we get
\[ m \le 260 C_1 \epsilon_1 s'. \]
Then by \eqref{eq:curve neighborhoods cover}, every $(\ell, m)$-neighborhood contains at least 
\[e^{4s' - 2 \ell - 300 C_1 \epsilon_1 s'}\]
different full $\Omega'(z')$. If $\ell < 10^{10} C_1 \epsilon_1 s' $, then 
it covers at least 
\[ e^{4s' - 3 \times 10^{10} C_1 \epsilon_1 s'} \ge e^{d_i s' + 20 C_1 \epsilon_1 s'} \]
different full $\Omega'(z')$ and we are done.

It remains to deal with the case of $\ell \ge 10^{10} C_1 \epsilon_1 s'$. This case is similar to the conclusion of the $\ell>s'$ case, and to deal with this case we shall need to use Proposition~\ref{prop:sl2xr2-closing-lemma} (which roughly says that our Diophantine assumptions guarantee that $\tilde\mu$ does not concentrate along orbits of $\SL(2,R)\ltimes \R^2$).

According to our assumption, the $(\ell, m)$-neighborhoods from $\fN$ cover at least 
\[ e^{d_i s' - 4 s'  + 2 \ell - 207C_1 \epsilon_1 s'} e^{4s' - 2 \ell -300 C_1 \epsilon_1 s' } = e^{d_i s' - 507 C_1 \epsilon_1 s'}\]
different full $\Omega'(z')$. Similarly to the case where $\ell \ge s' $, we have that 
all $\Omega'(z')$ covered by these $(\ell, m)$-neighborhoods are contained in
\[ \cN^\ell_1 = \cN^\ell (X_0, Y_0) = (X_0, Y_0) + [\beta e^{-s' -\ell}]\times [\beta e^{-s'}]. \]
\par As above, define 
  \[ \HS(z, r; \cN_1^\ell) := \left\{ y \in \fJ\cap \Sigma^r: \cC(\cT(y), z) \subset \cN_1^{\ell} \text{ and is highly concentrated } \right\}, \]
  and 
  \[ \SS(z, r; \cN_1^\ell) := \bigcup_{y \in \HS(z,  r; \cN_1^\ell )} \Omega'(y). \]
  We claim that 
  \[ \tilde \mu\left(\SS(z,  r; \cN_1^\ell)\right) \le e^{-2s'-800 C_1 \epsilon_1 s'}.  \]
In fact, by Lemma \ref{lm:curve-in-N}, there exists $z' \in \Theta'(z)$ so that the following holds: For any $y \in \fJ\cap \Sigma^r$ such that 
$\cC(\cT(y), z) \subset \cN_1^\ell$ and is highly concentrated, we have
  \[ \Omega'(y) \subset a(2s')Q(\fr_1, \ell-1, \beta ) B^{A_0 H}(\beta) \tilde z,  \]
  where $\tilde z = a(-2s')z'$.
  Therefore, 
  \[ \SS(z,r; \cN_1^\ell) \subset a(2s')Q(\fr_1, \ell-1, \beta ) B^{A_0 H}(\beta) \tilde z \cap \Sigma^r. \]
  By our discussion at the beginning of \S \ref{sec-reduction}, Proposition \ref{prop:sl2xr2-closing-lemma} holds for $\mu_{\cF_i}$. Thus, we have 
  \begin{align*} 
  \tilde \mu \left(a(2s')Q(\fr_1, \ell-1, \beta ) B^{A_0 H}(\beta) \tilde z\right) &\le \beta^{-15} \mu_{\cF_i} \left(Q(\fr_1, \ell-1 , \beta) B^{A_0 H}(\beta) \tilde z\right) \\ 
  &\le \beta^{-15}e^{-d_0 (\ell-1) /2} < e^{-800 C_1 \epsilon_1 s'}.\end{align*}
  Note that 
  \[  a(2s')Q(\fr_1, \ell-1, \beta ) B^{A_0 H}(\beta) \tilde z \cap \Sigma^r \]
  is a cross-section of $a(2s')Q(\fr_1, \ell-1, \beta ) B^{A_0 H}(\beta) \tilde z$ along the $U$-direction. We have that 
  \begin{align*}
      \tilde \mu \left(\SS(z, r; \cN_1^\ell)\right) &\le \tilde \mu\left(a(2s')Q(\fr_1, \ell-1, \beta ) B^{A_0 H}(\beta) \tilde z \cap \Sigma^r\right) \\ 
                                   &\le e^{-2s'} \tilde \mu \left(a(2s')Q(\fr_1, \ell-1, \beta ) B^{A_0 H}(\beta) \tilde z\right) \le e^{-2s'-800 C_1 \epsilon_1 s'}.
  \end{align*}
  This proves the claim. Similarly to the case $ \ell \ge s' $, the claim implies that~$\clubsuit$ holds with $\HS(z,r)$ replaced by $\HS(z,r) \setminus \HS(z, r;\cN_1^\ell)$. By repeating the above argument with $\HS(z,r)$ replaced by $\HS(z,r) \setminus \HS(z, r;\cN_1^\ell)$, we get another 
$$\cN^\ell_2 = \cN^\ell(X_2, Y_2) = (X_2, Y_2) + [\beta e^{-s' -\ell}]\times [\beta e^{-s'}]$$
disjoint with $\cN_1^\ell$ which also contains at least 
\[ e^{d_i s' - 507 C_1 \epsilon_1 s'}\]
different full $\Omega'(z')$. By our claim, the argument can be repeated for at least $e^{600 C_1 \epsilon_1 s'}$ times 
which gives at least 
\[ e^{d_i s' - 507 C_1 \epsilon_1 s'} e^{600 C_1 \epsilon_1 s'} \ge e^{d_i s' + 20 C_1 \epsilon_1 s'} \]
different full $\Omega'(z')$. 
\par This completes the proof. 
}
\end{proof}

\section{Proof of the main theorem}
\label{sec-main-proof}
\par In order to apply Proposition \ref{prop:high-dimension-to-equidistribution-a}, we first prove the following statement:
\begin{prop}
  \label{prop:high-dim-decomposition}
  \par Let $s \le \hat s < t$, $s' =(t-\hat s)/2$, and $\hat \mu$ be a probability measure on $X_{\beta^{1/4}}$ defined by the normalized Lebesgue measure on the $U$-orbit 
  $a(\hat s) u([1])x_0$ removed an exponentially small proportion. 
  Suppose $\hat \mu$ is $(4-\theta, s')$-good. Then $a(2s')_\ast \hat \mu$ can be written as 
  \begin{equation}
    \label{eq:decompose-hat-mu}
    a(2s')_\ast \hat \mu = \nu_{\S} \mu_{\S} + \sum_{k} \nu_k a(2s')_\ast \mu_k,
  \end{equation}
  where $\mu_{\S}$ is a probability measure on $X$, $\nu_{\S} = O(\beta)$, and for each $k$, $\mu_k$ is a 
  probability measure on 
  $$\exp\left(B_{\vv r^+}(\beta)\right)u\left([\beta]\right)x_k$$
  for some $x_k \in X$. Moreover, for any $f \in C_c^{\infty}(X)$,
  \[ \int f \dd \mu_k = \int_{[\beta]^2} \int_{[1]} f(\exp(v_1 \vv v_1 + w_1\vv w_1 )u(\beta r)) \dd r \dd \rho_k (v_1, w_1),\]
  where $\rho_k$ is a probability measure on $[\beta]^2$ with dimension larger than $2-2\theta$ at scale $s'$ (see Proposition 
  \ref{prop:high-dimension-to-equidistribution-a} for the definition).
\end{prop}
\begin{proof}
  \par Let us define
  \begin{equation} 
    \label{eq:define-Upsilons'}
    \Upsilon_{s'} := B^{A_0 A U^\ast}(\beta)  \exp\left(B_{\vv r^-} (\beta e^{-s'})\right)\exp\left(B_{\vv r^+} (\beta)\right) B^{U}(\beta),
  \end{equation}
  and decompose $X_{\beta^{1/4}}$ as
  \[ X_{\beta^{1/4}} = \bigcup_{k=1}^{\beta^{-100} e^{2s'}}  \Upsilon_{s'} x_k.   \]
  We say $x_k$ is \emph{light} if 
  \[ \hat \mu\left( \Upsilon_{s'} x_k\right) \le \beta^{-2} e^{-2s' -\theta s'}. \]
  Otherwise we say $x_k$ is \emph{heavy}.
  For each heavy $x_k$, let us denote
  \[ \bar{\nu}_k = \hat \mu \left( \Upsilon_{s'} x_k \right), \]
  and 
  \[ \bar{\mu}_k = \bar{\nu}^{-1}_k \hat \mu|_{ \Upsilon_{s'} x_k}.\]
  Let us define 
  \[ \Upsilon_{\mathrm{light}} :=  \bigcup_{x_k \text{ light }}  \Upsilon_{s'} x_k.\]
  Then it is easy to see that 
  \[  \hat \mu \left( \Upsilon_{\mathrm{light}} \right) \le \beta^{10}.\]
  Therefore, we have 
  \begin{equation} 
    \label{eq:first-decompose-hat-mu}
    \hat \mu = \sum_{x_k \text{ heavy }} \bar{\nu}_k \bar{\mu}_k +O(\beta^{10}).  \end{equation}
  We claim that for every heavy $x_k$, $\bar \mu_k$ is \emph{nice}, namely, for $y \in \supp \bar{\mu}_k$, 
  \[ \bar{\mu}_k (Q_{s'}(\beta) B^{A_0 H}(\beta)y) \le \beta^2 e^{-2s' + 2\theta s'}. \]
  In fact, since $\hat \mu$ is $(4-\theta, s')$-good, we have that 
  \[ \hat \mu(Q_{s'}(\beta) B^{A_0 H}(\beta)y) \le e^{-s'(4-\theta)}. \]
  Therefore, 
  \begin{align*}
    \bar{\mu}_k (Q_{s'}(\beta) B^{A_0 H}(\beta)y) &= {\hat \mu}^{-1}(\Upsilon_{s'} x_k)  \hat \mu(Q_{s'}(\beta) B^{A_0 H}(\beta)y) \\ 
         &\le \beta^2 e^{2s' + \theta s'} e^{-s'(4-\theta)} = \beta^2 e^{-2s' + 2\theta s'}. 
  \end{align*}
  This proves the claim.
  \par Now for every heavy $x_k$, let us define $\mu_k$ as follows: For each 
  $U$-orbit $\mathcal{U}$ of length $\beta$ contained in $\supp \bar{\mu}_k$, pick $y \in \mathcal U$ such that 
  \[ y \in  B^{A_0 A U^\ast}(\beta)  \exp\left(B_{\vv r^-} (\beta e^{-s'})\right) \exp\left(B_{\vv r^+} (\beta)\right) x_k. \] 
  Then we can write
  \[ y = a_0 a u^{\ast} \exp(\vv w^-) \exp(\vv w^+) x_k, \]
  where $a_0 a u^{\ast} \in B^{A_0 A U^\ast}(\beta)$, $\vv w^- \in B_{\vv r^-} (\beta e^{-s'})$, and $\vv w^+ \in B_{\vv r^+} (\beta)$.
Let us define 
\[ \bar{y} :=   \exp(\vv w^+) x_k.\]
Since $\|\vv w^-\| \le \beta e^{-s'}$, we have that 
\begin{equation} 
  \label{eq:bary-y-Qs}
  \bar{y} \in Q_{s'}(\beta) B^{A_0 H}(\beta)y. \end{equation}
Moreover, we may apply \eqref{eq:fvy}, \eqref{eq:fwy}, and \eqref{eq:fy} to conclude that for 
every $r' \in [\beta]$, there exists $f(r')$ such that $|r' - f(r')| \le \beta|r'|$, and
\[ a(2s')u\left(f(r')\right) \bar{y} \in B^G(\beta) u(r') y.  \]
Therefore, for any function $f \in C^{\infty}_c(X)$, 
\begin{equation}
  \label{eq:compare-u-orbit} 
  \left| \int_{[1]} f(a(2s')u(\beta r) y) \dd r -  \int_{[1]} f(a(2s')u(\beta r) \bar y) \dd r\right| \ll \beta \|f\|_{C^1}.
\end{equation}
Let us denote the union of $\bar y$'s constructed as above by $\mathcal T$ and define $\mu_k$ as the normalized Lebesgue measure on the union 
$$ \bigcup_{\bar y \in \mathcal T}u([\beta])\bar y.$$
By \eqref{eq:compare-u-orbit}, we have that $a(2s')_\ast\mu_k$ and $a(2s')_{\ast}\bar \mu_k$ are $O(\beta)$-close, namely, for any $f \in C_c^{\infty} (X)$,
\[ \left| \int f \dd a(2s')_\ast\mu_k - \int f \dd a(2s')_\ast \bar{\mu}_k \right| \ll \beta \|f\|_{C^1}.\]
Since $\mu_k$ is supported on $\exp(B_{\vv r^+} (\beta)) u([\beta])x_k$, we have that there 
exists a probabality measure $\rho_k$ on $[\beta]^2$ such that for $f \in C^\infty_c(X)$,
\[ \int f \dd \mu_k = \int_{[\beta]^2} \int_{[1]} f(\exp(v_1 \vv v_1 + w_1 \vv w_1) u(\beta r)) \dd \rho_k(v_1, w_1) \dd r. \]
We claim that $\rho_k$ has dimension larger than $2-2\theta$ at scale $s'$. In fact, 
by \eqref{eq:bary-y-Qs}, we have that for $y$ and $\bar y$ constructed as above,
$$u([\beta])\bar y \in Q_{s'}(\beta) B^{A_0 H}(\beta)y.$$
Therefore, given $\bar \mu_k$ is nice, we have $\mu_k$ is also nice. Therefore, for any square $S$ of size $\beta e^{-s'}$, we have
\[\rho_k(S) = \mu_k \left( Q_{s'}(\beta) B^{A_0 H}(\beta)y_S \right) \le \beta^2 e^{-2s' + 2\theta s'}.\]
Here $y_S = \exp(v_S \vv v_1+ w_S \vv w_1)$ where $(v_S, w_S)$ is the center of $S$. This proves the claim.
\par Combining the claim above with \eqref{eq:first-decompose-hat-mu} and \eqref{eq:compare-u-orbit}, we conclude \eqref{eq:decompose-hat-mu}.


\end{proof}
\par We are now equipped to prove Theorem \ref{thm:main-thm}. 
\begin{proof}[Proof of Theorem \ref{thm:main-thm}]
\par By Proposition \ref{prop:initial-bound-dimension}, we have that $\mu_{\cF_0}$ is $(d_0, s_0)$-good. By Proposition \ref{prop:dimension-improvement-random}, \ref{prop:dimension-improvement-strange-part}, and \ref{prop:dimension-improvement-structured-component} we get Proposition \ref{prop:improve-dimension}. By Proposition \ref{prop:reduction}, it implies that $\mu_{\cF_{i}}$ is $(d_i, s_i)$-good, where $d_i = d_0 + i \epsilon_2$ and $s_i = 2^{-i}s_0$(we remove an exponentially small proportion at each step). 
 We can proceed with this inductive construction until $d_i \ge 4 - \theta$. By Proposition \ref{prop:high-dim-decomposition}, we have that 
 \[ a(2s_i)_\ast \mu_{\cF_i} = \sum_{k} \nu_k a(2s_i)_\ast \mu_k + O(\beta).\]
Note that by Proposition \ref{prop:high-dim-decomposition}, every $\mu_k$ is defined by 
a measure $\rho_k$ on $[\beta]^2$ with dimension larger than $2-2\theta$ at scale $s'$.
Then Proposition \ref{prop:high-dimension-to-equidistribution-a}, we have 
 that $a(2 s_i)_\ast \mu_{k}$ is effectively equidistributed. This implies that $a(2s_i)_\ast \mu_{\cF_i}$ is effectively equidistributed. 
 Noting that $\mu_{\cF_i}$ differs from the normalized Lebesgue measure on 
 $$a(t - 2 s_i) u\left([1]\right)x_0$$ 
 only by removing an exponentially small proportion, this completes the proof.
\end{proof}

\bibliography{reference}{}
\bibliographystyle{alpha}

\end{document}